\documentclass{article}

\usepackage{a4}
\usepackage{paralist}
\usepackage{graphicx, psfrag}
\usepackage{pgfplotstable}
\usepackage{subcaption}

\usepackage{amsmath}
\usepackage{amsthm}
\usepackage{amsfonts}
\usepackage{booktabs}
\usepackage{graphicx}
\usepackage{diagbox}

\usepackage{multirow}
\usepackage{hyperref}

\usepackage[ulem=normalem,draft]{changes}

\title{Robust finite element solvers for distributed hyperbolic optimal
  control problems}
\author{Ulrich~Langer\footnote{Institute of Numerical Mathematics,
    Johannes Kepler University Linz, 
    and
    Johann Radon Institute for Computational and Applied Mathematics (RICAM),
    Austrian Academy of Sciences,      
    Altenberger Stra{\ss}e 69, 4040 Linz,
    Austria, Email: ulanger@numa.uni-linz.ac.at},
  \;
  Richard~L\"oscher\footnote{Institut f\"{u}r Angewandte Mathematik,
    Technische Universit\"{a}t Graz, Steyrergasse 30, 8010 Graz, Austria,
    Email: loescher@math.tugraz.at}, 
  \; Olaf~Steinbach\footnote{Institut f\"{u}r Angewandte Mathematik,
    Technische Universit\"{a}t Graz, Steyrergasse 30, 8010 Graz, Austria,
    Email: o.steinbach@tugraz.at}, 
  \; Huidong~Yang\footnote{Faculty of Mathematics, University of Vienna,
  and 
  Doppler Laboratory for Mathematical Modeling and Simulation of Next
  Generations of Ultrasound Devices (MaMSi),
  Oskar--Morgenstern--Platz 1, A-1090 Wien, Austria,
  Email: huidong.yang@univie.ac.at}
}  

\date{}

\def\pub{1}		

\newcommand{\norm}[1]{\|#1\|}

\pgfplotsset{compat=1.16}

\newtheorem{theorem}{Theorem}
\newtheorem{lemma}{Lemma}
\newtheorem{corollary}{Corollary}
\newtheorem{proposition}{Proposition}
\newtheorem{remark}{Remark}

\begin{document}

\maketitle

\ifnum\pub=0 
\fi
\ifnum\pub=1
\begin{center}
\vspace*{-4ex}
 \date{\it Dedicated to Gundolf Haase on the occasion of his 60-th birthday}\\[4ex]
\end{center}
\fi

\begin{abstract}
  We propose, analyze, and test new robust iterative solvers for systems
  of linear algebraic equations arising from the space-time finite element
  discretization of reduced optimality systems defining the approximate
  solution of hyperbolic distributed, tracking-type optimal control problems 
  with both the standard $L^2$ and the more general energy regularizations.
  In contrast to the usual time-stepping approach, we discretize the
  optimality system by space-time continuous piecewise-linear finite element
  basis functions which are defined on fully unstructured simplicial meshes. 
  If we aim at the asymptotically best approximation of the given desired
  state $y_d$ by the computed finite element state $y_{\varrho h}$,
  then the optimal choice of the regularization parameter $\varrho$ is 
  linked to the space-time finite element mesh-size $h$ by the relations
  $\varrho=h^4$ and $\varrho=h^2$ for the $L^2$ and the energy regularization,
  respectively. For this setting, we can construct robust (parallel)
  iterative solvers for the reduced finite element optimality systems.
  These results can be generalized to variable regularization parameters 
  adapted to the local behavior of the mesh-size that can heavily change
  in the case of adaptive mesh refinements. The numerical results illustrate
  the theoretical findings firmly.
  \end{abstract} 
\noindent
\noindent
\begin{msc}
49J20,  
49M05,  
35L05,  
65M60,  
65M15,  
65N22   
\end{msc}

\section{Introduction}\label{Sec:Introduction}
Let us first consider abstract optimal control problems (OCPs) of the form:
Find the state $y_\varrho \in Y$ and the control $u_\varrho \in U$  
minimizing the cost functional
\begin{equation}
\label{Sec:Introduction:Eqn:AbstractCostFunctional}
\mathcal{J}(y_\varrho,u_\varrho) := \mathcal{J}_\varrho(y_\varrho,u_\varrho):= 
\frac{1}{2} \, \|y_\varrho - y_d\|_H^2 + \frac{\varrho}{2} \, \|u_\varrho \|_U^2,
\end{equation}
subject to the state equation
\begin{equation}
\label{Sec:Introduction:Eqn:AbstractStateEquation}
B y_\varrho = u_\varrho \quad \mbox{in} \;\; U \subset P^*,
\end{equation}
where the desired state (target) $y_d \in H$ and the regularization
parameter $\varrho > 0$ are given. We mention that, in optimal control,
$\varrho$ also allows to influence the costs of the control in terms of
$\|  u \|_U^2$. The state space $Y$, the adjoint state space $P$, the
observation space $H$, and the control space $U$ are Hilbert spaces
equipped with the corresponding norms and scalar products. We assume
that $Y \subset H \subset Y^*$ and $P \subset H \subset P^*$ are Gelfand
triples, where $Y^*$ and $P^*$ denote the dual spaces of $Y$ and $P$,
respectively. The duality products 
$\langle \cdot, \cdot \rangle : Y^* \times Y \to \mathbb{R}$ 
and 
$\langle \cdot, \cdot \rangle : P^* \times P \to \mathbb{R}$ 
are assumed to be extensions of the scalar product $\langle\cdot, \cdot\rangle_H$ in $H$.
The state operator $B$ is usually an isomorphism as mapping from $Y$ to $P^*$.
So, we are interested to control the state equation
\eqref{Sec:Introduction:Eqn:AbstractStateEquation} not only in $U=H$, where
$H=L^2$ in the standard setting ($L^2$-regularization), 
but also in $U=P^*$ that is sometimes called energy regularization.
Such kind of optimal control problems were already studied in the classical 
monograph \cite{LSTY:Lions:1968Book} by Lions who also 
admitted additional control constraints.
Since then many books and  papers on the analysis and numerics 
of such kind of optimal control problems often with additional
inequality constraints imposed on the control $u$ or{/}and the state $y$
have been published. We here refer the reader only to the 
books 
\cite{LLSY:DeLosReyes:2015Book,LSTY:HinzePinnauUlbrichUlbrich:2009Book,
  LLSY:Troeltzsch:2010Book}, and the recent omnibus volume 
\cite{LLSY:HerzogHeinkenschlossKaliseStadlerTrelat:2022Proceedings}
on  optimization and control for partial differential equations (PDEs).

We will here only consider tracking-type, distributed hyperbolic OCPs
that are represented by the model state operator
$B = \Box = \partial_{tt} - \Delta_x$ (wave operator).
The reduced optimality system that characterizes the unique 
solution of the optimal control problem under consideration 
is discretized by an unstructured simplicial finite element (FE) method 
that is a real space-time finite element method; 
see \cite{LLSY:LangerSteinbachTroeltzschYang:2021c,
  LLSY:LangerSteinbachTroeltzschYang:2021b}
and \cite{LLSY:LoescherSteinbach:2024SINUM} 
for the parabolic and hyperbolic cases, respectively.
This all-at-once space-time 
discretization of the reduced optimality system leads to 
a symmetric, but indefinite (SID) system of FE
equations of the form
\begin{equation}
\label{Sec:Introduction:Eqn:AbstractDiscreteReducedOptimalitySystem}
  \begin{bmatrix}
    A_{\varrho h} & B_h\\
    B_h^\top & -M_h
  \end{bmatrix}
  \begin{bmatrix}
   \mathbf{p}_h\\ 
   \mathbf{y}_h\\ 
  \end{bmatrix}
  =
  \begin{bmatrix}
    \mathbf{0}_h\\ 
    -\mathbf{y}_{dh}\\ 
  \end{bmatrix}.
\end{equation}
as in the elliptic case, where the matrix $B_h$ is the FE representation
of the state operator $B$, $M_h$ is nothing but the mass matrix, 
$A_{\varrho h}$ represents the regularization term, and the subscript $h$ is
a suitable discretization parameter. In the standard case of $L^2$
regularization with a constant regularization (cost) parameter $\varrho$,
the matrix $A_{\varrho h}$ equals $\varrho^{-1} \overline{M}_h$, where
$\overline{M}_h$ is the mass matrix from the finite element space for the
approximation of the adjoint state $p$. The matrices $M_h$ and $A_{\varrho h}$
are symmetric and positive definite (SPD). In contrast to this approach to 
time-dependent optimal control problems, the standard time-stepping
discretization combined with a FE space discretization produces smaller
systems of the form
\eqref{Sec:Introduction:Eqn:AbstractDiscreteReducedOptimalitySystem} 
at each time step; see, e.g., 
\cite{LLSY:KunischReiterer2015ANM, PeraltaKunisch:2022}, and 
the references therein. 

There is a huge amount of publications on preconditioners and iterative
solvers for general systems of algebraic equations with symmetric and
indefinite system matrices such as
\eqref{Sec:Introduction:Eqn:AbstractDiscreteReducedOptimalitySystem}.
We refer the reader to the survey papers
\cite{LLSY:BenziGolubLiesen:2005ActaNumerica,LLSY:Wathen:2015ActaNumerica}, 
the books
\cite{LLSY:BaiPan:2021Book,LLSY:ElmanSilvesterWathen:2005Book,
  LLSY:Rozloznik:2018Book}, the review paper \cite{LLSY:MardalWinther:2011NLA},
and the references therein for a comprehensive overview on saddle point
solvers in general. In particular, there are many papers devoted to the
efficient solution of SID systems arising from PDE-constrained OCPs with
the standard $L^2$ regularization and fixed regularization parameter $\varrho$. 
More recently, preconditioners leading to $\varrho$-robust iterative solvers 
have been developed for PDE-constrained OCPs subject to different state
equations without and with control and{/}or state constraints; see, e.g., 
\cite{LLSY:AxelssonKaratson:2020NumerMath,
LLSY:AxelssonNeytchevaStroem:2018JNM,
LLSY:BaiBenziChenWang:2013IMANA,
LLSY:DravinsNeytcheva:2021,
LLSY:PearsonStollWathen:2014NLA,
LLSY:PearsonWathen:2012NLA,
LLSY:SchielaUlbrich:2014SIOPT,
LLSY:SchoeberlZulehner:2007SIMAX,
LLSY:SchulzWittum:2008CVS,
LLSY:StollWathen:2012NLA}
and the references provided in these papers.

In this paper, we first investigate the deviation of the exact state
$y_\varrho$ from the desired state $y_d$ with respect to (wrt) the
$H = L^2(Q)$ norm in dependence on the regularization parameter $\varrho$
and the regularity of the desired state $y_d$. It turns out that the
quantitative behavior is practically the same as was first proved for
elliptic optimal control problems with both $L^2$ and energy regularization
in \cite{LLSY:NeumuellerSteinbach:2021a}. After the simplicial space-time
finite element (FE) discretization, we choose $\varrho$ in such a way that
the finite element state $y_{\varrho h}$, corresponding to the nodal vector
$\mathbf{y}_h$ as part of the solution of
\eqref{Sec:Introduction:Eqn:AbstractDiscreteReducedOptimalitySystem}, provides an asymptotically
optimal approximation to the desired state $y_d$ in the $L^2$ norm. It was
already shown in \cite{LLSY:LoescherSteinbach:2024SINUM} that $\varrho = h^2$
is always the optimal choice in the case of the energy regularization 
independent of the regularity of the desired state $y_d$. In this paper, we
also investigate the standard $L^2$ regularization for which we get
$\varrho = h^4$ as optimal choice. These choices of $\varrho$ provide not
only an optimal balance between the regularization and discretization errors,
but also a well-conditioned primal Schur Complement (SC)
$S_{\varrho h} = B_h^\top A_{\varrho h}^{-1} B_h + M_h$ of the system matrix
in the SID system \eqref{Sec:Introduction:Eqn:AbstractDiscreteReducedOptimalitySystem}. 
More precisely, we show that the Schur complement  $S_{\varrho h}$ is
spectrally equivalent to the mass matrix $M_h$ and, therefore, to the
diagonal lumped mass matrix $D_h = \text{lump}(M_h)$ with computable spectral
equivalence constants. This result is crucial for the construction of fast
iterative solvers for the reduced algebraic optimality system
\eqref{Sec:Introduction:Eqn:AbstractDiscreteReducedOptimalitySystem}. 
It turns out that the Schur-Complement Preconditioned Conjugate Gradient
(SC-PCG) method for solving the SPD SC problem
\begin{equation}
\label{Sec:Introduction:Eqn:SchurComplementSystem}
S_{\varrho h} \mathbf{y}_h = \mathbf{y}_{dh},
\end{equation}
which arises from \eqref{Sec:Introduction:Eqn:AbstractDiscreteReducedOptimalitySystem} by
eliminating the adjoint FE state $\mathbf{p}_h$ from
\eqref{Sec:Introduction:Eqn:AbstractDiscreteReducedOptimalitySystem}, 
is an efficient alternative to the solution of the SID system
\eqref{Sec:Introduction:Eqn:AbstractDiscreteReducedOptimalitySystem} by means of the closely
related Bramble-Pasciak PCG (BP-PCG) \cite{LLSY:BramblePasciak:1988a},
especially, in the case of the $L^2$ regularization when 
$A_{\varrho h}^{-1}$ can be replaced by
$(\text{lump}(\overline{M}_{\varrho h}))^{-1}$ ensuring a fast matrix-by-vector
multiplication. We note that these results remain valid for the corresponding
variable choice of the regularization parameter $\varrho$ adapted to the
local behavior of the size of the simplicial space-time mesh that can heavily
vary in the case of adaptive FE discretisations as used in some of our
numerical experiments.

The remainder of the paper is organized as follows: 
In Section~\ref{Sec:PreliminariesAndSpecifications}, we introduce some
preliminary material, and specify the hyperbolic OCPs that we are going to
investigate. More precisely, we consider the standard $L^2$ regularization
and the more general energy regularization. The space-time finite element
discretization of these hyperbolic OCPs on unstructured simplicial meshes
is presented and analyzed in
Section~\ref{Sec:SpaceTimeFiniteElementDiscretization}.
Section~\ref{Sec:Solvers} is devoted to efficient iterative methods for 
solving the algebraic systems arising from the space-time finite element 
discretization of the reduced optimality systems.
In Section~\ref{Sec:NumericalResults}, we present and discuss our 
numerical results. Finally, we draw some conclusions and give an outlook 
in Section~\ref{Sec:ConclusionAndOutlook}.

\section{Preliminaries and specifications}
\label{Sec:PreliminariesAndSpecifications}
As a model problem, we consider a distributed optimal control problem
subject to the wave equation with homogeneous Dirichlet boundary and
initial conditions. Therefore, let $\Omega \subset {\mathbb{R}}^d$,
$d=1,2,3$, be a bounded spatial domain with, for $d=2,3$, 
Lipschitz boundary $\Gamma = \partial \Omega$, and let $0<T<\infty$ be
a given finite time horizon. Further we introduce the space-time cylinder
$Q := \Omega \times (0,T)$, its lateral boundary
$\Sigma := \Gamma \times (0,T)$, its bottom
$\Sigma_0 := \Omega \times \{0\}$, and its top
$\Sigma_T := \Omega \times \{T\}$. For a given target $y_d \in L^2(Q)$ and
a regularization parameter $\varrho > 0$, we consider the minimization of
the cost functional
\eqref{Sec:Introduction:Eqn:AbstractCostFunctional} with $H=L^2(Q)$,
subject to the homogeneous initial-boundary value problem 
for the wave equation 
\begin{equation}\label{eq:wave-PDE-constraint}
  \Box y_\varrho := \partial_{tt} y_\varrho - \Delta_x y_\varrho =
  u_\varrho \;\mbox{in}\,Q,\; y_\varrho = 0 \;\mbox{on}\,\Sigma,\;
  y_\varrho = \partial_t y_\varrho = 0 \;\mbox{on}\,\Sigma_0.
\end{equation}
In order to derive a variational formulation of the wave equation
\eqref{eq:wave-PDE-constraint}, we introduce
\begin{align*}
  H_{0;0,}^{1,1}(Q) &:= \{y \in L^2(Q): \nabla_x y \in [L^2(Q)]^d, \partial_t y\in L^2(Q),
                          y=0\,\text{on}\, \Sigma, y=0\,\text{on}\, \Sigma_0\},\\
   H_{0;,0}^{1,1}(Q) &:= \{q \in L^2(Q): \nabla_x q \in [L^2(Q)]^d, \partial_t q\in L^2(Q),
                          q=0\,\text{on}\, \Sigma, q=0\,\text{on}\, \Sigma_T\},                       
\end{align*}
both equipped with the norm 
$|v|_{H^1(Q)} = (\norm{\partial_t v}_{L^2(Q)}^2 + \norm{\nabla_x v}_{L^ 2(Q)}^2)^{1/2}$.
Note that $y \in H^{1,1}_{0;0,}(Q)$ covers zero initial conditions $y(x,0)=0$,
while, for $q \in H^{1,1}_{0;,0}(Q)$, we have $q(x,T)=0$, $x \in \Omega$. We now
consider the variational formulation of \eqref{eq:wave-PDE-constraint}
to find $y_\varrho \in H^{1,1}_{0;0,}(Q)$ such that
\begin{equation}\label{eq:wave-VF-L2}
  b(y_\varrho,q) :=
  - \langle \partial_t y_\varrho , \partial_t q \rangle_{L^2(Q)} +
  \langle \nabla_x y_\varrho, \nabla_x q \rangle_{L^2(Q)} =
  \langle u_\varrho , q \rangle_{L^2(Q)}
\end{equation}
is satisfied for all $q \in H^{1,1}_{0;,0}(Q)$. Unique solvability of
\eqref{eq:wave-VF-L2} follows when assuming $u_\varrho \in L^2(Q)$, see,
e.g., \cite{LSTY:Ladyzhenskaya:1985a,LLSY:SteinbachZank:2020}. This
motivates to consider the optimal control problem with $L^2$
regularization first.

\subsection{The $L^2$ regularization $U = L^2(Q)$}\label{sec:wave-l2-reg}
Let us first consider the more common $L^2$ regularization with $U=L^2(Q)$.
Then, for any $u_\varrho \in L^2(Q)$, the variational formulation
\eqref{eq:wave-VF-L2} admits a unique solution $y_\varrho \in H_{0;0,}^{1,1}(Q)$
satisfying the stability estimate
\[
  \| y_\varrho \|_{H_{0;0,}^{1,1}(Q)} \leq
  \frac{T}{\sqrt{2}} \, \| u_\varrho \|_{L^2(Q)},
\]
see, e.g., \cite[Theorem 5.1, p.~169]{LSTY:Ladyzhenskaya:1985a}, or
\cite[Theorem 5.1]{LLSY:SteinbachZank:2020}. Thus, we can define the
solution operator $y_\varrho = \mathcal{S}u_\varrho$ with
$\mathcal{S} : L^2(Q)\to H_{0;0,}^{1,1}(Q)$, and we can consider the
reduced cost functional 
\[
  \widehat{\mathcal{J}}(u_\varrho) =
  \frac{1}{2} \, \| \mathcal{S}u_\varrho-y_d \|_{L^2(Q)}^2 +
  \frac{1}{2} \, \varrho \, \| u_\varrho \|_{L^2(Q)}^2,
\]
for which the minimizer satisfies the gradient equation 
\begin{equation}\label{eq:wave-gradient-equation}
  \mathcal{S}^\ast(\mathcal{S}u_\varrho-y_d) +
  \varrho \, u_\varrho = 0 \quad \text{in} \; L^2(Q). 
\end{equation}     
When introducing the adjoint state
$p_\varrho := \mathcal{S}^\ast(\mathcal{S}u_\varrho-y_d) \in P = H_{0;,0}^{1,1}(Q)$
as the unique weak solution of the adjoint problem
\begin{equation} \label{eq:wave-adjoint-equation}
  \partial_{tt} p_\varrho - \Delta_x p_\varrho =
  y_\varrho - y_d \; \mbox{in}\,Q,\; p_\varrho = 0 \;\mbox{on}\,\Sigma,\;
  p_\varrho = \partial_t p_\varrho = 0 \;\mbox{on}\,\Sigma_T,
\end{equation}
we end up with the optimality system, including the forward problem
\eqref{eq:wave-PDE-constraint}, the adjoint problem
\eqref{eq:wave-adjoint-equation}, and the gradient equation
\eqref{eq:wave-gradient-equation}. 

\begin{remark}
  We note that, by the gradient equation \eqref{eq:wave-gradient-equation},
  we have
  \begin{equation}\label{gradient equation u p}
    u_\varrho = -\varrho^{-1} \, p_\varrho \in H_{0;,0}^{1,1}(Q).
  \end{equation}
  Thus, actually the control $u_\varrho \in H^{1,1}_{0;,0}(Q)$ is more regular,
  but also
  inherits, probably unpleasant, boundary and terminal conditions from
  the adjoint state. 
\end{remark}

\noindent
When eliminating the control $u_\varrho = \Box y_\varrho$, from
\eqref{eq:wave-PDE-constraint}, we get by the gradient equation that
$p_\varrho + \varrho \, \Box y_\varrho =0$, and the reduced optimality system in
variational form is to find
$(p_\varrho,y_\varrho)\in H_{0;,0}^{1,1}(Q)\times H_{0;0,}^{1,1}(Q)$
such that 
\begin{equation}\label{eq:wave-continuous-VF-optimality-system-common-L2}
  \begin{array}{rcrcll}
    \varrho^{-1} \, \langle p_\varrho, q \rangle_{L^2(Q)}
    & + & b(y_\varrho,q)
    & = & 0 & \forall \, q \in H_{0;,0}^{1,1}(Q), \\[2mm]
    -b(z,p_\varrho)
    & + & \langle y_\varrho,z \rangle_{L^2(Q)}
    & = & \langle y_d,z \rangle_{L^2(Q)} & \forall \, z\in H_{0;0,}^{1,1}(Q). 
  \end{array}
\end{equation}
Unique solvability of
\eqref{eq:wave-continuous-VF-optimality-system-common-L2} follows
from the way we derived the system. 

\begin{remark}
  In addition, we can eliminate the adjoint variable
  $p_\varrho = -\varrho \, u_\varrho = - \varrho \, \Box y_\varrho$ in the
  adjoint equation \eqref{eq:wave-adjoint-equation} to conclude 
  \begin{align*}
    \varrho \, \Box ^2 y_\varrho  = - \Box p_\varrho = y_d - y_\varrho,  
  \end{align*}
  and, therefore, we get
  \begin{equation}\label{eq:wave-operator-equation-common-L2}
    \begin{array}{rcrclcl}
      \multicolumn{3}{r}{\varrho \, \Box^2 y_\varrho +y_\varrho}
      & = & y_d  & &  \mbox{in}\; Q,\\[1mm]
      && y_\varrho = \Box \, y_\varrho & = & 0 & &  \mbox{on}\;  \Sigma, \\[1mm]
      && y_\varrho=\partial_t y_\varrho
          & = & 0 & &   \mbox{on}\;  \Sigma_0, \\[1mm]
      && \Box \, y_\varrho=\partial_t \Box \, y_\varrho &= & 0
                   & &   \mbox{on}\;  \Sigma_T,
    \end{array}
  \end{equation}
  which is nothing but a kind of bi-wave equation with boundary and terminal
  conditions inherited from the adjoint state $p_\varrho$. 	
\end{remark}

\noindent
As a last step, we present some estimates for the distance
$\| y_\varrho-y_d \|_{L^2(Q)}$ of the regularized state $y_\varrho$ from
the target $y_d$, which only depends on the regularization parameter $\varrho$,
and on the regularity of the target.

\begin{lemma}\label{lem:regularization-error-estimates-L2}
  Let $y_d\in L^2(Q)$. For the unique solution
  $(p_\varrho,y_\varrho) \in H_{0;,0}^{1,1}(Q)\times H_{0;0,}^{1,1}(Q)$ of
  \eqref{eq:wave-continuous-VF-optimality-system-common-L2}
  there holds 
  \begin{equation}\label{eq:reg-error-estimate-L2:L2-L2}
    \| y_\varrho-y_d \|_{L^2(Q)} \leq \| y_d \|_{L^2(Q)}.
  \end{equation}
  If in addition $y_d\in H_{0;0,}^{1,1}(Q)$ such that $\Box y_d \in L^2(Q)$,
  then 
  \begin{equation}\label{eq:reg-error-estimate-L2:L2-H2}
    \| y_\varrho-y_d \|_{L^2(Q)} \leq
    \sqrt{\varrho} \, \| \Box y_d \|_{L^2(Q)}.
  \end{equation}
  Moreover, we also have
  \begin{equation}\label{eq:reg-error-estimate-L2:H2-H2}
    \| \Box y_\varrho \|_{L^2(Q)}\leq \| \Box y_d \|_{L^2(Q)}.
  \end{equation}
\end{lemma}

\begin{proof}
  Firstly, let $y_d\in L^2(Q)$. Testing
  \eqref{eq:wave-continuous-VF-optimality-system-common-L2} with
  $q = p_\varrho$ and $z=y_\varrho$, we obtain
  \begin{equation*}
    \langle y_\varrho-y_d,y_\varrho \rangle_{L^2(Q)}  =
    b(y_\varrho,p_\varrho) = - \varrho^{-1} \, \| p_\varrho \|_{L^2(Q)}^2, 
  \end{equation*}
  from which we further deduce, using a Cauchy--Schwarz inequality, that
  \begin{equation*}
    \| y_\varrho-y_d \|_{L^2(Q)}^2 + \varrho^{-1} \, \| p_\varrho \|_{L^2(Q)}^2
    = \langle y_d-y_\varrho,y_d \rangle_{L^2(Q)} \leq
    \| y_\varrho-y_d \|_{L^2(Q)}\| y_d \|_{L^2(Q)},
  \end{equation*}
  which gives \eqref{eq:reg-error-estimate-L2:L2-L2}. 
  If now $y_d\in H_{0;0,}^{1,1}(Q)$ such that $\Box y_d\in L^2(Q)$, we can
  test \eqref{eq:wave-continuous-VF-optimality-system-common-L2} with
  $z=y_d-y_\varrho$, and using the relations \eqref{eq:wave-PDE-constraint}
  and \eqref{eq:wave-gradient-equation}, i.e.,
  $p_\varrho = -\varrho \, \Box y_\varrho$, we get 
  \begin{eqnarray*}
    \| y_d-y_\varrho \|_{L^2(Q)}^2
    & = & \langle y_d-y_\varrho,y_d-y_\varrho\rangle_{L^2(Q)} =
          b(y_\varrho-y_d,p_\varrho) 
          = b(y_\varrho,p_\varrho) - b(y_d,p_\varrho) \\
    & = & -\varrho \, \langle \Box y_\varrho, \Box y_\varrho \rangle_{L^2(Q)} +
          \varrho \, \langle \Box y_d,\Box y_\varrho \rangle_{L^2(Q)}.  
  \end{eqnarray*}
  Reordering and applying a Cauchy--Schwarz inequality, this gives 
  \begin{equation*}
    \| y_d-y_\varrho \|_{L^2(Q)}^2 +
    \varrho \, \| \Box y_\varrho \|_{L^2(Q)}^2 =
    \varrho \, \langle \Box y_d,\Box y_\varrho \rangle_{L^2(Q)} \leq
    \varrho \, \| \Box y_d \|_{L^2(Q)} \| \Box y_\varrho \|_{L^2(Q)},
  \end{equation*}
  from which \eqref{eq:reg-error-estimate-L2:H2-H2} and
  \eqref{eq:reg-error-estimate-L2:L2-H2} follow. 
\end{proof}

\begin{corollary}
  From the gradient equation \eqref{gradient equation u p}, the primal
  wave equation in \eqref{eq:wave-PDE-constraint}, and
  \eqref{eq:reg-error-estimate-L2:H2-H2}, we conclude
  \begin{equation}\label{Estimate p L2 H2}
    \| p_\varrho \|_{L^2(Q)} = \varrho \, \| u_\varrho \|_{L^2(Q)} =
    \varrho \, \| \Box y_\varrho \|_{L^2(Q)} \leq
    \varrho \, \| \Box y_d \|_{L^2(Q)},
  \end{equation}
  while from the wave equation in \eqref{eq:wave-adjoint-equation} and
  using \eqref{eq:reg-error-estimate-L2:L2-H2}, this gives
  \begin{equation}\label{Estimate p H2 H2}
    \| \Box p_\varrho \|_{L^2(Q)} =
    \| y_\varrho - y_d \|_{L^2(Q)} \leq \sqrt{\varrho} \,
    \| \Box y_d \|_{L^2(Q)}.
  \end{equation}
\end{corollary}

\begin{proposition}\label{prop:regularization-interpolation-estimates-L2}
  The error estimates  \eqref{eq:reg-error-estimate-L2:L2-H2} and
  \eqref{eq:reg-error-estimate-L2:H2-H2} as well as \eqref{Estimate p L2 H2}
  and \eqref{Estimate p H2 H2} may motivate the use of a space interpolation
  argument in order to derive related error estimates in $H^1(Q)$.
  Unfortunately, this does not hold true in general. At this point,
  we therefore assume that the given data are such that the following
  regularization error estimates hold true, i.e.,
    \begin{equation}\label{Estimate y H1 H2}
    |y_\varrho - y_d |_{H^1(Q)} \leq c \, \varrho^{1/4} \,
    \| \Box y_d \|_{L^2(Q)},
  \end{equation}
  and
  \begin{equation}\label{Estimate p H1 H2}
    | p_\varrho |_{H^1(Q)} \leq c \, \varrho^{3/4} \, \| \Box y_d \|_{L^2(Q)} .
  \end{equation}
\ifnum\pub=0   
All our numerical experiments performed for smooth targets confirm the
estimates \eqref{Estimate y H1 H2} and \eqref{Estimate p H1 H2}.
\fi  
\ifnum\pub=1   
All our numerical experiments performed for smooth targets confirm the
estimates \eqref{Estimate y H1 H2} and \eqref{Estimate p H1 H2}; see
Appendix.
\fi 

\end{proposition}

\begin{remark}\label{rem:regularization-interpolation-estimates-L2}
  The regularization error estimates \eqref{Estimate y H1 H2} and
  \eqref{Estimate p H1 H2} are a simple consequence of the space
  interpolation type estimate
  \begin{equation}\label{space interpolation}
    \| v \|^2_{H^1(Q)} \leq c \, \| \Box v \|_{L^2(Q)} \| v \|_{L^2(Q)}
  \end{equation}
  for all $v \in H^1(Q)$ with $\Box v \in L^2(Q)$ and
  $v=\partial_tv=0$ on $\Sigma_0$. In order to prove
  \eqref{space interpolation} we can use the normalized eigenfunctions
  $\phi_k \in H^1_0(\Omega)$ with eigenvalues $\mu_k$ of the spatial
  Dirichlet eigenvalue problem for the Laplacian to write
  \[
    v(x,t) = \sum\limits_{k=1}^\infty V_k(t) \phi_k(x) \, , \quad
    V_k(0)=V_k'(0)=0,
  \]
  and it turns out that \eqref{space interpolation} is a consequence
  of the estimate
  \begin{equation}\label{lokal space interpolation}
    \| V_k' \|^2_{L^2(0,T)} + \mu_k \, \| V_k \|^2_{L^2(0,T)} \leq c \,
    \| V_k'' + \mu_k \, V_k \|_{L^2(0,T)} \| V_k \|_{L^2(0,T)},
    \; k \in {\mathbb{N}} .
  \end{equation}
  In particular, for, see {\rm  \cite[Theorem 4.2.6]{LLSY:Zank:2020}},
  \[
    V_k(t) := \frac{1}{\sqrt{T^3}} \int_0^t s \, \sin (\sqrt{\mu_k}s) \, ds
    \quad \mbox{for} \; t \in [0,T]
  \]
  we conclude $c^{-1}=\mathcal{O}(\sqrt{1/\mu_k})$, and thus
  $c \to \infty$ as $k \to \infty$. 
  Hence, this analysis indicates that the regularization error estimates
  \eqref{Estimate y H1 H2} and \eqref{Estimate p H1 H2} are only violated
  when high-oscillating contributions appear.
\end{remark}
\subsection{The energy regularization in
  $U=P^\ast=[H_{0;,0}^{1,1}(Q)]^\ast$}
\label{Subsec:wave-h1-reg}
We note that so far we needed $u_\varrho \in L^2(Q)$ to admit a unique
solution of the variational formulation \eqref{eq:wave-VF-L2}. As we test
\eqref{eq:wave-VF-L2} with functions $q \in P = H_{0;,0}^{1,1}(Q)$, a natural
question to appear is, whether we can also choose the control in the dual
space $u_\varrho\in P^\ast = [H_{0;,0}^{1,1}(Q)]^\ast$. But, as it turns out,
the operator $B:H_{0;0,}^{1,1}(Q)\to P^\ast$ as implied by the bilinear form
$b(y,q) = \langle B y , q \rangle_Q$ for all $y \in H^{1,1}_{0;0,}(Q)$ and
for all $q \in H^{1,1}_{0;,0}(Q)$ does not define an isomorphism, see
\cite[Theorem 1.1]{LLSY:SteinbachZank:2022}. Recapitulating the work of
\cite{LLSY:SteinbachZank:2022}, see
also \cite{LLSY:LoescherSteinbach:2024SINUM},
we will define suitable spaces, such that the wave operator is an isomorphism.
The first issue to overcome is the establishment of an inf-sup condition,
guaranteeing the injectivity of the operator. It fails to hold in the
above setting, since the
initial condition $\partial_t y_\varrho(x,t)\big|_{t=0}$ enters the variational
formulation naturally, which is not appropriate in this case. In order to
incorporate it in a meaningful sense, we will modify the state space.
Let $Q_- := \Omega \times (-T,T)$ denote the enlarged space-time domain,
and define the zero extension of a function $y\in L^2(Q)$ by
\[
\widetilde{y}(x,t) :=
\begin{cases}
	y(x,t) & \mbox{for} \; (x,t)\in Q, \\
	0, & \mbox{else.}
\end{cases}
\]
Then, we consider the application of the wave operator in a distributional
sense, i.e., for all $\varphi\in C_0^\infty(Q_-)$ we define
\[
\langle \Box \widetilde{y}, \varphi \rangle_{Q_-}
:=
\int_{Q_-} \widetilde{y}(x,t) \, \Box \varphi(x,t) \, dx \, dt =
\int_Q y(x,t) \, \Box \varphi(x,t) \, dx \, dt.  
\]
Using this definition, we can introduce the space 
\[
\mathcal{H}(Q) := \Big \{
y=\widetilde{y}_{|_{Q}} : \widetilde{y} \in L^2(Q_-) , \;
\widetilde{y}_{|\Omega \times (-T,0)}=0, \;
\Box \widetilde{y} \in [H_0^1(Q_-)]^* \Big \},
\]
with the graph norm
$\|y\|_{\mathcal{H}(Q)} :=
(\|y\|_{L^2(Q)}^2 + \|\Box \widetilde{y} \|_{[H_0^1(Q_-)]^*}^2)^{1/2}$.
The normed vector space $(\mathcal{H}(Q),\|\cdot\|_{\mathcal{H}(Q)})$ is a
Banach space, with $H_{0;0,}^{1,1}(Q)\subset \mathcal{H}(Q)$;
see \cite[Lemma 3.5]{LLSY:SteinbachZank:2022}.
Therefore, we can introduce the space
\[
Y:=\mathcal{H}_{0;0,}(Q) :=
\overline{H_{0;0,}^{1,1}(Q)}^{\| \cdot \|_{\mathcal{H}(Q)}} \subset
\mathcal{H}(Q),
\]
which will serve as state space. For $y \in Y$, an equivalent norm is
given by
\[
\| y \|_{Y} = \| \Box \widetilde{y}\|_{[H^1_0(Q_-)]^*},
\]
see \cite[Lemma 3.6]{LLSY:SteinbachZank:2022}. It turns out that 
$ B : {\mathcal{H}}_{0;0,}(Q) \to [H^{1,1}_{0;,0}(Q)]^\ast $ defined by
$\langle B y , q \rangle_Q = \langle \Box \widetilde{y} ,
{\mathcal{E}}q \rangle_{Q_-}$ for all
$(y,q) \in {\mathcal{H}}_{0;0,}(Q) \times H^{1,1}_{0;,0}(Q)$, and using
some bounded extension
${\mathcal{E}} :H^{1,1}_{0;,0}(Q) \to H^1_0(Q_-)$, e.g., reflection in time,
is an isomorphism; 
see \cite[Lemma 3.5, Theorem 3.9]{LLSY:SteinbachZank:2022}. Moreover, 
we have
\begin{equation}\label{eq:evaluation-wave-bilinear-form}
  \langle \Box \widetilde{y} , {\mathcal{E}}q \rangle_{Q_-} =
  - \langle \partial_t y , \partial_t q \rangle_{L^2(Q)} 
  + \langle \nabla_x y , \nabla_x q \rangle_{L^2(Q)} 
\end{equation}
for $y \in H^{1,1}_{0;0,}(Q)$ and $q \in H^{1,1}_{0;,0}(Q)$, 
which in particular applies when considering conforming space-time
finite element spaces $Y_h \subset H^{1,1}_{0;0,}(Q) \subset Y$,
and $P_h \subset H^{1,1}_{0;,0}(Q)$.

For any given $u_\varrho\in P^\ast = [H^{1,1}_{0;,0}(Q)]^\ast$, we now find
a unique $y_\varrho \in Y = {\mathcal{H}}_{0;0,}(Q)$ satisfying
\begin{equation}\label{eq:generalized-VF-wave}
  \langle B y_\varrho , q \rangle_Q =
  \langle \Box \widetilde{y}_\varrho , {\mathcal{E}} q \rangle_{Q_-} =
  \langle u_\varrho , q \rangle_Q \quad
  \mbox{for all} \; q \in H^{1,1}_{0;,0}(Q).
\end{equation}
Thus, we might consider the reduced cost functional 
\begin{equation}\label{eq:cost-functional-energy}
  \widetilde{\mathcal{J}}(y_\varrho) =
  \frac{1}{2} \, \| y_\varrho-y_d \|_{L^2(Q)}^2 +
  \frac{\varrho}{2} \, \| By_\varrho \|_{P^\ast}^2. 
\end{equation} 
To realize the norm of the dual space $P^\ast$, we make use of the
following auxiliary Riesz operator $ A : P \to P^\ast$ defined by
\begin{equation}\label{eq:wave-A-operator}
  \langle Ap,q \rangle_Q :=
  \langle \partial_t p,\partial_t q \rangle_{L^2(Q)} +
  \langle \nabla _x p,\nabla_x q \rangle_{L^2(Q)}
  \quad \text{for all } p,q\in P;
\end{equation}
see also \cite{LLSY:LoescherSteinbach:2024SINUM}.
With this, the reduced cost functional becomes 
\[
  \widetilde{\mathcal{J}}(y_\varrho) 
  =
  \frac{1}{2} \, \langle y_\varrho-y_d,y_\varrho-y_d \rangle_{L^2(Q)} +
  \frac{\varrho}{2} \,
  \langle A^{-1} By_\varrho, By_\varrho \rangle_Q,
\]
for which the minimizer is characterized as the solution $y_\varrho\in Y$
of the gradient equation
\begin{equation}\label{eq:wave-VF-Schur-complement}
  \varrho \, B^\ast A^{-1} B y_\varrho + y_\varrho = y_d \quad \text{in } Y^\ast. 
\end{equation}
Note that the operator $S:=B^\ast A^{-1} B : Y\to Y^\ast$ is an isomorphism,
since $A : P \to P^\ast$, and $B : Y \to P^\ast$ are isomorphic and therefore,
\eqref{eq:wave-VF-Schur-complement} admits a unique solution. In particular,
see \cite[Lemma 3.1]{LLSY:LoescherSteinbach:2024SINUM},
$S:Y\to Y^\ast$ is bounded, self-adjoint, and $Y$-elliptic, and thus
defines an equivalent norm 
\begin{equation}\label{eq:wave-norm-equivalence}
  \| y \|_Y \leq \| y \|_S :=
  \sqrt{\langle Sy,y \rangle_Q} \leq
  2 \, \| y \|_Y \quad \text{for all }y\in Y. 
\end{equation} 
Depending on the regularity of the target $y_d$ and on the regularization
parameter $\varrho > 0$, we can show the following
regularization error estimates.

\begin{lemma}[{\cite[Theorem 3.2]{LLSY:LoescherSteinbach:2024SINUM}}]
  \label{lem:wave-regularization-error-estimates}
  Let $y_d \in L^2(Q)$ be given. For the unique solution $y_\varrho \in Y$
  of \eqref{eq:wave-VF-Schur-complement} there holds 
  \begin{align}\label{eq:wave-rho-H-H-estimate}
    \| y_\varrho -y_d \|_{L^2(Q)} \leq \| y_d \|_{L^2(Q)}.
  \end{align}
  Further, if $y_d\in Y$, then 
  \begin{align}\label{eq:wave-rho-HandX-X-estimate}
    \| y_\varrho -y_d \|_{L^2(Q)} \leq \sqrt{\varrho} \,
    \| y_d \|_S, \quad \text{and} \quad
    \| y_\varrho-y_d \|_S \leq \| y_d \|_S.
  \end{align}
  Moreover, it holds
  \begin{equation}\label{eq:wave-rho-XandX-stability}
    \| y_\varrho \|_S \leq \| y_d \|_S.
  \end{equation}
  At last, if $y_d\in Y$ such that $Sy_d\in L^2(Q)$, then it holds
  \begin{align}\label{eq:wave-rho-HandX-Syd-estimate}
    \| y_\varrho -y_d \|_{L^2(Q)} \leq \varrho \, \| Sy_d \|_{L^2(Q)},
    \quad \text{and} \quad \| y_\varrho-y_d \|_S \leq
    \sqrt{\varrho} \, \| Sy_d \|_{L_2(Q)},
  \end{align}
  as well as
  \begin{equation}\label{eq:wave-rho-SandS-stability}
    \| Sy_\varrho \|_{L^2(Q)}\leq \| Sy_d \|_{L^2(Q)}.
  \end{equation}
\end{lemma}

\noindent
From the above results, and using a space interpolation argument, we
conclude the following estimate, see
\cite[Corollary 3.3]{LLSY:LoescherSteinbach:2024SINUM}:
Let $y_d\in H^{s,s}_{0;0,}(Q) := [H^{1,1}_{0;0,}(Q),L^2(Q)]_s$, for
$s\in[0,1]$, or $y_d \in H^s(Q) \cap H^{1,1}_{0;0,}(Q)$ such that
$Sy_d\in H^{s-2}(Q)$ for $s\in [1,2]$. Then, 
\begin{equation}\label{eq:wave-interpolation-L2Hs-eq1}
  \| y_d-y_\varrho \|_{L^2(Q)} \leq
  c \, \varrho^{s/2} \, \| y_d \|_{H^s(Q)}, \quad s\in[0,2].
\end{equation}

\begin{remark}\label{rem:lower-order-convergence-energy-regularization}
  The operator $A = -\Delta_{(x,t)} : P = H_{0;,0}^{1,1}(Q) \to P^\ast$
  corresponds to the space-time Laplacian with mixed Dirichlet and Neumann
  boundary conditions. Therefore, the solution $p\in P$ of $Ap=u$ in
  $Q$ admits the regularity $p\in H^{r+1}(Q) \cap P$ for given
  $u\in H^{r-1}(Q)$, and some $0\leq r\leq 1$, depending on the geometry
  of the space-time domain, see, e.g.,
  \cite{LLSY:Dauge:1988,LLSY:Grisvard:1985}. In particular,
  for $y_d\in H^2(Q)\cap H_{0;0,}^{1,1}(Q)$ it holds that
  $B y_d\in L^2(Q)$. But, we can in general \emph{not}
  guarantee that $ A^{-1} By_d\in H^2(Q)$, and subsequently,
  $Sy_d = B^\ast A^{-1} By_d\in L^2(Q)$ does \emph{not} need to hold true.
  This loss of regularity might lead to lower convergence rates in the
  numerical examples. 
\end{remark}

\begin{remark}
  Instead of $Y=\mathcal{H}_{0;0,}(Q)$ and $P^\ast=[H_{0;,0}^{1,1}(Q)]^\ast$ we
  might as well consider the strong formulation of the wave equation with
  $P^\ast=L^2(Q)$, and a suitable ansatz space
  $\mathcal{Y} \subset H^{1,1}_{0;,0}(Q)$ such that the wave operator
  $B : \mathcal{Y} \to L^2(Q)$ is an isomorphism, see
  \cite[Section 4.3]{LLSY:Zank:2020}. Then, choosing
  $A:=\operatorname{id}:L^2(Q)\to L^2(Q)$, we can redo the above steps,
  deriving the optimality equation \eqref{eq:wave-VF-Schur-complement} with
  a bounded, self-adjoint, and elliptic operator
  $S:= B^\ast B:\mathcal{Y}\to \mathcal{Y}^\ast$, and related regularization
  error estimates depending on the regularity of the target, and on the
  parameter $\varrho>0$. In particular, the $L^2$ regularization corresponds
  to the concept of the energy regularization, if we consider the wave
  operator in a strong form.   	  
\end{remark}

%
%
\section{Space-time finite element discretization}
\label{Sec:SpaceTimeFiniteElementDiscretization}
From now on, let us assume that $\Omega \subset \mathbb{R}^d$ is either
polygonally ($d=2$) or polyhedrally ($d=3$) bounded.
Let $\mathcal{T}_h = \{ \tau_\ell \}_{\ell=1}^N$ be an admissible, globally
quasi-uniform decomposition of $Q$ into shape regular simplicial finite
elements $\tau_\ell \subset \mathbb{R}^{d+1}$ of mesh size
$h_\ell = |\tau_\ell|^{1/(d+1)}$, $\ell=1,\ldots,N$. Further,
let $h = \max_{\ell=1,\ldots,N}h_\ell$ denote the maximal mesh size.
For the Galerkin discretization of the above derived optimality equations,
we introduce conforming finite element spaces of, e.g., piecewise
linear and continuous basis functions,
\[
  Y_h = \text{span} \{ \varphi_k \}_{k=1}^{n_h} =
  S_h^1(\mathcal{T}_h) \cap H^{1.1}_{0;0,}(Q)
  \subset Y = \mathcal{H}_{0;0,}(Q),
\]
and 
\[
  P_h = \text{span} \{\psi_i \}_{i=1}^{m_h} = S_h^1(\mathcal{T}_h) \cap
  H^{1.1}_{0;,0}(Q) \subset P = H_{0;,0}^{1,1}(Q).
\]
In the following, we will formulate discrete variational formulations
for the optimality systems, and derive related error estimates, which will
enable us to link the regularization parameter $\varrho$ to the mesh
size $h$, revealing an asymptotically optimal choice $\varrho = h^4$, and
$\varrho=h^2$, in the case of the $L^2$ regularization and the energy
regularization in $P^\ast$, respectively. 

\subsection{The $L^2$ regularization $U=L^2(Q)$}
\label{SubSec:SpaceTimeFiniteElementDiscretization:L2}
In order to derive error estimates, it will be useful to first apply the
transformation $p_\varrho = \sqrt{\varrho} \, \widetilde{p}_\varrho$. Then,
\eqref{eq:wave-continuous-VF-optimality-system-common-L2} becomes:
find $(\widetilde{p}_\varrho,y_\varrho) \in P \times H_{0;0,}^{1,1}(Q)$ such
that   
\begin{equation}\label{eq:wave-transformed-VF-optimality-system-common-L2}
  \begin{array}{rcrcll}
    \displaystyle \frac{1}{\sqrt{\varrho}} \,
    \langle \widetilde{p}_{\varrho}, q \rangle_{L^2(Q)}
    & \hspace*{-2mm} + \hspace*{-2mm}
    & b(y_{\varrho },q)
    & \hspace*{-2mm} = \hspace*{-2mm} & 0,
    & \forall  q\in P, \\[2mm]
    -b(z,\widetilde{p}_{\varrho })
    & \hspace*{-2mm} + \hspace*{-2mm}
    & \displaystyle \frac{1}{\sqrt{\varrho}} \,
          \langle y_{\varrho },z \rangle_{L^2(Q)}
    & \hspace*{-2mm} = \hspace*{-2mm} &
          \displaystyle \frac{1}{\sqrt{\varrho}} \,
          \langle y_d,z \rangle_{L^2(Q)},
    & \forall  z\in H_{0;0,}^{1,1}(Q). 
  \end{array}
\end{equation}
The Galerkin variational formulation is then to find
$(\widetilde{p}_{\varrho h},y_{\varrho h})\in P_h\times Y_h$ such that   
 \begin{equation}\label{eq:wave-DVF-optimality-system-common-L2}
  \begin{array}{rcrcll}
    \displaystyle \frac{1}{\sqrt{\varrho}} \,
    \langle \widetilde{p}_{\varrho h}, q_h \rangle_{L^2(Q)}
    & \hspace*{-2mm} + \hspace*{-2mm}
    & b(y_{\varrho h},q_h)
    & \hspace*{-2mm} = \hspace*{-2mm} & 0,
    & \forall  q_h\in P_h, \\[2mm]
    -b(z_h,\widetilde{p}_{\varrho h})
    & \hspace*{-2mm} + \hspace*{-2mm}
    & \displaystyle \frac{1}{\sqrt{\varrho}} \,
          \langle y_{\varrho h},z_h \rangle_{L^2(Q)}
    & \hspace*{-2mm} = \hspace*{-2mm} &
          \displaystyle \frac{1}{\sqrt{\varrho}} \,
          \langle y_d,z_h \rangle_{L^2(Q)},
    & \forall  z_h \in Y_h. 
  \end{array}
\end{equation}

\begin{lemma}\label{lem:unique-solvability-L2-reg-DVF}
  For any $y_d \in L^2(Q)$, the Galerkin formulation
  \eqref{eq:wave-DVF-optimality-system-common-L2} admits a unique solution
  $(\widetilde{p}_{\varrho h},y_{\varrho h}) \in P_h\times Y_h$. 
\end{lemma}

\begin{proof}
  Testing \eqref{eq:wave-DVF-optimality-system-common-L2} with
  $q_h = \widetilde{p}_{\varrho h}$, and with $z_h= y_{\varrho h}$, and
  summing up both equations, we get 
  \[
    \| \widetilde{p}_{\varrho h} \|_{L^2(Q)}^2 +
    \| y_{\varrho h} \|_{L^2(Q)}^2 = \langle y_d,y_{\varrho h} \rangle_{L^2(Q)}. 
  \]
  Thus, for the homogeneous case $y_d = 0$, we see that
  $\widetilde p_{\varrho h}=y_{\varrho h} = 0$, which yields uniqueness
  for the solution of the linear problem. Moreover, in the finite
  dimensional case, uniqueness implies existence.  
\end{proof}

\begin{lemma}\label{lem:Ceas-lemma-common-L2}
  Assume the global inverse inequality
  \begin{equation}\label{eq:global-inverse-inequality}
    |v_h|_{H^1(Q)} \leq c_{\text{\tiny inv}} \, h^{-1} \,
    \| v_h \|_{L^2(Q)} \quad \text{for all } v_h\in S_h^1(\mathcal{T}_h).
  \end{equation}
  If we choose $\varrho =h^4$, then
  \begin{eqnarray}\label{eq:Cea-common-L2}
    && \hspace*{-7mm}h^{-2}
       \| \widetilde{p}_\varrho-\widetilde{p}_{\varrho h} \|_{L^2(Q)}^2
       + |\widetilde{p}_\varrho-\widetilde{p}_{\varrho h}|_{H^1(Q)}^2 +
       h^{-2} \| y_\varrho-y_{\varrho h} \|_{L^2(Q)}^2 +
       |y_\varrho-y_{\varrho h}|_{H^1(Q)}^2 \\
    && \hspace*{-4mm} \leq \, c \, \Big[
       h^{-2}\| \widetilde{p}_\varrho - q_h \|_{L^2(Q)}^2 +
       |\widetilde{p}_\varrho-q_h|_{H^1(Q)}^2 +
       h^{-2} \| y_\varrho-z_h \|_{L^2(Q)}^2 +
       |y_\varrho-z_h|_{H^1(Q)}^2 \Big] \nonumber
  \end{eqnarray}
  holds for all $(q_h,z_h)\in P_h\times Y_h$.  
\end{lemma}

\begin{proof}
  For given $(p,y) \in P \times H_{0;0,}^{1,1}(Q)$, let
  $(p_h,y_h) \in P_h\times Y_h$ denote the unique solution of 
  \begin{equation}\label{eq:wave-DVF-Galerkin-orthog-common-L2}
    \begin{array}{rclll}
      \displaystyle \frac{1}{\sqrt{\varrho}} \,
      \langle p_h, q_h \rangle_{L^2(Q)} + b(y_h,q_h)
      & \hspace*{-2mm} = \hspace*{-2mm}
      & \displaystyle \frac{1}{\sqrt{\varrho}} \,
            \langle p, q_h \rangle_{L^2(Q)} + b(y,q_h),
      & \forall q_h \in P_h, \\[2mm] 
      \displaystyle - b(z_h,p_h) + \frac{1}{\sqrt{\varrho}} \,
      \langle y_h,z_h \rangle_{L^2(Q)}
      & \hspace*{-2mm} = \hspace*{-2mm}
      & \displaystyle -b(z_h,p) + \frac{1}{\sqrt{\varrho}} \,
            \langle y,z_h \rangle_{L^2(Q)}, & \forall  z_h\in Y_h, 
    \end{array}
  \end{equation}
  which induces the Galerkin projection $(p,y)\to (p_h,y_h)$. If we can
  show that the Galerkin projection is bounded, this immediately implies
  Cea's lemma, i.e., the estimate \eqref{eq:Cea-common-L2}. Using the
  global inverse inequality \eqref{eq:global-inverse-inequality} and
  \eqref{eq:wave-DVF-Galerkin-orthog-common-L2}, we compute
  \begin{eqnarray*}
    &&\hspace{-6mm}\varrho^{-1/2} \, \| p_h \|_{L^2(Q)}^2 +
       |p_h|_{H^1(Q)}^2 + \varrho^{-1/2} \, \| y_h \|_{L^2(Q)}^2 +
       |y_h|_{H^1(Q)}^2 \\
    && \hspace*{-4mm}
       \leq \, \Big( \varrho^{-1/2} + c_{\text{\tiny inv}} h^{-2} \Big) \,
       \| p_h \|_{L^2(Q)}^2 +
       \Big( \varrho^{-1/2} + c_{\text{\tiny inv}} h^{-2} \Big) \,
       \| y_h \|_{L^2(Q)}^2 \\
    && \hspace*{-4mm}
       = \, \Big( 1 + c_{\text{\tiny inv}} h^{-2} \varrho^{1/2} \Big) \,
       \Big[ \varrho^{-1/2} \, \| p_h \|_{L^2(Q)}^2 +
       \varrho^{-1/2} \, \| y_h \|_{L^2(Q)}^2 \Big] \\
    && \hspace*{-4mm}
       = \, \Big( 1 + c_{\text{\tiny inv}} h^{-2} \varrho^{1/2} \Big)
       \Big[ \frac{1}{\sqrt{\varrho}} \, \langle p_{h}, p_{h} \rangle_{L^2(Q)} +
       b(y_h,p_h) - b(y_h,p_h) + \frac{1}{\sqrt{\varrho}} \,
       \langle y_h,y_h \rangle_{L^2(Q)} \Big] \\
    && \hspace*{-4mm} = \,
       \Big( 1 + c_{\text{\tiny inv}} h^{-2} \varrho^{1/2} \Big)
       \Big[ \frac{1}{\sqrt{\varrho}} \,
       \langle p,p_h \rangle_{L^2(Q)} + b(y,p_h)
       - b(y_h,p) + \frac{1}{\sqrt{\varrho}} \,
       \langle y,y_h \rangle_{L^2(Q)} \Big] \\
    && \hspace*{-4mm}
       \leq \, \Big( 1 + c_{\text{\tiny inv}} h^{-2} \varrho^{1/2} \Big)
       \Big[ \frac{1}{\sqrt{\varrho}} \, \| p \|_{L^2(Q)} \| p_h \|_{L^2(Q)}
       + |y|_{H^1(Q)} |p_h|_{H^1(Q)} \\
    && \hspace{40mm} + |y_h|_{H^1(Q)} |p|_{H^1(Q)} +
       \frac{1}{\sqrt{\varrho}} \, \| y \|_{L^2(Q)} \| y_h \|_{L^2(Q)} \Big] \\
    && \hspace*{-4mm} \leq \,
       \Big( 1 + c_{\text{\tiny inv}} h^{-2} \varrho^{1/2} \Big)
       \Big[ \frac{1}{\sqrt{\varrho}} \, \| p \|_{L^2(Q)}^2 + |p|_{H^1(Q)}^2
       + \frac{1}{\sqrt{\varrho}} \, \| y \|_{L^2(Q)}^2 +
       |y|_{H^1(Q)}^2 \Big]^{1/2} \\
    &&  \hspace{25mm} \cdot \Big[ \frac{1}{\sqrt{\varrho}} \,
       \| p_h \|_{L^2(Q)}^2 +|p_h|_{H^1(Q)}^2 + \frac{1}{\sqrt{\varrho}} \,
       \| y_h \|_{L^2(Q)}^2 + |y_h|_{H^1(Q)}^2 \Big]^{1/2}.
  \end{eqnarray*}
  Thus, choosing $\varrho = h^4$, we obtain 
  \begin{eqnarray*}
    && h^{-2} \, \| p_h \|_{L^2(Q)}^2 + |p_h|_{H^1(Q)}^2 +
       h^{-2} \, \| y_h \|_{L^2(Q)}^2 + |y_h|_{H^1(Q)}^2 \\
    && \hspace*{1cm} \leq \, \Big( 1 + c_{\text{\tiny inv}} \Big) \Big[
       h^{-2} \, \| p \|_{L^2(Q)}^2 + |p|_{H^1(Q)}^2 +
       h^{-2} \, \| y \|_{L^2(Q)}^2 + |y|_{H^1(Q)}^2 \Big], 
  \end{eqnarray*}
  implying the desired bound. 
\end{proof}

\noindent
Combining the regularization error estimates with the above best
approximation, we can now characterize the error
$\| y_d-y_{\varrho h} \|_{L^2(Q)}$ depending on the regularity of the
target $y_d$. 

\begin{theorem}\label{thm:l2-reg:fem-estiamtes}
  For $y_d \in L^2(Q)$, let $(\widetilde{p}_{\varrho h},y_{\varrho h}) \in
  P_h \times Y_h$ be the unique solution of
  \eqref{eq:wave-DVF-optimality-system-common-L2}. Then, 
  \begin{equation}\label{eq:common-L2:discrete-L2-L2}
    \| y_{\varrho h}-y_d \|_{L^2(Q)}\leq \| y_d \|_{L^2(Q)}.
  \end{equation}
  Moreover, let the assumptions of Lemma \ref{lem:Ceas-lemma-common-L2} hold,
  i.e., a global inverse inequality, and $\varrho=h^4$.
  Let $y_d\in H_{0;0,}^{1,1}(Q) \cap H^2(Q)$. Then,
  \begin{equation}\label{eq:common-L2:discrete-L2-H2}
    \| y_{\varrho h}-y_d \|_{L^2(Q)} \leq c \, h^2 \, | y_d |_{H^2(Q)} .
  \end{equation}
\end{theorem}

\begin{proof}
  The estimate \eqref{eq:common-L2:discrete-L2-L2} follows the lines of the
  continuous case in Lemma \ref{lem:regularization-error-estimates-L2},
  equation \eqref{eq:reg-error-estimate-L2:L2-L2}. To show
  \eqref{eq:common-L2:discrete-L2-H2}, first note, that by a triangle
  inequality, we have that 
  \begin{equation*}
    \| y_{\varrho h}-y_d \|_{L^2(Q)} \leq
    \| y_{\varrho h}-y_\varrho \|_{L^2(Q)} + \| y_\varrho-y_d \|_{L^2(Q)}. 
  \end{equation*}
  The second term can further be estimated, using
  \eqref{eq:reg-error-estimate-L2:L2-H2} for $\varrho = h^4$, by
  \begin{equation*}
    \| y_\varrho-y_d \|_{L_2(Q)} \leq \sqrt{\varrho} \,
    \| \Box y_d \|_{L^2(Q)} \leq h^2 \, \| y_d \|_{H^2(Q)}.
  \end{equation*}
  For the first term, we consider the estimate
    \eqref{eq:Cea-common-L2}, i.e.,
    \begin{eqnarray*}
      h^{-2} \, \| y_\varrho-y_{\varrho h} \|_{L^2(Q)}^2
      & \leq & c \, \Big[
       h^{-2} \, \| \widetilde{p}_\varrho - q_h \|_{L^2(Q)}^2 +
               |\widetilde{p}_\varrho-q_h|_{H^1(Q)}^2 \\
      && \hspace*{2cm} + h^{-2} \, \| y_\varrho-z_h \|_{L^2(Q)}^2 +
         |y_\varrho-z_h|_{H^1(Q)}^2 \Big] ,
    \end{eqnarray*}
    and it remains to bound all terms of the right hand side. Let
    $q_h = \Pi_h \widetilde{p}_\varrho$ be the Scott--Zhang quasi
    interpolation \cite{LLSY:ScottZhang:1990}, satisfying the best
    approximation and stability estimates
    \begin{equation*}
      \| \widetilde{p}_\varrho - \Pi_h \widetilde{p}_\varrho \|_{L_2(Q)}^2
      \, \leq \, c \, h^2 \, |\widetilde{p}_\varrho|_{H^1(Q)}^2, \quad
      | \widetilde{p}_\varrho- \Pi_h \widetilde{p}_\varrho |_{H^1(Q)}^2
      \, \leq \, c \, |\widetilde{p}_\varrho|_{H^1(Q)}^2 .  
    \end{equation*}
    With this, using $\widetilde{p}_\varrho = \varrho^{-1/2} \, p_\varrho$
    and \eqref{Estimate p H1 H2}, we have, recall $\varrho=h^4$
    \begin{eqnarray*}
      && h^{-2} \, \| \widetilde{p}_\varrho - q_h \|_{L^2(Q)}^2 +
         |\widetilde{p}_\varrho-q_h|_{H^1(Q)}^2 \, \leq \,
         c \, |\widetilde{p}_\varrho|^2_{H^1(Q)} \, = \,
         c \, \varrho^{-1} \, |p_\varrho|^2_{H^1(Q)} \\
      && \hspace*{2cm} \leq \,
         c \, \varrho^{1/2} \, \| \Box y_d \|_{L^2(Q)}^2 \, = \,
         c \, h^2 \, \| \Box y_d \|^2_{L^2(Q)} \, \leq \,
         c \, h^2 \, |y_d|_{H^2(Q)}^2 .
    \end{eqnarray*}
    Next, we consider, using a triangle inequality,
    \eqref{eq:reg-error-estimate-L2:L2-H2}, and choosing
    $z_h = \Pi_hy_d$ to conclude, recall $\varrho=h^4$,
    \begin{eqnarray*}
      \| y_\varrho - z_h \|_{L^2(Q)}
      & \leq & \| y_\varrho - y_d \|_{L^2(Q)}+
               \| y_d - \Pi_h y_d \|_{L^2(Q)} \\
      & \leq & \sqrt{\varrho} \, \| \Box y_d \|_{L^2(Q)} +
               c \, h^2 \, |y_d|_{H^2(Q)} \leq c \, h^2 \, |y_d|_{H^2(Q)} .
    \end{eqnarray*}
    Moreover, now using \eqref{Estimate y H1 H2}, we also have
    \begin{eqnarray*}
      |y_\varrho - \Pi_hy_d|_{H^1(Q)}
      & \leq & |y_\varrho - y_d |_{H^1(Q)} + |y_d - \Pi_h y_d|_{H^1(Q)} \\
      & \leq & \varrho^{1/4} \, \| \Box y_d \|_{L^2(Q)} +
               c \, h \, |y_d|_{H^2(Q)} \, \leq \,
               c \, h \, |y_d|_{H^2(Q)} .
    \end{eqnarray*}
    Finally, collecting all terms together, the assertion follows.
\end{proof}

\subsection{The energy regularization in
  $U=P^\ast=[H_{0;,0}^{1,1}(Q)]^\ast$}
\label{SubSec:SpaceTimeFiniteElementDiscretization:Energy}
This section follows the results presented in
\cite{LLSY:LoescherSteinbach:2024SINUM}. Recall, that the state
$y_\varrho\in Y=\mathcal{H}_{0;0,}(Q)$, minimizing the reduced cost
functional \eqref{eq:cost-functional-energy}, was characterized as the
unique solution of the operator equation \eqref{eq:wave-VF-Schur-complement}.
This is equivalent to the variational formulation to find
$y_\varrho\in Y$ such that 
\begin{equation}\label{eq:energy-VF}
  \varrho \, \langle Sy_\varrho,z \rangle_Q +
  \langle y_\varrho,z \rangle_{L^2(Q)} =
  \langle y_d,z \rangle_{L^2(Q)}\quad \text{for all }z\in Y,  
\end{equation} 
with the linear operator
$S:= B^\ast A^{-1} B: Y\to Y^\ast$, which is bounded,
self-adjoint and $Y$-elliptic. The Galerkin variational formulation of the
problem is then to find $y_{\varrho h}\in Y_h$ such that 
\begin{equation}\label{eq:energy-DVF}
  \varrho \, \langle Sy_{\varrho h},z_h \rangle_Q +
  \langle y_{\varrho h},z_h \rangle_{L^2(Q)} =
  \langle y_d,z_h \rangle_{L^2(Q)}\quad \text{for all }z_h\in Y_h.  
\end{equation} 
Due to the choice of a conforming subspace $Y_h\subset Y$ we obtain, using
standard arguments, unique solvability of \eqref{eq:energy-DVF} and the
Cea type a priori estimate
\begin{equation}\label{eq:energy:Cea-type-estimate}
  \varrho \, \| y_{\varrho}-y_{\varrho h} \|_S^2 +
  \| y_{\varrho }-y_{\varrho h} \|_{L^2(Q)}^2 \leq
  \inf_{z_h\in Y_h} \Big[ \varrho \,
  \| y_{\varrho }-z_h \|_S^2 + \| y_{\varrho}-z_h \|_{L^2(Q)}^2 \Big]. 
\end{equation}
Combining this best approximation result with the regularization error
estimates in Lemma \ref{lem:wave-regularization-error-estimates}, we can
derive an asymptotically optimal choice for the regularization parameter
$\varrho$ depending solely on the regularity of target.   

\begin{lemma}[{\cite[Theorem 4.1]{LLSY:LoescherSteinbach:2024SINUM}}]
  \label{lem:energy-FEM-estiamtes-Hs}
  Let $y_d\in H^{s,s}_{0;0,}(Q):=[H^{1,1}_{0;0,}(Q),L_2(Q)]_s$ for some
  $s\in[0,1]$, or $y_d\in H^{1,1}_{0;0,}(Q)\cap H^s(Q)$ for $s\in [1,2]$.
  If we choose $\varrho = h^2$, then for the unique solution
  $y_{\varrho h}\in Y_h$ of \eqref{eq:energy-DVF} there holds 
  \begin{equation}\label{eq:energy:FEM-estaimte-Hs}
    \| y_{\varrho h}-y_d \|_{L^2(Q)} \leq c \, h^s \, \| y_d \|_{H^s(Q)}.
  \end{equation}
\end{lemma}

\noindent
For the numerical treatment, the variational formulation
\eqref{eq:energy-DVF} is not suitable, as the realization of the operator
$S$ is not computable. Thus, in the last step of our analysis we will
consider a computable realization of $S$, leading to a perturbed variational
formulation. Therefore, let $y\in Y$ be arbitrary but fixed and let us
consider the auxiliary problem to find $p_y\in P$ such that 
\begin{equation*}
  \langle Ap_y,q \rangle_Q =
  \langle By,q \rangle_Q \quad \text{for all }q\in P. 
\end{equation*}  
Then, $Sy = B^\ast p_y$. To define an approximation, we introduce
$p_{yh}\in P_h$ as unique solution of 
\begin{equation}\label{eq:energy:auxiliary-perturbed}
  \langle Ap_{y h},q_h \rangle_Q =
  \langle By,q_h \rangle_Q \quad \text{for all }q_h\in P_h, 
\end{equation}  
and define $\widetilde Sy := B^\ast p_{yh}$. Then we consider the perturbed
variational formulation to find $\widetilde y_{\varrho h}\in Y_h$ such that 
\begin{equation}\label{eq:energy:perturbed-DVF}
  \varrho \, \langle \widetilde S\widetilde y_{\varrho h},z_h \rangle_Q +
  \langle \widetilde y_{\varrho h},z_h \rangle_{L^2(Q)} =
  \langle y_d,z_h \rangle_{L^2(Q)} \quad \text{for all }z_h\in Y_h.  
\end{equation} 
Note, that due to the properties of $A:P\to P^\ast$ and $B:Y\to P^\ast$,
the operator $\widetilde S:Y\to Y^\ast$ is bounded, symmetric and positive
semi-definite. Thus, \eqref{eq:energy:perturbed-DVF} admits a unique
solution. Moreover, we see that the perturbation error solely depends on
the best approximation properties of $P_h\subset P$. Thus, using a Strang
lemma argument, which requires an inverse inequality, we can prove analogous
estimates as in Lemma \ref{lem:energy-FEM-estiamtes-Hs} for the solution of
the perturbed variational formulation. 

\begin{theorem}[{\cite[Corollary 4.7]{LLSY:LoescherSteinbach:2024SINUM}}]
  \label{thm:energy:perturbed-DVF}
  Let the global inverse inequality \eqref{eq:global-inverse-inequality}
  hold and choose $\varrho = h^2$. Then the unique solution
  $\widetilde y_{\varrho h}\in Y_h$ of \eqref{eq:energy-DVF} satisfies
  \begin{equation}
    \| \widetilde y_{\varrho h}-y_d \|_{L^2(Q)} \leq
    c \, h^s \, \| y_d \|_{H^s(Q)},\, s\in[0,2],
  \end{equation}
  if $y_d\in H^{s,s}_{0;0,}(Q)$ for $s\in[0,1]$, or
  $y_d\in H^{1,1}_{0;0,}(Q)\cap H^s(Q)$ for $s\in[1,2]$. 
\end{theorem} 

\noindent
When introducing $\widetilde p_{\varrho h} = -\varrho p_{\widetilde y_{\varrho h}}$,
where $p_{\widetilde y_{\varrho h}}$ solves \eqref{eq:energy:auxiliary-perturbed}
for $y=\widetilde y_{\varrho h}$, we see that the perturbed variational
formulation \eqref{eq:energy:perturbed-DVF} is equivalent to the coupled
system to find
$(\widetilde p_{\varrho h},\widetilde y_{\varrho h})\in P_h\times Y_h$ such that 
\begin{equation}\label{eq:wave-DVF-optimality-system-common-L2 double}
  \begin{array}{rcrcll}
    \varrho^{-1} \, \langle A \widetilde p_{\varrho h}, q_h \rangle_Q
    & + & \langle B \widetilde y_{\varrho h},q_h \rangle_Q
    & = & 0 &\; \text{for all } q_h\in P_h,\\[1mm]
    - \langle B z_h,\widetilde p_{\varrho h} \rangle_Q
    & + &\langle \widetilde y_{\varrho h},z_h \rangle_{L^2(Q)}
    &=& \langle y_d,z_h \rangle_{L^2(Q)},&\; \text{for all } z_h\in Y_h.
  \end{array}
\end{equation}
This will be the starting point for the numerical treatment of the problem. 
We stress again that, by \eqref{eq:evaluation-wave-bilinear-form},
we have for $y_h\in Y_h\subset H_{0;0,}^{1,1}(Q)$ and
$q_h\in P_h\subset H^{1,1}_{0;,0}(Q)$ that 
\[
  \langle B y_h , q_h \rangle_Q
  =
  - \langle \partial_t y_h , \partial_t q_h \rangle_{L^2(Q)}
  + \langle \nabla_x y_h , \nabla_x q_h \rangle_{L^2(Q)}
  \; \mbox{for all} \; y_h \in Y_h, \, q_h \in P_h .
\]

\begin{remark}\label{Sec:FEM:Remark:LocalInverseInequality}
  In the proof of Lemma \ref{lem:Ceas-lemma-common-L2} and
  Theorem \ref{thm:energy:perturbed-DVF} we have used a global inverse
  inequality, which in general assumes a globally quasi-uniform mesh.
  However, in the numerical treatment we will also consider a variable
  regularization parameter
  $\varrho(x,t) = h_\tau^r,\; \forall \, (x,t) \in \tau,\;
  \forall \tau \in \mathcal{T}_h$
  with $r=4$ for $L^2$-regularization, and $r=2$ for energy regularization,
  where it seems to be sufficient to consider a local inverse inequality
  \begin{equation}\label{Sec:FEM:LocalInverseInequality}
    \|\nabla v_h\|_{L^2(\tau)} \le c_\text{\tiny inv} \,
    h_\tau^{-1} \ \|v_h\|_{L^2(\tau)}
    \quad \forall v_h \in S_h^1(\mathcal{T}_h),
    \forall \tau \in \mathcal{T}_h;
  \end{equation}
  see \cite{LLSY:LangerLoescherSteinbachYang:2024CAMWA} for a related
  approach for a distributed optimal control problem subject to the
  Poisson equation.  
\end{remark}

\section{Solvers}\label{Sec:Solvers}
Let us first specify the submatrices $B_h$ and $M_h$ appearing in the SID
system \eqref{Sec:Introduction:Eqn:AbstractDiscreteReducedOptimalitySystem} 
for hyperbolic OCPs. The coefficients $B_h[j,k]$ of the $m_h \times n_h$
rectangular wave matrix $B_h$ are defined by
\begin{equation}\label{Sec:Solvers:Eqn:Bh[j,k]}
  B_h[j,k] =
  - \langle \partial_t\varphi_{k}, \partial_t\psi_{j} \rangle_{L^2(Q)} 
  + \langle \nabla_x \varphi_{k}, \nabla_x\psi_{j} \rangle_{L^2(Q)},
\end{equation}
for all $j=1,\ldots,m_h$ and $k=1,\ldots,n_h$,
whereas the coefficients $M_h[l,k]$ of the SPD $n_h \times n_h$ mass
matrix $M_h$ are given by
\begin{equation}\label{Sec:Solvers:Eqn:Mh[l,k]}
  M_h[l,k] = \langle \varphi_{k}, \varphi_{l} \rangle_{L^2(Q)} \;
  \forall \, l,k = 1,\ldots,n_h.
\end{equation}
Later we will heavily use that the mass matrix $M_h$ is spectrally equivalent 
to the lumped mass matrix $D_h =  \text{lump}(M_h)$ satisfying 
the spectral equivalent inequalities
\begin{equation}\label{Sec:Solvers:Eqn:MhDh}
  (d+2)^{-1} D_h \le M_h \le D_h;
\end{equation}
see, e.g., \cite{LLSY:LangerLoescherSteinbachYang:2023arXiv:2304.14664}.
The $m_h \times m_h$ matrix $A_{\varrho h}$ is also SPD as we will see later 
when we consider the $L^2$ and the energy regularization in 
Subsections \ref{Sec:Solvers:SubSec:L2} and \ref{Sec:Solvers:SubSec:Energy},
respectively.
    
There are many methods for solving the SID system 
\eqref{Sec:Introduction:Eqn:AbstractDiscreteReducedOptimalitySystem};
see Section~\ref{Sec:Introduction} for some references. Here we focus on
Bramble–Pasciak's PCG (PB-PCG) \cite{LLSY:BramblePasciak:1988a}.
The basic idea consists in transforming the SID system 
\eqref{Sec:Introduction:Eqn:AbstractDiscreteReducedOptimalitySystem}
to the equivalent SPD system 
\begin{equation}
\label{Sec:Solvers:Eqn:Kx=y}
  {\mathcal K}_h
  \begin{bmatrix}
   \mathbf{p}_h\\ 
   \mathbf{y}_h\\ 
  \end{bmatrix}
  =
  \begin{bmatrix}
    \mathbf{0}\\ 
    \mathbf{y}_{dh}\\ 
  \end{bmatrix}
  := 
   \begin{bmatrix}
    A_{\varrho h}\widehat{A}_{\varrho h}^{-1}-I_h & 0 \\
    B_h^\top\widehat{A}_{\varrho h}^{-1} & -I_h
  \end{bmatrix}
  \begin{bmatrix}
    \mathbf{0}\\ 
   -\mathbf{y}_{dh}\\ 
  \end{bmatrix},
\end{equation}
where the new system matrix
\begin{equation*}
  \begin{aligned}
    {\mathcal K}_h=&
    \begin{bmatrix}
      A_{\varrho h}\widehat{A}_{\varrho h}^{-1}-I_h & 0\\ 
      B_{\varrho h}^\top \widehat{A}_{\varrho h}^{-1} & -I_h\\
    \end{bmatrix}
    \begin{bmatrix}
      A_{\varrho h} & B_h\\ 
      B_{h}^\top & -M_h\\
    \end{bmatrix}\\
    =&
    \begin{bmatrix}
      (A_{\varrho h}-\widehat{A}_{\varrho h})\widehat{A}_{\varrho h}^{-1}A_{\varrho h} & (A_{\varrho h}-\widehat{A}_{\varrho h})\widehat{A}_{\varrho h}^{-1}B_h\\ 
      B_h^\top \widehat{A}_{\varrho h}^{-1}(A_{\varrho h}-\widehat{A}_{\varrho h})
      & B_h^\top \widehat{A}_{\varrho h}^{-1}B_h+M_h \\
    \end{bmatrix}
  \end{aligned}
\end{equation*}
is SPD provided that $\widehat{A}_{\varrho h}$ is a properly scaled
preconditioner for $A_{\varrho h}$ such that the spectral equivalence
inequalities
\begin{equation}\label{Sec:Solvers:Eqn:SpectralEquivalenceInequalities:A}
  \widehat{A}_{\varrho h} < A_{\varrho h} \le
  \overline{c}_A \, \widehat{A}_{\varrho h}
\end{equation}
hold for some $h$-independent, positive constant $\overline{c}_A$. 
Now we can solve the SPD system \eqref{Sec:Solvers:Eqn:Kx=y} by means
of the PCG preconditioned by the SPD Bramble--Pasciak preconditioner
\begin{equation}\label{Sec:Solvers:Eqn:BPpreconditioner}
   \widehat{\mathcal{K}}_h
   =
   \begin{bmatrix}
    A_{\varrho h} - \widehat{A}_{\varrho h} & 0\\
    0 & \widehat{S}_{\varrho h}
  \end{bmatrix},
\end{equation}
where $\widehat{S}_h$ is some SPD preconditioner for the exact Schur complement 
$S_{\varrho h} = B_h^\top {A}_{\varrho h}^{-1}B_h+M_h$ such that  the spectral
equivalence inequalities%
\begin{equation}\label{Sec:Solvers:Eqn:SpectralEquivalenceInequalities:S}
  \underline{c}_S \, \widehat{S}_{\varrho h} < S_{\varrho h} \le
  \overline{c}_S \, \widehat{S}_{\varrho h}
\end{equation}
hold for some $h$-independent, positive constants $\underline{c}_S$ and
$\overline{c}_S$. The spectral equivalence inequalities 
\eqref{Sec:Solvers:Eqn:SpectralEquivalenceInequalities:A}
and
\eqref{Sec:Solvers:Eqn:SpectralEquivalenceInequalities:S}
yield the spectral equivalence inequalities
\begin{equation}\label{Sec:Solvers:Eqn:SpectralEquivalenceInequalities:K}
  \underline{c}_\mathcal{K} \, \widehat{\mathcal{K}}_h < {\mathcal{K}}_h
  \le \overline{c}_\mathcal{K} \, \widehat{\mathcal{K}}_h,
\end{equation}
where the positive constants $\underline{c}_\mathcal{K}$ and
$\overline{c}_\mathcal{K}$
can explicitly be computed from $\overline{c}_A$, $\underline{c}_S$, and
$\overline{c}_S$; see the original paper \cite{LLSY:BramblePasciak:1988a},
and \cite{LLSY:Zulehner:2002MCOM} for an improvement of the lower bound
$\underline{c}_\mathcal{K}$. Now the standard  PCG convergence rate estimates
in the ${\mathcal{K}}_h$ energy norm  directly follow from
\eqref{Sec:Solvers:Eqn:SpectralEquivalenceInequalities:K}.

Alternatively, we can solve the primal SPD Schur complement system
\begin{equation}\label{Sec:Solvers:Eqn:Sy=yd}
  {S}_{\varrho h}\mathbf{y}_h :=
  (B_h^\top {A}_{\varrho h}^{-1}B_h+M_h) \mathbf{y}_h = \mathbf{y}_{dh}
\end{equation}
by means of the standard PCG preconditioned by $\widehat{S}_{\varrho h}$.
This Schur Complement PCG (SC-PCG) has one drawback. 
The matrix-by-vector multiplication ${S}_{\varrho h} * \mathbf{y}_h^n$ 
requires the application of ${A}_{\varrho h}^{-1}$ that cannot easily be 
replaced by a preconditioner without perturbing the discretization error.
We will discuss this issue for the $L^2$ and energy regularizations 
in the following subsections separately.

\subsection{$L^2$ Regularization and Mass Lumping}
\label{Sec:Solvers:SubSec:L2}
For the standard $L^2$ regularization, the regularization matrix
$A_{\varrho h}$ is nothing but the SPD $m_h \times m_h$ mass matrix
$\overline{M}_{\varrho h}$, the coefficients of which are defined by
\begin{equation}\label{Sec:Solvers:Eqn:overlineMh[j,i]}
  \overline{M}_{\varrho h}[j,i] =
  \langle \varrho^{-1} \psi_{i}, \psi_{j} \rangle_{L^2(Q)} \;
  \forall \, j,i = 1,\ldots,m_h.
\end{equation}
Here we permit variable regularization of the form
\begin{equation}\label{Sec:Solvers:Eqn:L_2_VariableRegularization}
  \varrho(x,t) = h_\tau^4, \; \forall \, (x,t) \in \tau,\;
  \forall \tau \in \mathcal{T}_h,
\end{equation}
which we implemented in all numerical experiments when adaptive mesh
refinement is used. It is clear that
\eqref{Sec:Solvers:Eqn:L_2_VariableRegularization} 
turns to $\varrho = h^4$ in the case of uniform mesh refinement 
for which we have made the error analysis in 
Subsection~\ref{SubSec:SpaceTimeFiniteElementDiscretization:L2}.

It is also clear that $\overline{M}_{\varrho h}$ is spectrally equivalent to
$\overline{D}_{\varrho h} = \text{lump}(\overline{M}_{\varrho h})$ 
with the same spectral equivalence constants as given in
\eqref{Sec:Solvers:Eqn:MhDh} for $M_h$ and $D_h$, i.e.
\begin{equation}\label{Sec:Solvers:Eqn:overlineMhoverlineDh}
  (d+2)^{-1} \, \overline{D}_{\varrho h} \le \overline{M}_{\varrho h}
  \le \overline{D}_{\varrho h} .
\end{equation}
Now the following spectral equivalence inequalities 
are valid for the Schur complement $B^\top_h A_{\varrho h}^{-1} B_h + M_h$.

\begin{theorem}\label{Theorem:SpectralEquivalenceTheoremL2Regularization1}
Let us consider the optimally balanced, mesh-dependent, 
variable regularization \eqref{Sec:Solvers:Eqn:L_2_VariableRegularization},
and let $M_h$ as defined in \eqref{Sec:Solvers:Eqn:Mh[l,k]} with
$D_h = \text{\rm lump}(M_h)$.
Then the spectral equivalence inequalities 
\begin{equation}\label{Sec:Solvers:Eqn:L2RegSpectralEquivalenceInequalities}
  (d+2)^{-1} \, D_h \, \le \, M_h \, \le \,
  B^\top_h A_{\varrho h}^{-1} B_h + M_h 
  \, \le \, (c_\text{\tiny inv}^4+1) \, M_h 
  \, \le \, (c_\text{\tiny inv}^4+1) \, D_h
\end{equation}
hold for both $A_{\varrho h} = \overline{M}_{\varrho h}$ and 
$A_{\varrho h} = \overline{D}_{\varrho h} := \text{\rm lump}(\overline{M}_{\varrho h})$ 
corresponding to the standard $L^2$ regularization and the mass-lumped
$L^2$ regularization, respectively. The constant $c_\text{\tiny inv}$
originates from the inverse inequalities 
\eqref{Sec:FEM:LocalInverseInequality}.
\end{theorem}

\begin{proof}
  Using the spectral equivalence inequalities
  \eqref{Sec:Solvers:Eqn:overlineMhoverlineDh}, Cauchy's inequalities, and
  the inverse inequalities \eqref{Sec:FEM:LocalInverseInequality},
  we get the estimates
  \begin{eqnarray*}
    && \hspace*{-5mm}
       (B^\top_h\overline{D}_{\varrho h}^{-1} B_h \mathbf{y}_h,\mathbf{y}_h) 
    \, \le \, (B^\top_h\overline{M}_{\varrho h}^{-1} B_h \mathbf{y}_h,\mathbf{y}_h)
            =
            \sup_{\mathbf{q}_h \in \mathbb{R}^{m_h}}
            \frac{(B_h \mathbf{y}_h,\mathbf{q}_h)^2}
            {(\overline{M}_{\varrho h}\mathbf{q}_h,\mathbf{q}_h)} \\
    && \hspace*{-3mm} = \, \sup_{q_h \in P_h}
          \frac{\Big[ - \langle \partial_t y_h , \partial_t q_h
          \rangle_{L^2(Q)} + \langle \nabla_x y_h , \nabla_x q_h
          \rangle_{L^2(Q)} \Big]^2}
          {\langle \varrho^{-1} q_h , q_h \rangle_{L^2(Q)}} \\
    && \hspace*{-3mm} = \sup_{q_h \in P_h}
          \frac{\Big[ \langle \varrho^{\frac{1}{4}}
          \widehat{\nabla} y_h , \varrho^{-\frac{1}{4}}
          \nabla q_h\rangle_{L^2(Q)}\Big]^2}
          {\langle \varrho^{-1} q_h , q_h \rangle_{L^2(Q)}} \leq
       \sup_{q_h \in P_h}
       \frac{\|\varrho^{\frac{1}{4}} \widehat{\nabla} y_h\|_{L_2(Q)}^2
       \|\varrho^{-\frac{1}{4}} \nabla q_h\|_{L^2(Q)}^2}
       {\|\varrho^{-\frac{1}{2}} q_h\|_{L^2(Q)}^2}\\
    && \hspace*{-3mm} = \sup_{q_h \in P_h}
       \frac{\|\varrho^{\frac{1}{4}} \widehat{\nabla} y_h\|_{L_2(Q)}^2
       \sum\limits_{\tau \in {\mathcal{T}}_h} h_\tau^{-2} \,
       \| \nabla q_h \|^2_{L^2(\tau)}}
       {\|\varrho^{-\frac{1}{2}} q_h\|_{L^2(Q)}^2} \\
    && \hspace*{-3mm} \le \sup_{q_h \in P_h}
       \frac{\|\varrho^{\frac{1}{4}} \widehat{\nabla} y_h\|_{L_2(Q)}^2
       c_\text{\tiny inv}^2 \sum\limits_{\tau \in {\mathcal{T}}_h} h_\tau^{-4} \,
       \| q_h \|^2_{L^2(\tau)}}{\|\varrho^{-\frac{1}{2}} q_h\|_{L^2(Q)}^2} \\
    && \hspace*{-3mm} = \, c_\text{\tiny inv}^2
       \|\varrho^{\frac{1}{4}} \nabla y_h\|_{L_2(Q)}^2 =
       c_\text{\tiny inv}^2 \sum\limits_{\tau \in {\mathcal{T}}_h}
       h_\tau^2 \, \| \nabla y_h \|^2_{L^2(\tau)}
       \leq c_\text{\tiny inv}^4 \, \| y_h \|^2_{L^2(Q)}
       =
       c_\text{\tiny inv}^4 (M_h \mathbf{y}_h,\mathbf{y}_h)
\end{eqnarray*}
for all $\mathbf{y}_h \in \mathbb{R}^{n_h}$, $\mathbf{y}_h
\leftrightarrow y_h \in Y_h= S_h^1(\mathcal{T}_h)\cap H^{1.1}_{0;0,}(Q)$,
where $\widehat{\nabla} = (\nabla_x, - \partial_t)^\top$, and
$\nabla = (\nabla_x, \partial_t)^\top$ is the space-time gradient.
Combining theses estimates with the spectral equivalence inequalities
\eqref{Sec:Solvers:Eqn:MhDh} completes the proof of the theorem.
\end{proof}

\noindent
We note that the constant choice $\varrho = h^4$ leads to the same 
spectral equivalence inequalities
\eqref{Sec:Solvers:Eqn:L2RegSpectralEquivalenceInequalities}
as in the case of variable regularization since the constant regularization 
is a special case of variable regularization when we set $h_\tau = h$
for all $\tau \in \mathcal{T}_h$.

Thanks to \eqref{Sec:Solvers:Eqn:overlineMhoverlineDh} 
and \eqref{Sec:Solvers:Eqn:L2RegSpectralEquivalenceInequalities}, 
we can choose 
\begin{equation}\label{Sec:Solvers:Eqn:L2RegPrePBCG}
  \widehat{A}_{\varrho h} 
  = \delta (d+2)^{-1} \overline{D}_{\varrho h}
    := \delta (d+2)^{-1} \text{lump}(\overline{M}_{\varrho h})
  \; \text{and} \;
  \widehat{S}_{\varrho h} = D_h := \text{lump}(M_h)
\end{equation}
yielding the spectral equivalence constants
$\overline{c}_A = (d+2)/\delta$, $\underline{c}_S = 1/(d+2)$
and $\overline{c}_S = (c_\text{\tiny inv}^4+1)$,
where $\delta < 1$ is a properly chosen, positive scaling parameter.
Therefore, the PB-PCG is an asymptotically optimal solver for 
the SID system
\eqref{Sec:Introduction:Eqn:AbstractDiscreteReducedOptimalitySystem}
in the case of the $L^2$ regularization.

Moreover, we can replace the mass matrix
${A}_{\varrho h} = \overline{M}_{\varrho h}$ by the lumped mass matrix 
$\overline{D}_{\varrho h}:=\text{lump}(\overline{M}_{\varrho h})$ in 
the discrete optimality system
\eqref{Sec:Introduction:Eqn:AbstractDiscreteReducedOptimalitySystem}
without affecting the discretization error as was shown in
\cite{LLSY:LangerLoescherSteinbachYang:2023arXiv:2304.14664}
in the case of elliptic OCPs for $\varrho = h^4$.
Then the matrix-by-vector multiplication 
$(B^\top_h \overline{D}_{\varrho h}^{-1} B_h + M_h) * \mathbf{y}_h^n$ is fast.
Now the SC-PCG with the Schur complement preconditioner 
$\widehat{S}_{\varrho h} = D_h := \text{lump}(M_h)$ 
is an asymptotically optimal solver for the SC system
\eqref{Sec:Solvers:Eqn:Sy=yd}. This mass-lumped SC-PCG converges in the 
$B^\top_h \overline{D}_{\varrho h}^{-1} B_h + M_h$ energy norm that 
is equivalent to the $L^2(Q)$ norm on the FE space due to
Theorem~\ref{Theorem:SpectralEquivalenceTheoremL2Regularization1}. 
This is exactly the norm in which we want to approximate the target $y_d$.

\subsection{Energy Regularization}\label{Sec:Solvers:SubSec:Energy}
For the energy regularization, the regularization matrix $A_{\varrho h}$
is nothing but the SPD $m_h \times m_h$ diffusion stiffness matrix
$\overline{K}_{\varrho h}$, the coefficients of which are defined by
\begin{equation}\label{Sec:Solvers:Eqn:overlineKh[j,i]}
  \overline{K}_{\varrho h}[j,i] =
  \langle \varrho^{-1} \nabla \psi_{i}, \nabla \psi_{j} \rangle_{L^2(Q)}
  \; \forall \, j,i = 1,\ldots,m_h.
\end{equation}
Here we again permit variable regularization of the form
\begin{equation}\label{Sec:Solvers:Eqn:varroh=h2}
  \varrho(x,t) = h_\tau^2,\; \forall \, (x,t) \in \tau,\;
  \forall \tau \in \mathcal{T}_h,
\end{equation}
which we implemented in all numerical experiments when an adaptive 
mesh refinement is used. It is clear that \eqref{Sec:Solvers:Eqn:varroh=h2} 
turns to $\varrho = h^2$ in the case of uniform mesh refinement 
for which we have made the error analysis in 
Subsection~\ref{SubSec:SpaceTimeFiniteElementDiscretization:Energy}.

Now the Schur complement  $B^\top_h \overline{K}_{\varrho h}^{-1} B_h + M_h$
is again spectrally equivalent to $D_h$ as the following spectral
equivalence theorem shows.

\begin{theorem}\label{Theorem:SpectralEquivalenceTheoremEnergyRegularization1}
  Let us consider the optimally balanced, mesh-dependent, variable
  regularization \eqref{Sec:Solvers:Eqn:varroh=h2},
  and let $M_h$ as defined in \eqref{Sec:Solvers:Eqn:Mh[l,k]} with
  $D_h = \text{\rm lump}(M_h)$. Then the spectral equivalence inequalities
\begin{equation}
  \label{Sec:Solvers:Eqn:EnergyRegSpectralEquivalenceInequalities}
  (d+2)^{-1} \, D_h \, \le \,  M_h \, \le \,
  B^\top_h A_{\varrho h}^{-1} B_h + M_h 
  \, \le \, (c_\text{\tiny inv}^2+1) \, M_h 
  \, \le \, (c_\text{\tiny inv}^2+1) \, D_h
\end{equation}
hold for $A_{\varrho h} = \overline{K}_{\varrho h}$.
The constant $c_\text{\tiny inv}$ originates from the inverse inequalities 
\eqref{Sec:FEM:LocalInverseInequality}.
\end{theorem}

\begin{proof}
  Using the spectral equivalence inequalities
  \eqref{Sec:Solvers:Eqn:overlineMhoverlineDh}, Cauchy's inequalities,
  and the inverse inequalities \eqref{Sec:FEM:LocalInverseInequality},
  we get the estimates
  \begin{eqnarray*}
    && \hspace*{-5mm}
       (B^\top_h\overline{K}_{\varrho h}^{-1} B_h \mathbf{y}_h,\mathbf{y}_h) 
        =\sup_{\mathbf{q}_h \in \mathbb{R}^{m_h}}
        \frac{(B_h \mathbf{y}_h,\mathbf{q}_h)^2}
        {(\overline{K}_{\varrho h}\mathbf{q}_h,\mathbf{q}_h)}\\
    && = \sup_{q_h\in P_h }
        \frac{\Big[ - \langle \partial_t y_h , \partial_t q_h
        \rangle_{L^2(Q)} + \langle \nabla_x y_h, \nabla_x q_h
        \rangle_{L^2(Q)} \Big]^2}
        {\langle \varrho^{-1} \nabla q_h , \nabla q_h \rangle_{L^2(Q)}} \\
    && = \sup_{q_h \in P_h}
        \frac{\Big[ \langle \varrho^{\frac{1}{2}} \widehat{\nabla} y_h,
        \varrho^{-\frac{1}{2}} \nabla q_h\rangle_{L^2(Q)}\Big]^2}
        {\langle \varrho^{-1} \nabla q_h , \nabla q_h\rangle_{L^2(Q)}} 
       \le  \sup_{q_h \in P_h}
          \frac{\|\varrho^{\frac{1}{2}} \widehat{\nabla} y_h\|_{L^2(Q)}^2
          \|\varrho^{-\frac{1}{2}} \nabla q_h\|_{L^2(Q)}^2}
          {\|\varrho^{-\frac{1}{2}} \nabla q_h\|_{L^2(Q)}^2}\\
    && = \, \|\varrho^{\frac{1}{2}} {\nabla} y_h\|_{L^2(Q)}^2
       = \sum\limits_{\tau \in {\mathcal{T}}_h} h_\tau^2 \,
       \| \nabla y_h \|^2_{L^2(\tau)}
       \le
       c_\text{\tiny inv}^2 \, \| y_h \|^2_{L^2(Q)}
       =
       c_\text{\tiny inv}^2 \, (M_h \mathbf{y}_h,\mathbf{y}_h), 
\end{eqnarray*}
for all $\mathbf{y}_h \in \mathbb{R}^{n_h}$,
$\mathbf{y}_h \leftrightarrow y_h \in Y_h=
S_h^1(\mathcal{T}_h)\cap H^{1.1}_{0;0,}(Q)$,
where $\widehat{\nabla} = (\nabla_x, - \partial_t)^T$, and
$\nabla = (\nabla_x, \partial_t)^T$ is the space-time gradient.
Combining theses estimates with the spectral equivalence inequalities
\eqref{Sec:Solvers:Eqn:MhDh} completes the proof of the theorem.
\end{proof}

\noindent
We again note that the constant choice $\varrho = h^2$ leads to the same 
spectral equivalence inequalities
\eqref{Sec:Solvers:Eqn:EnergyRegSpectralEquivalenceInequalities}
as in the case of variable regularization since the constant regularization 
is a special case of variable regularisation when we set $h_\tau = h$ for
all $\tau \in \mathcal{T}_h$.

Let us again solve the SID system
\eqref{Sec:Introduction:Eqn:AbstractDiscreteReducedOptimalitySystem}
by means of the PB-PCG. Thanks to
Theorem~\ref{Theorem:SpectralEquivalenceTheoremEnergyRegularization1},
we can use $\widehat{S}_{\varrho h} = D_h := \text{lump}(M_h)$ as very
efficient SC preconditioner with the spectral equivalence constants 
$\underline{c}_S = 1/(d+2)$ and $\overline{c}_S = (c_\text{\tiny inv}^2+1)$.
The construction of a properly scaled preconditioner $\widehat{A}_{\varrho h}$ 
for ${A}_{\varrho h} = \overline{K}_{\varrho h}$ is more involved.
In our numerical experiments, we will choose a properly scaled SPD
algebraic multigrid (AMG) preconditioners that can be represented in the form 
\begin{equation}\label{Sec:Solvers:Eqn:ScaledAMGpreconditioner} 
  \widehat{A}_{\varrho h} = \widehat{K}_{\varrho h} 
  := \delta (1-\eta^i) K_{\varrho h} (I_h - E_{\varrho h}^i)^{-1}
\end{equation}
with a positive scaling parameter $\delta < 1$, where $E_{\varrho h}$ denotes
the corresponding AMG error propagation (iteration) matrix,
and $\eta \in [0,1)$ is a bound for the convergence rate with respect to 
the ${K}_{\varrho h}$ energy norm, i.e.
$\|E_{\varrho h}\|_{K_{\varrho h}} \le \eta < 1$. We can choose the components
of the multigrid preconditioner 
${K}_{\varrho h} (I_h - E_{\varrho h}^i)^{-1}$ in such a way that it is SPD,
$E_{\varrho h}$ is self-adjoint and not negative in the ${K}_{\varrho h}$
energy inner product, and 
\begin{equation}
\label{Sec:Solvers:Eqn:AMGpreconditioner:SpectralEquivalenceInequalities1} 
 (1-\eta^i)   K_{\varrho h} (I_h - E_{\varrho h}^i)^{-1} \le K_{\varrho h} \le
 K_{\varrho h} (I_h - E_{\varrho h}^i)^{-1};
\end{equation}
see \cite{LLST:JungLangerMeyerQueckSchneider:1989a} for details. The
spectral equivalence inequalities
\eqref{Sec:Solvers:Eqn:AMGpreconditioner:SpectralEquivalenceInequalities1} 
immediately yield \eqref{Sec:Solvers:Eqn:SpectralEquivalenceInequalities:A}
with $\overline{c}_A = 1/(\delta (1 - \eta^i))$.
Therefore, due to this result and the results of
Theorem~\ref{Theorem:SpectralEquivalenceTheoremEnergyRegularization1},
the PB-PCG with \eqref{Sec:Solvers:Eqn:ScaledAMGpreconditioner} 
and again $\widehat{S}_{\varrho h} = D_h := \text{lump}(M_h)$
is an asymptotically optimal solver for the SID system
\eqref{Sec:Introduction:Eqn:AbstractDiscreteReducedOptimalitySystem}
in the case of the energy regularization too, where $i=1$ is a good choice.
We note that, in the case of constant $\varrho = h^2$, we 
have $\widehat{K}_{\varrho h} = \varrho^{-1} K_h (I_h - E_h^i)^{-1}$,
where  $K_h = K_{1 h}$ and $E_h = E_{1 h}$.

In order to solve the corresponding Schur complement system
\eqref{Sec:Solvers:Eqn:Sy=yd} efficiently by means of PCG, we replace
$A_{\varrho h}^{-1}$ by an iterative approximation,
e.g. produced by AMG as in our numerical experiments, 
i.e., instead of  the exact Schur complement system, we solve the
inexact Schur complement system
\begin{equation}\label{Sec:Solvers:Eqn:InexactSy=yd}
  (B_h^\top (I_h - E_{\varrho h}^{i}) {K}_{\varrho h}^{-1}B_h+M_h)
  \mathbf{\tilde y}_h = \mathbf{y}_{dh}
\end{equation}
where we want to choose $i$ such that
$\|{\tilde y}_h - y_d\|_{L^2(Q)} = {\mathcal{O}}(\|y_h - y_d\|_{L^2(Q)}) =
{\mathcal{O}}(h^s)$. It is obviously sufficient to show that
$\|{\tilde y}_h - y_h\|_{L^2(Q)} = {\mathcal{O}}(h^s)$.

\begin{lemma}
\label{Sec:Solvers:Lemma:ErrorEstimate}
Let us choose the optimally balanced regularization $\varrho$ as given
by \eqref{Sec:Solvers:Eqn:varroh=h2},
and let $\|E_{\varrho h}\|_{K_{\varrho h}} \le \eta$ with some $h$-independent
rate $\eta \in (0,1)$. Then the estimates
\begin{equation}
  \|\mathbf{\tilde y}_h - \mathbf{y}_h\|_{M_h} =
  \|{\tilde y}_h - y_h\|_{L^2(Q)} 
  \le  c_\text{\tiny inv}^2 \eta^i \, \|\mathbf{\tilde y}_h\|_{M_h}
  \le  c_\text{\tiny inv}^2 \eta^i \, \|y_d\|_{L^2(Q)}
\end{equation}
hold.
\end{lemma}
\begin{proof}
Substracting the exact SC system \eqref{Sec:Solvers:Eqn:Sy=yd} from the 
inexact SC system \eqref{Sec:Solvers:Eqn:InexactSy=yd},
multiplying  this difference by the error $\mathbf{\tilde y}_h - \mathbf{y}_h$,
and using \eqref{Sec:Solvers:Eqn:EnergyRegSpectralEquivalenceInequalities},
we arrive at the estimates
\begin{eqnarray*}
  \|\mathbf{\tilde y}_h - \mathbf{y}_h\|_{M_h}^2
  &\le& ((B_h^\top K_{\varrho h}^{-1}B_h + M_h)
        (\mathbf{\tilde y}_h - \mathbf{y}_h),
        \mathbf{\tilde y}_h - \mathbf{y}_h)  \\
  &=& (K_{\varrho h}E_{\varrho h}^{i}K_{\varrho h}^{-1}B_h\mathbf{\tilde y}_h,
      K_{\varrho h}^{-1}B_h(\mathbf{\tilde y}_h - \mathbf{y}_h))\\
  &=& (E_{\varrho h}^{i} \mathbf{x}_h,\mathbf{z}_h)_{K_{\varrho h}}
      \le \|E_{\varrho h}^{i} \mathbf{x}_h\|_{K_{\varrho h}}
      \|\mathbf{z}_h\|_{K_{\varrho h}}\\
  &\le& \|E_{\varrho h}\|_{K_{\varrho h}}^{i} \|\mathbf{x}_h\|_{K_{\varrho h}}
        \|\mathbf{z}_h\|_{K_{\varrho h}}\\
  &\le& \eta^{i} (B_h^\top K_{\varrho h}^{-1}B_h\mathbf{\tilde y}_h,
        \mathbf{\tilde y}_h)^{1/2}
        (B_h^\top K_{\varrho h}^{-1}B_h (\mathbf{\tilde y}_h - \mathbf{y}_h),
        \mathbf{\tilde y}_h - \mathbf{y}_h)^{1/2}\\
  &\le& \eta^{i}c_\text{\tiny inv}^2 \,
        \|\mathbf{\tilde y}_h\|_{M_h}
        \|\mathbf{\tilde y}_h - \mathbf{y}_h\|_{M_h}
\end{eqnarray*}
where we used the setting
$\mathbf{x}_h =K_{\varrho h}^{-1}B_h\mathbf{\tilde y}_h$ 
and $\mathbf{z}_h =K_{\varrho h}^{-1}B_h(\mathbf{\tilde y}_h - \mathbf{y}_h)$ 
to simplify long notations. From the inexact SC system
\eqref{Sec:Solvers:Eqn:InexactSy=yd}, we can derive the estimate
$\|\mathbf{\tilde y}_h\|_{M_h} \le \|y_d\|_{L^2(Q)}$
that completes the proof of the lemma.
\end{proof}

\begin{lemma}\label{Sec:Solvers:Lemma:SpectralEstimate} 
  Let $\|E_{\varrho h}\|_{K_{\varrho h}} \le \eta$ with some $h$-independent
  rate $\eta \in (0,1)$.
  Then the following spectral equivalence inequalities are valid:
  \begin{equation}
    \label{Sec:Solvers:Eqn:AMGpreconditioner:SpectralEquivalenceInequalities2} 
    0 \le (1-\eta^i)\, B_h^\top K_{\varrho h}^{-1}B_h \le
    B_h^\top (I_h - E_{\varrho h}^{i}) {K}_{\varrho h}^{-1}B_h \le
    B_h^\top K_{\varrho h}^{-1}B_h.
  \end{equation}
\end{lemma}
  
\begin{proof}
  The spectral equivalence inequalities
  \eqref{Sec:Solvers:Eqn:AMGpreconditioner:SpectralEquivalenceInequalities2}
  now follow from the spectral equivalence inequalities
  \eqref{Sec:Solvers:Eqn:AMGpreconditioner:SpectralEquivalenceInequalities1}. 
\end{proof}

\noindent
Lemma~\ref{Sec:Solvers:Lemma:SpectralEstimate} immediately yields that the
inexact SC $B_h^\top (I_h - E_{\varrho h}^{i}) {K}_{\varrho h}^{-1}B_h + M_h$
satisfies the same spectral equivalence inequalities
\eqref{Sec:Solvers:Eqn:EnergyRegSpectralEquivalenceInequalities}
like the exact SC $B_h^\top - E_{\varrho h}^{i}) {K}_{\varrho h}^{-1}B_h+ M_h$.
Thus, the inexact SC system \eqref{Sec:Solvers:Eqn:InexactSy=yd} can be
solved by means of the $D_h := \text{lump}(M_h)$ preconditioned PCG
requiring ${\mathcal{O}}(\ln(\varepsilon^{-1}))$
to reduce the initial error by some given factor $\varepsilon \in (0,1)$ 
in the $M_h$ energy norm that is equivalent to the energy norm defined 
by the inexact SC. Lemma~\ref{Sec:Solvers:Lemma:SpectralEstimate} states
that the discretization error is asymptotically not affected by the
inner iterations provided that the number $i$ of inner iterations is fixed 
to $i= \ln(h^{-1})$. If the inner iteration has an $h$-independent rate
$\eta \in (0,1)$ and  asymptotically optimal arithmetical complexity like
AMG, and if we choose $\varepsilon = {\mathcal{O}}(h^s)$ 
then the PCG need $k=\mathcal{O}(\ln(h^{-1}))$ iterations  and
${\mathcal{O}}((\ln(h^{-1}))^2 h^{-d})$ arithmetical operations
to produce an approximation
$\mathbf{\tilde y}_h^k \leftrightarrow {\tilde y}_h^k$ which differs from
$y_d$ in the order ${\mathcal{O}}(h^s)$ of the discretization error
with respect to the $L^2(Q)$ norm. It is clear that the arithmetical
complexity can be reduced to ${\mathcal{O}}((\ln(h^{-1})) h^{-d})$
by using a nested iteration setting on a sequence of finer and finer meshes 
that can be generated adaptively; see
Tables~\ref{Table:wave_ex3_energy_2d_nested} for numerical results using
nested iterations on uniformly and adaptively refined meshes
in $2$ space dimension.
\section{Numerical Results}
\label{Sec:NumericalResults}
We perform numerical experiments for three different benchmark examples 
with targets $y_d$ possessing different regularity:

\begin{itemize}
\item {\bf Example~1:} {\it Smooth Target}, where the target function
  is defined by
  \begin{equation}
    \label{Sec:NumericalResults:Eqn:Example1:SmoothTarget}
    y_d(x,t) = t^2 \prod_{i=1}^d{\sin(\pi x_i)} \in
    C^\infty(\overline{Q})\cap H^{1,1}_{0;0,}(Q) \subset Y.
  \end{equation}
\item {\bf Example~2:} {\it  Continuous Target} that is given by
  the continuous piecewise multi-linear (tri-linear for $d=2$) 
  target  function
  \begin{equation}
    \label{Sec:NumericalResults:Eqn:Example2:ContinuousTarget}
    y_{d}(x,t) =\phi(t)\prod_{i=1}^{d}\phi(x_i)\in
    H_0^{3/2-\varepsilon}(Q),\; \varepsilon>0,
  \end{equation}
  where
  \begin{equation*}
    \phi(s)=
    \begin{cases}
      1, & \textup{ if } s=0.5,\\
      0, & \textup{ if } s\notin [0.25, 0.75],\\
      \textup{linear}, & \textup{else}.
    \end{cases}
  \end{equation*}
  We note that this target function belongs to the state space $Y$ too.
\item {\bf Example~3:} {\it  Discontinuous Target} that is defined by
  the discontinuous function
  \begin{equation}\label{Sec:NumericalResults:Eqn:Example3:DiscontinuousTarget}
    y_d(x,t)=
    \begin{cases}
    1, \; \textup{ if } (x,t)\in (0.25, 0.75)^{d+1}\subset Q, \\
    0, \; \textup{ else},
  \end{cases}
  \in H^{1/2-\varepsilon}(Q), \; \varepsilon>0,
\end{equation}
which does not belong to the the state space $Y$.
\end{itemize}

\noindent
For $d=2$, the space-time domain is given by
$Q=\Omega\times (0, T) \subset \mathbb{R}^3$ with
$\Omega=(0,1)^2$ and $T=1$. The domain $Q$ is uniformly decomposed into
$384$ tetrahedrons with $5$ equidistant vertices in each direction. 
This yields an initial coarse mesh with $125$ vertices in
total and the mesh size $h = 2^{-(l+1)} = 0.25$ at the level $l = 1$.
The uniform refinement of the tetrahedrons
is based on Bey's algorithm  as described in \cite{LLSY:JB95}. 
This uniform refinement results in $(2^{l+1} +1)^{d+1=3}$ vertices, 
and the mesh size $h=2^{-(l+1)}$ that yields 
$\varrho = h^4 = 2^{-4(l+1)}$ ($L^2$-regularization) and
$\varrho = h^2 = 2^{-2(l+1)}$ (energy regularization),
where $l$ is running from $1$ (coarsest mesh) to $L=6$ (finest mesh).

In the case of three space dimensions $d=3$, we consider the space-time
domain $Q=\Omega\times(0, T) \subset \mathbb{R}^4$ with $\Omega=(0,1)^3$
and $T=1$. The initial decomposition of $Q$ contains $178$ vertices and
$960$ pentatops. The refinement of the pentatops uses the bisection method
proposed in \cite{LLSY:RS08}. The mesh size is
$h\approx (\textup{\#Vertices})^{-1/4}=2.74$e$-1$
with $\textup{\#Vertices}=178$ on the starting level $l=1$,
  and, on the finest level $L=17$, the mesh size is
  $h\approx (\textup{\#Vertices})^{-1/4}=2.22$e$-2$ with
  $\textup{\#Vertices}=4,144,513$.

Besides the uniform mesh refinement described above for $d=2$ and $d=3$,
we also provide numerical experiments for an adaptive mesh refinement 
based on the computable error representation
\begin{equation}\label{Sec:NumericalResults:Eqn:ErrorRepresentation}
\|y_{\varrho h} - y_d \|_{L^2(Q)}^2 
= \sum_{\tau \in \mathcal{T}_h} \|y_{\varrho h} - y_d \|_{L^2(\tau)}^2,
\end{equation}
and the  maximum marking strategy, i.e., an element $\tau \in \mathcal{T}_h$
will be refined if 
$\| y_{\varrho h} - y_d\|_{L^2(\tau)} \ge
\theta \, \max_{\tau \in \mathcal{T}_h} \| y_{\varrho h} - y_d\|_{L^2(\tau)}$,
where we have chosen $\theta = 0.5$;
cf.~\cite{LLSY:BabuskaVogelius:1984NumerMath}.

In our numerical experiments, we compare the performance of the 
SC-PCG and the BP-PCG, presented in the preceding section,  
with the standard preconditioned GMRES (PGMRES) for solving 
the equivalent non-symmetric and positive definite system 
\begin{equation*}
  \begin{bmatrix}
    A_{\varrho h} & B_h\\
    -B_h^\top & M_h
  \end{bmatrix}
  \begin{bmatrix}
   \mathbf{p}_h\\ 
   \mathbf{y}_h\\ 
  \end{bmatrix}
  =
  \begin{bmatrix}
    \mathbf{0}_h\\ 
    \mathbf{y}_{dh}\\ 
  \end{bmatrix},
\end{equation*}
with the block-diagonal matrix
\begin{equation*}
  \label{Sec:NumericalResults:GMRESPreconditioner}
  \begin{bmatrix}
    \widehat{A}_{\varrho h} & 0\\
    0 & \textup{ lump } (M_h)
  \end{bmatrix},
\end{equation*}
as preconditioner, where 
$\widehat{A}_{\varrho h}=  \text{lump}(\overline{M}_{\varrho h})$
for the $L^2$ regularization, and, in the case of the energy regularization,
$\widehat{A}_{\varrho h}=A_{\varrho h}(I_h-E_{\varrho h}^j)^{-1}$  
is defined by  the classical Ruge--St\"{u}ben algebraic multigrid (AMG)
preconditioner \cite{LLSY:RugeStuben} with 
$j=2$ AMG V-cycles and $2$ Gauss--Seidel pre-smoothing
and post-smoothing steps at each level. 
The SC-PCG is always preconditioned by $D_h = \text{lump}(M_h)$ 
as discussed in Section~\ref{Sec:Solvers}.
SC-CG means that we run the CG without any preconditioning.
We always solve the SC with $\widehat{A}_{\varrho h} = \text{lump}(M_h)$.

We stop the iterations as soon as the initial error is reduced  by a factor
of $10^{11}$ in the norm that is defined by the square root of scalar
product between the preconditioned residual and the residual. For instance,
in the case of the BP-PCG iteration, this norm is nothing but the  
${\mathcal K}_{h} \widehat{\mathcal K}_{h}^{-1}{\mathcal K}_{h}$ energy norm,
i.e. $\|\cdot\|_{{\mathcal K}_{h} \widehat{\mathcal K}_{h}^{-1}{\mathcal K}_{h}}
= (\widehat{\mathcal K}_{h}^{-1}{\mathcal K}_{h}\,\cdot\,,
{\mathcal K}_{h}\,\cdot\,)^{1/2}$. The initial guess is always the zero 
vector with exception of the nested iteration 
where we interpolate the initial guess from the coarser mesh.

In the following two subsection, $\|\cdot\|$ always denotes the 
$L^2$ norm $\|\cdot\|_{L^2(Q)}$.

\subsection{$L^2$-Regularization and Mass Lumping}
\label{SubSec:L2Regularization}
In the BP-PCG, we use \eqref{Sec:Solvers:Eqn:L2RegPrePBCG} for
$\widehat{A}_{\varrho h}$ and $\widehat{S}_{\varrho h}$, where we have set
$\delta=0.98$. Further, for the inverse operation of $M_{\varrho h}^{-1}$ applied
to a given vector $\mathbf{v}$ inside each SC-PCG/CG iteration, we apply the
Ruge--St\"{u}ben AMG \cite{LLSY:RugeStuben} method to solve
$M_{\varrho h}\mathbf{w}=\mathbf{v}$ until the
relative preconditioned residual error is reduced by a factor of $10^{12}$.

We first consider the case of the uniform refinement of the space-time 
cylinder $Q \subset \mathbb{R}^3$($d=2$) across $6$ levels of refinement. 
Tables~\ref{Table:wave_ex1_l2_2d}, \ref{Table:wave_ex2_l2_2d}, and
\ref{Table:wave_ex3_l2_2d} provide the numerical results for 
{\bf Examples~1}, {\bf 2}, and {\bf 3}, respectively. In the second column
of the tables, we observe that the discretization error $\|y_{\varrho h}-y_d\|$ 
behaves like expected from the theoretical results presented in
Subsection~\ref{SubSec:SpaceTimeFiniteElementDiscretization:L2}.
More precisely, the experimental order of convergence (EOC) 
corresponds to the regularity of the target. The third column of the tables
displays the  iteration numbers needed 
to reduce the initial error by the factor $10^{-11}$
for the PGMRES, the SC-PCG/CG and the PB-PCG solvers. 
Here we see that the robustness of the proposed preconditioners is
confirmed by almost mesh-independent iteration numbers for all solvers.

\begin{table}[ht]
  {\small
  \begin{tabular}{|cr|cc|ccc|}
    \hline
    \multicolumn{2}{|c|}{}&\multicolumn{2}{c|}{Convergence}&\multicolumn{3}{c|}{Solvers (Number of Iterations)}\\
    \hline
    Level&\#Vertices
    & 
    $\|y_{\varrho h}-y_d\|$ & EOC & PGMRES &SC-PCG/CG& PB-PCG \\
    \hline
    $1$&$125$&$1.166$e$-1$&$-$ &$67$&$60/65$&$97$ \\
    $2$&$729$&$2.688$e$-2$&$2.12$&$325$&$239/317$&$290$\\
    $3$&$4,913$&$5.564$e$-3$&$2.27$ &$403$&$256/377$&$301$\\
    $4$&$35,937$&$1.105$e$-3$&$2.33$&$400$&$250/389$&$293$\\
    $5$&$274,625$&$2.138$e$-4$&$2.37$&$393$&$241/395$&$284$\\
    $6$&$2,146,689$&$4.172$e$-5$&$2.36$&$381$&$235/398$&$276$\\ 
    \hline
  \end{tabular}
  }
  \caption{{\bf Example~1} 
    (Smooth Target \ref{Sec:NumericalResults:Eqn:Example1:SmoothTarget},
    $d=2$, $L^2$ regularization): 
  Convergence in the $L^2(Q)$-norm,
  and number of iterations 
  for attaining the relative accuracy $10^{-11}$.}
  \label{Table:wave_ex1_l2_2d}
\end{table}

\begin{table}[ht]
  {\small
  \begin{tabular}{|cr|cc|ccc|}
    \hline
    \multicolumn{2}{|c|}{}&\multicolumn{2}{c|}{Convergence}&\multicolumn{3}{c|}{Solvers (Number of Iterations)}\\
    \hline
    Level & \#Vertices
    &
    $\|y_{\varrho h}-y_d\|
    $  & EOC & PGMRES &SC-PCG/CG  & PB-PCG \\
    \hline
    $1$&$125$&$5.668$e$-2$&$-$ &$66$    &$59/61$ &$98$ \\
    $2$&$729$&$4.069$e$-2$&$0.48$&$350$ &$248/324$ &$305$\\
    $3$&$4,913$&$1.454$e$-2$&$1.49$&$436$ &$267/361$ &$313$\\
    $4$&$35,937$&$4.808$e$-3$&$1.60$&$429$ &$257/345$ &$302$\\
    $5$&$274,625$&$1.727$e$-3$&$1.48$&$415$ &$244/321$ &$287$\\
    $6$&$2,146,689$&$6.121$e$-4$&$1.50$&$399$ &$236/291$ &$277$\\
    \hline
  \end{tabular}
  }
  \caption{{\bf Example~2} 
  (Continuous Target \eqref{Sec:NumericalResults:Eqn:Example2:ContinuousTarget}, $d=2$, $L^2$ regularization): 
  Convergence in the $L^2(Q)$-norm, and number of iterations for attaining the relative accuracy $10^{-11}$.}
  \label{Table:wave_ex2_l2_2d}
\end{table}

\begin{table}[ht]
  {\small
  \begin{tabular}{|cr|cc|ccc|}
    \hline
    \multicolumn{2}{|c|}{}&\multicolumn{2}{c|}{Convergence}&\multicolumn{3}{c|}{Solvers (Number of Iterations)}\\
    \hline
    Level&\#Vertices
    &
    $\|y_{\varrho h}-y_d\|$  
    & EOC & PGMRES &SC-PCG/CG & PB-PCG \\
    \hline
    $1$&$125$&$2.668$e$-1$&$-$ & $68$ &$57/65$&$105$ \\
    $2$&$729$&$2.085$e$-1$&$0.36$&$347$ &$246/316$&$299$\\
    $3$&$4,913$&$1.562$e$-1$&$0.42$ &$432$ &$270/367$&$315$\\
    $4$&$35,937$&$1.128$e$-1$&$0.47$ &$442$&$268/364$&$312$\\
    $5$&$274,625$&$8.064$e$-2$&$0.48$ &$445$&$263/344$&$308$\\
    $6$&$2,146,689$&$5.734$e$-2$&$0.49$ &$442$&$259/316$&$302$\\
    \hline
  \end{tabular}
  }
  \caption{{\bf Example~3} 
    (Discontinuous Target \eqref{Sec:NumericalResults:Eqn:Example3:DiscontinuousTarget}, $d=2$, $L^2$ regularization): 
    Convergence in the $L^2(Q)$-norm, and number of iterations for attaining the relative accuracy $10^{-11}$.}
  \label{Table:wave_ex3_l2_2d}
\end{table}

For the $L^2$ regularization, we also consider the mass-lumped Schur
complement system $(B_h^\top \overline{D}_{\varrho h}^{-1}B_h + M_h)
\mathbf{y}_h = \mathbf{y}_{dh}$, which is solved by means of the PCG method
preconditioned by the lumped mass matrix $D_h$. The numerical behavior of
the $L^2$ error between the space-time finite element state approximations
$y_{\varrho h}$ and  the three targets $y_d$ as well as the number of
mass-lumped SC-PCG iterations are shown in
Tables~\ref{Table:wave_ex1ex2ex3_l2_inexactschur_2d} and
\ref{Table:wave_ex1ex2ex3_l2_inexactschur_3d} for two and three space
dimensions, respectively. We observe that the
convergence rates depend on the regularity of the targets as expected;
see also the convergence history illustrated in the two plots of
Figure~\ref{fig:convergence_l2_2d_3d} corresponding to $d=2$ and $d=3$.
Moreover, if we compare the errors $\|y_{\varrho h}-y_d\|$ of 
Table~\ref{Table:wave_ex1ex2ex3_l2_inexactschur_2d} with the corresponding 
errors in Tables~\ref{Table:wave_ex1_l2_2d}, \ref{Table:wave_ex2_l2_2d},
and  \ref{Table:wave_ex3_l2_2d},
then we see that the mass-lumping in the Schur complement does not affect 
the accuracy of the approximations at all. Furthermore, the mass-lumped
SC-PCG solver is robust as the almost constant iteration numbers show. 

\begin{table}[ht]
  {\small
  \begin{tabular}{|cr|cc|cc|cc|}
    \hline
    \multicolumn{2}{|c|}{}&\multicolumn{2}{c|}{Target
      (\ref{Sec:NumericalResults:Eqn:Example1:SmoothTarget})}&\multicolumn{2}{c|}{Target
      (\ref{Sec:NumericalResults:Eqn:Example2:ContinuousTarget})}&\multicolumn{2}{c|}{Target
    (\ref{Sec:NumericalResults:Eqn:Example3:DiscontinuousTarget})}\\
    \hline
    Level&\#Vertices&$\|y_{\varrho h}-y_d\|$&PCG&$\|y_{\varrho h}-y_d\|$&PCG&$\|y_{\varrho h}-y_d\|$&PCG\\
    \hline
    $1$&$125$&$1.022$e$-1$&$47$&$5.457$e$-2$&$48$&$2.599$e$-1$&$46$\\
    $2$&$729$&$2.540$e$-2$&$135$&$3.780$e$-2$&$141$&$2.034$e$-1$&$138$\\
    $3$&$4,913$&$5.374$e$-3$&$142$&$1.343$e$-2$&$149$&$1.520$e$-1$&$150$\\
    $4$&$35,937$&$1.071$e$-3$&$140$&$4.498$e$-3$&$144$&$1.098$e$-1$&$150$\\
    $5$&$274,625$&$2.066$e$-4$&$137$&$1.622$e$-3$&$139$&$7.857$e$-2$&$149$\\
    $6$&$2,146,689$&$3.998$e$-5$&$134$&$5.760$e$-4$&$136$&$5.588$e$-2$&$150$\\
    \hline
  \end{tabular}
  }
  \caption{PCG for the  mass-lumped SC system ($d=2$):  $L^2$ error and number of 
  mass-lumped SC-PCG iterations for attaining the relative accuracy $10^{-11}$.
  }
  \label{Table:wave_ex1ex2ex3_l2_inexactschur_2d}
\end{table}

\begin{table}[ht]
  {\small
  \begin{tabular}{|cr|cc|cc|cc|}
    \hline
    \multicolumn{2}{|c|}{}&\multicolumn{2}{c|}{Target
      (\ref{Sec:NumericalResults:Eqn:Example1:SmoothTarget})}&\multicolumn{2}{c|}{Target
      (\ref{Sec:NumericalResults:Eqn:Example2:ContinuousTarget})}&\multicolumn{2}{c|}{Target (\ref{Sec:NumericalResults:Eqn:Example3:DiscontinuousTarget})}\\
    \hline
    Level&\#Vertices&$\|y_{\varrho h}-y_d\|$& PCG &$\|y_{\varrho h}-y_d\|$& PCG &$\|y_{\varrho h}-y_d\|$& PCG\\
    \hline
    $1$&$178$&$1.034$e$-1$&$34$&$5.957$e$-2$&$37$&$2.501$e$-1$&$34$\\
    $2$&$235$&$1.028$e$-1$&$35$&$2.880$e$-2$&$35$&$2.593$e$-1$&$35$\\
    $3$&$315$&$9.660$e$-2$&$79$&$2.380$e$-2$&$78$&$1.830$e$-1$&$79$\\
    $4$&$715$&$7.193$e$-2$&$261$&$2.217$e$-2$&$261$&$1.968$e$-1$&$260$\\
    $5$&$1,493$&$5.076$e$-2$&$370$&$2.161$e$-2$&$369$&$1.734$e$-1$&$384$\\
    $6$&$2,185$&$4.359$e$-2$&$423$&$2.047$e$-2$&$422$&$1.767$e$-1$&$418$\\
    $7$&$3,465$&$3.579$e$-2$&$563$&$1.949$e$-2$&$573$&$1.630$e$-1$&$579$\\
    $8$&$9,225$&$2.060$e$-2$&$840$&$1.684$e$-2$&$861$&$1.554$e$-1$&$870$\\
    $9$&$19,057$&$1.261$e$-2$&$628$&$1.372$e$-2$&$640$&$1.390$e$-1$&$658$\\
    $10$&$26,593$&$1.036$e$-2$&$628$&$1.086$e$-2$&$643$&$1.330$e$-1$&$653$\\
    $11$&$47,073$&$7.926$e$-3$&$689$&$9.046$e$-3$&$703$&$1.257$e$-1$&$713$\\
    $12$&$134,113$&$4.420$e$-3$&$1050$&$5.896$e$-3$&$1080$&$1.144$e$-1$&$1121$\\
    $13$&$273,281$&$2.779$e$-3$&$792$&$4.220$e$-3$&$785$&$1.031$e$-1$&$821$\\
    $14$&$372,481$&$2.246$e$-3$&$693$&$3.732$e$-3$&$711$&$9.833$e$-2$&$734$\\
    $15$&$700,161$&$1.722$e$-3$&$724$&$3.212$e$-3$&$734$&$9.240$e$-2$&$771$\\
    $16$&$2,051,841$&$9.570$e$-4$&$1074$&$2.403$e$-3$&$1074$&$8.232$e$-2$&$1146$\\
    $17$&$4,144,513$&$5.951$e$-4$&$808$&$1.685$e$-3$&$775$&$7.490$e$-2$&$843$\\    
    \hline
  \end{tabular}
  }
    \caption{PCG for the  mass-lumped SC system ($d=3$):  $L^2$ error and number of 
  mass-lumped SC-PCG iterations for attaining the relative accuracy $10^{-11}$.
  }
  \label{Table:wave_ex1ex2ex3_l2_inexactschur_3d}
\end{table}

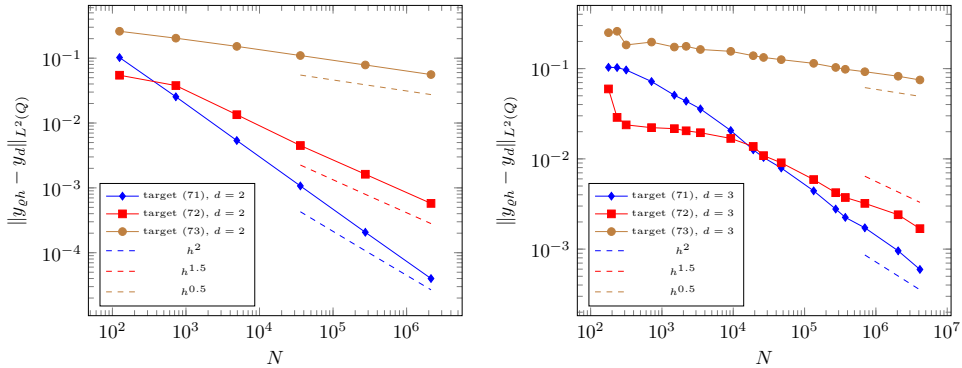
\begin{figure}[htpb!]
  \begin{tikzpicture}[scale=0.725,transform shape]
    \begin{axis}[
        xmode = log,
        ymode = log,
        xlabel=$N$,
        ylabel=$\| y_{\varrho h}- y_d \|_{L^2(Q)}$,
        legend pos=south west,
        legend style={font=\tiny}]
      
      \addplot [solid, mark=diamond*, color=blue] table [col sep=
        &, y=err, x=N]{2d_smooth.dat};
      \addlegendentry{target
        (\ref*{Sec:NumericalResults:Eqn:Example1:SmoothTarget}), $d=2$}

      \addplot  [solid, mark=square*, color=red] table [col sep=
        &, y=err, x=N]{2d_linear.dat};
      \addlegendentry{target
        (\ref*{Sec:NumericalResults:Eqn:Example2:ContinuousTarget}), $d=2$}
      
      \addplot [solid, mark=*, color=brown] table [col sep=
        &, y=err, x=N]{2d_discontinuous.dat};
      \addlegendentry{target
        (\ref*{Sec:NumericalResults:Eqn:Example3:DiscontinuousTarget}), $d=2$}
      
      \addplot [dashed, thin, color=blue] table [col sep=
        &, y=err, x=N]{2d_smooth_exact.dat};
      \addlegendentry{$h^{2}$}
      
      \addplot [dashed, thin, color=red] table [col sep=
        &, y=err, x=N]{2d_continuous_exact.dat};
      \addlegendentry{$h^{1.5}$}
      
      \addplot [dashed, thin, color=brown] table [col sep=
        &, y=err, x=N]{2d_discontinuous_exact.dat};
      \addlegendentry{$h^{0.5}$}
    \end{axis}
    
    \hspace{6.5cm}
    
    \begin{axis}[
        xmode = log,
        ymode = log,
        xlabel=$N$,
        ylabel=$\| y_{\varrho h}- y_d \|_{L^2(Q)}$,
        legend pos=south west,
        legend style={font=\tiny}]
      
      \addplot [solid, mark=diamond*, color=blue] table [col sep=
        &, y=err, x=N]{3d_smooth.dat};
      \addlegendentry{target (\ref*{Sec:NumericalResults:Eqn:Example1:SmoothTarget}), $d=3$}

      \addplot  [solid, mark=square*, color=red] table [col sep=
        &, y=err, x=N]{3d_linear.dat};
      \addlegendentry{target (\ref*{Sec:NumericalResults:Eqn:Example2:ContinuousTarget}), $d=3$}
      
      \addplot [solid, mark=*, color=brown] table [col sep=
        &, y=err, x=N]{3d_discontinuous.dat};
      \addlegendentry{target (\ref*{Sec:NumericalResults:Eqn:Example3:DiscontinuousTarget}), $d=3$}
      
      \addplot [dashed, thin, color=blue] table [col sep=
        &, y=err, x=N]{3d_smooth_exact.dat};
      \addlegendentry{$h^{2}$}
      
      \addplot [dashed, thin, color=red] table [col sep=
        &, y=err, x=N]{3d_continuous_exact.dat};
      \addlegendentry{$h^{1.5}$}
      
      \addplot [dashed, thin, color=brown] table [col sep=
        &, y=err, x=N]{3d_discontinuous_exact.dat};
      \addlegendentry{$h^{0.5}$}
      
    \end{axis}
  \end{tikzpicture}
  \caption{Convergence history for all targets when solving the mass-lumped SC
    system: $d=2$ (left) and $d=3$ (right).}
  \label{fig:convergence_l2_2d_3d}
\end{figure} 

In order to reduce the computational complexity even further, 
we may use the nested SC-PCG iteration with the preconditioner $D_h$ for
solving the mass-lumped Schur complement system 
on a sequence of uniformly or adaptively refined meshes.
We here only consider {\bf Example~3} with the discontinuous target 
\eqref{Sec:NumericalResults:Eqn:Example3:DiscontinuousTarget}.
At the coarsest level, we solve the mass-lumped SC system until the 
initial error is reduced by a factor of $10^6$.
We use the adaptive threshold $\alpha\left[N_{l}/N_{l-1}\right]^{\beta/3}$ 
to control the error at the refined levels  $l=2,3,\ldots$,
where $N_l$ is the number of degrees of freedom at the level $l$. 
In the numerical experiments, we set $\alpha=0.4$ for $d=2$ and
$\alpha=0.1$ for $d=3$, 
and use $\beta=0.5$ and $\beta=0.75$ for the uniform and adaptive refinement, 
respectively.
The performance of this nested mass-lumped SC-PCG iteration is documented 
in Tables~\ref{Table:wave_ex3_l2_2d_lumpedmass_nested} and
\ref{Table:wave_ex3_l2_3d_lumpedmass_nested}
for $d=2$ and $d=3$, respectively.
The adaptive refinement shows a much better convergence than the uniform one. 

In fact, it is straightforward to parallelize the mass-lumped SC-PCG solver for
this mass-lumped SC system; see the measured performance using $256$ cores
for {\bf Example~3} ($d=2$) in Table
\ref{Table:wave_ex3_l2_2d_parallel_lumpedmass_nested}. 
The parallel
solver is implemented using the open source MFEM (https://mfem.org/), and
tested on the high performance cluster RADON1
(https://www.oeaw.ac.at/ricam/hpc). We observe a very good
parallel efficiency.

\begin{table}[ht]
  {\small
  \begin{tabular}{|r|cc|c|r|c|c|c|}
    \hline
    \multicolumn{4}{|c|}{Uniform}&\multicolumn{3}{c|}{Adaptive}\\
    \hline
    \#Vertices &$\|\tilde{y}_{\varrho h}-y_d\|$&EOC&SC-PCG&\#Vertices &$\|\tilde{y}_{\varrho h}-y_d\|$&SC-PCG\\
    \hline
    $125$&$2.599$e$-1$&$-$&$37$ [0.002 s] &$125$&$2.599$e$-1$&$37$ [0.002 s]\\
    $729$&$2.105$e$-1$&$0.30$&$5$ [0.002 s] &$223$&$2.717$e$-1$&$1$ [0.0001 s]\\
    $4,913$&$1.502$e$-1$&$0.49$&$7$ [0.02 s] &$1,072$&$1.786$e$-1$&$12$ [0.006 s]\\
    $35,937$&$1.091$e$-1$&$0.46$&$7$ [0.15 s] &$4,750$&$1.260$e$-1$&$9$ [0.02 s]\\
    $274,625$&$7.808$e$-2$&$0.48$&$8$ [1.74 s]&$18,267$&$9.518$e$-2$&$12$ [0.12 s]\\
    $2,146,689$&$5.555$e$-2$ &$0.49$&$8$ [13.20 s]&$28,533$&$8.631$e$-2$&$12$ [0.23 s]\\
    $16,974,593$&$3.940$e$-2$ &$0.50$&$8$ [102.90 s]&$86,893$&$6.466$e$-2$&$12$ [0.77 s]\\
    &&& &$106,903$&$6.144$e$-2$&$12$ [1.15 s]\\
    &&& &$362,570$&$4.538$e$-2$&$12$ [3.72 s]\\
    &&& &$404,330$&$4.397$e$-2$&$11$ [4.40 s]\\
    &&& &$1,507,002$&$3.195$e$-2$&$12$ [18.44 s]\\
    \hline
  \end{tabular}
  }
    \caption{{\bf Example~3} 
    (Discontinuous Target \eqref{Sec:NumericalResults:Eqn:Example3:DiscontinuousTarget}, $d=2$, $L^2$ regularization, 
    mass lumping, nested iteration):
    Convergence in the  $L^2(Q)$-norm, number of nested SC-PCG iterations,
    and time in seconds.}
  \label{Table:wave_ex3_l2_2d_lumpedmass_nested}
\end{table}

\begin{table}[ht]
  {\small
  \begin{tabular}{|r|cc|c|r|c|c|c|}
    \hline
    \multicolumn{4}{|c|}{Uniform}&\multicolumn{3}{c|}{Adaptive}\\
    \hline
    \#Vertices &$\|\tilde{y}_{\varrho h}-y_d\|$&EOC&SC-PCG&\#Vertices &$\|\tilde{y}_{\varrho h}-y_d\|$&SC-PCG\\
    \hline
    $4,913$&$1.520$e$-1$&$-$&$88$ [0.02 s] &$4,913$&$1.520$e$-1$&$88$ [0.02 s]\\
    $35,937$&$1.083$e$-1$&$0.49$&$8$ [0.004 s] &$7,848$&$1.312$e$-1$&$10$ [0.003 s]\\
    $274,625$&$7.744$e$-2$&$0.48$&$8$ [0.005 s]&$23,967$&$8.892$e$-2$&$10$ [0.004 s]\\
    $2,146,689$&$5.527$e$-2$&$0.49$&$8$ [0.024 s]&$44,470$&$7.413$e$-2$&$10$ [0.006 s]\\
    $16,974,593$&$3.931$e$-2$&$0.49$&$8$ [0.18 s]&$84,302$&$6.290$e$-2$&$9$ [0.006 s]\\
    $135,005,697$&$2.789$e$-2$&$0.49$&$8$ [1.31 s]&$189,462$&$5.034$e$-2$&$9$ [0.007 s]\\
    &&& &$552,590$&$3.709$e$-2$&$11$ [0.01 s]\\
    &&& &$747,512$&$3.510$e$-2$&$10$ [0.01 s]\\
    &&& &$1,586,023$&$2.723$e$-2$&$13$ [0.04 s]\\
    \hline
  \end{tabular}
  }
  \caption{{\bf Example~3} 
    (Discontinuous Target \eqref{Sec:NumericalResults:Eqn:Example3:DiscontinuousTarget}, $d=2$, $L^2$ regularization, 
    mass lumping, parallel nested iteration):
    Convergence in the $L^2(Q)$-norm, number of nested SC-PCG iterations,
    and time in seconds, using 256 cores.}
  \label{Table:wave_ex3_l2_2d_parallel_lumpedmass_nested}
\end{table}

\begin{table}[ht]
  {\small
  \begin{tabular}{|r|c|c|r|c|c|c|}
    \hline
    \multicolumn{3}{|c|}{Uniform}&\multicolumn{3}{c|}{Adaptive}\\
    \hline
    \#Vertices &$\|\tilde{y}_{\varrho h}-y_d\|$&SC-PCG&\#Vertices &$\|\tilde{y}_{\varrho h}-y_d\|$&SC-PCG\\
    \hline
    $178$&$2.501$e$-1$&$31$ [0.003 s] &$178$&$2.501$e$-1$&$31$ [0.003 s]\\
    $715$&$2.001$e$-1$&$24$ [0.01 s] &$296$&$2.006$e$-1$&$12$ [0.002 s]\\
    $2,185$&$1.802$e$-1$&$18$ [0.027 s] &$569$&$1.880$e$-1$&$7$ [0.003 s]\\
    $9,225$&$1.548$e$-1$&$43$ [0.26 s] &$1,316$&$1.620$e$-1$&$4$ [0.003 s]\\
    $19,057$&$1.380$e$-1$&$20$ [0.31 s] &$2,167$&$1.475$e$-1$&$21$ [0.03 s]\\
    $47,073$&$1.262$e$-1$&$42$ [4.45 s] &$6,479$&$1.274$e$-1$&$16$ [0.08 s]\\
    $273,281$&$1.023$e$-1$&$25$ [10.62 s] &$18,895$&$1.127$e$-1$&$21$ [0.32 s]\\
    $700,161$&$9.261$e$-2$&$45$ [82.68 s]&$48,705$&$9.441$e$-2$&$24$ [1.32 s]\\
    $2,051,841$&$8.217$e$-2$&$54$ [194.07 s]&$77,141$&$8.833$e$-2$&$23$ [2.07 s]\\
    $5,585,665$&$7.132$e$-2$&$33$ [496.00 s]&$245,196$&$7.890$e$-2$&$14$ [4.00 s]\\
    $10,828,545$&$6.642$e$-2$&$40$ [808.90 s]&$378,810$&$6.860$e$-2$&$24$ [11.53 s]\\
    $32,127,745$&$5.920$e$-2$&$54$ [3471.20 s]&$603,678$&$6.351$e$-2$&$28$ [42.35 s]\\
    $$&$$$$&$$ &$762,073$&$6.177$e$-2$&$32$ [54.94 s]\\
    $$&$$$$&$$ &$1,343,769$&$5.786$e$-2$&$41$ [98.94 s]\\
    \hline
  \end{tabular}
  }
      \caption{{\bf Example~3} 
    (Discontinuous Target \eqref{Sec:NumericalResults:Eqn:Example3:DiscontinuousTarget}, $d=3$, $L^2$ regularization, 
    mass lumping, nested iteration):
    Convergence in the  $L^2(Q)$-norm, number of nested SC-PCG iterations,
    and time in seconds.}
  \label{Table:wave_ex3_l2_3d_lumpedmass_nested}
\end{table}

\subsection{Energy Regularization}
\label{SubSec:EnergyRegularization}
In the BP-PCG, we use \eqref{Sec:Solvers:Eqn:ScaledAMGpreconditioner} as
$\widehat{A}_{\varrho h}$ with $\delta=0.25$, and $D_h = \mbox{lump}(M_h)$ as
$\widehat{S}_{\varrho h}$. Furthermore, for the application of
$A_{\varrho h}^{-1}$ to a given vector $\mathbf{v}$ inside each 
SC-PCG/CG iteration, 
the Ruge--St\"{u}ben AMG method \cite{LLSY:RugeStuben}  has been used to
solve $A_{\varrho h}\mathbf{w}=\mathbf{v}$ until the relative preconditioned
residual error is reduced by a factor of $10^{12}$.

In Tables~\ref{Numa:wave_ex1_energy_2d}-\ref{Table:wave_ex3_energy_3d}, 
we provide the convergence studies of the space-time finite element
approximations to the targets 
\eqref{Sec:NumericalResults:Eqn:Example1:SmoothTarget},
\eqref{Sec:NumericalResults:Eqn:Example2:ContinuousTarget},
and 
\eqref{Sec:NumericalResults:Eqn:Example3:DiscontinuousTarget},
which correspond to {\bf Examples~1}, {\bf 2} and {\bf 3},
in the $L^2$-norm, and the corresponding number of iterations for the
preconditioned GMRES, SC-PCG/CG and PB-PCG solvers
for two as well as three space dimensions.
We observe the expected convergence rate for Examples~2 and 3, 
cf. Tables~\ref{Table:wave_ex2_energy_2d}-\ref{Table:wave_ex3_energy_3d},
whereas the convergence rates for the smooth target from  Example~1  
is reduced for both $d=2$ and $d=3$;
see Tables~\ref{Numa:wave_ex1_energy_2d} and
\ref{Numa:wave_ex1_energy_3d}, respectively.
This phenomena has been explained in Remark \ref{rem:lower-order-convergence-energy-regularization}; see also \cite{LLSY:LoescherSteinbach:2024SINUM}.
Figure~\ref{fig:convergence_energy_2d_3d} illustrates 
the corresponding convergence history.
The robustness of the proposed preconditioners are well confirmed by almost
mesh-independent numbers of iterations for all solvers in all cases.

\begin{table}[ht]
  {\small
  \begin{tabular}{|cr|cc|ccc|}
    \hline
    \multicolumn{2}{|c|}{}&\multicolumn{2}{c|}{Convergence}&\multicolumn{3}{c|}{Solvers (Number of Iterations)}\\
    \hline
    Level&\#Vertices&$\|y_{\varrho h}-y_d\|$&EOC&GMRES& SC-PCG/CG & PB-PCG\\
    \hline
    $1$&$125$&$8.082$e$-2$&$-$ & $59$ & $36/37$&$63$  \\
    $2$&$729$&$3.749$e$-2$&$1.11$&$96$ &$52/67$&$103$\\
    $3$&$4,913$&$1.550$e$-2$&$1.27$ &$99$ &$53/73$&$109$\\
    $4$&$35,937$&$5.994$e$-3$&$1.37$ &$97$&$52/76$&$113$\\
    $5$&$274,625$&$2.272$e$-3$&$1.40$ &$94$&$51/79$&$118$\\
    $6$&$2,146,689$&$8.273$e$-4$ &$1.46$ &$91$&$50/80$&$118$\\
    \hline
  \end{tabular}
  }
  \caption{{\bf Example~1} 
  (Smooth Target \ref{Sec:NumericalResults:Eqn:Example1:SmoothTarget}, $d=2$, energy regularization): 
  Convergence in the $L^2(Q)$-norm, and number of iterations  
  for attaining the relative accuracy $10^{-11}$.}        
  \label{Numa:wave_ex1_energy_2d}
\end{table}

\begin{table}[ht]
  {\small
  \begin{tabular}{|cr|cc|ccc|}
    \hline
    \multicolumn{2}{|c|}{}&\multicolumn{2}{c|}{Convergence}&\multicolumn{3}{c|}{Solvers (Number of Iterations) }\\
    \hline
    Level&\#Vertices&$\|y_{\varrho h}-y_d\|$&EOC&GMRES&SC-PCG/CG& PB-PCG\\
    \hline
    $1$&$178$&$9.152$e$-2$&$-$&$40$&$23/24$&$51$\\
    $2$&$235$&$8.818$e$-2$&$0.52$&$43$&$26/26$&$59$\\
    $3$&$315$&$8.016$e$-2$&$1.30$&$94$&$52/43$&$91$\\
    $4$&$715$&$6.117$e$-2$&$1.32$&$139$&$78/97$&$155$\\
    $5$&$1,493$&$4.843$e$-2$&$1.29$&$143$&$74/128$&$167$\\
    $6$&$2,185$&$4.373$e$-2$&$1.04$&$151$&$80/141$&$179$\\
    $7$&$3,465$&$3.801$e$-2$&$1.21$&$206$&$106/125$&$238$\\
    $8$&$9,225$&$2.743$e$-2$&$1.34$&$210$&$110/158$&$261$\\
    $9$&$19,057$&$2.113$e$-2$&$1.44$&$190$&$99/203$&$239$\\
    $10$&$26,593$&$1.893$e$-2$&$1.32$&$183$&$96/209$&$240$\\
    $11$&$47,073$&$1.570$e$-2$&$1.31$&$227$&$122/177$&$308$\\
    $12$&$134,113$&$1.097$e$-2$&$1.37$&$226$&$118/194$&$312$\\
    $13$&$273,281$&$8.462$e$-3$&$1.44$&$200$&$104/247$&$264$\\
    $14$&$372,481$&$7.564$e$-3$&$1.48$&$194$&$100/243$&$263$\\
    $15$&$700,161$&$6.146$e$-3$&$1.32$&$229$&$121/206$&$318$\\
    $16$&$2,051,841$&$4.145$e$-3$&$1.46$&$225$&$118/219$&$316$\\
    $17$&$4,144,513$&$3.201$e$-3$&$1.49$&$199$&$103/272$&$275$\\
    \hline
  \end{tabular}
  }
  \caption{{\bf Example~1} 
  (Smooth Target \ref{Sec:NumericalResults:Eqn:Example1:SmoothTarget}, $d=3$, energy regularization): 
  Convergence in the $L^2(Q)$-norm, and number of 
  iterations for attaining the relative accuracy $10^{-11}$.}        
  \label{Numa:wave_ex1_energy_3d}
\end{table}

\begin{table}[ht]
  {\small
  \begin{tabular}{|cr|cc|ccc|}
    \hline
    \multicolumn{2}{|c|}{}&\multicolumn{2}{c|}{Convergence}&\multicolumn{3}{c|}{Solvers (Number of Iterations)}\\
    \hline
    Level&\#Vertices&$\|y_{\varrho h}-y_d\|$&EOC&PGMRES&SC-PCG/CG&PB-PCG\\
    \hline
    $1$&$125$&$4.817$e$-2$&$-$ & $61$ & $36/38$&$57$ \\
    $2$&$729$&$3.146$e$-2$&$0.61$&$101$ &$54/69$&$94$\\
    $3$&$4,913$&$1.541$e$-2$&$1.03$&$103$&$55/73$&$96$\\
    $4$&$35,937$&$6.295$e$-3$&$1.29$&$100$&$53/73$&$98$\\
    $5$&$274,625$&$2.445$e$-3$&$1.36$ &$97$&$52/73$&$99$\\
    $6$&$2,146,689$&$9.191$e$-4$&$1.41$ &$94$&$51/72$&$99$\\
    \hline
  \end{tabular}
  }
  \caption{{\bf Example~2} 
  (Continuous Target \eqref{Sec:NumericalResults:Eqn:Example2:ContinuousTarget}, $d=2$, energy regularization): 
  Convergence in the $L^2(Q)$-norm, and number of iterations  
  for attaining the relative accuracy $10^{-11}$.} 
  \label{Table:wave_ex2_energy_2d}
\end{table}
\begin{table}[ht]
  {\small
  \begin{tabular}{|cr|cc|ccc|}
    \hline
    \multicolumn{2}{|c|}{}&\multicolumn{2}{c|}{Convergence}&\multicolumn{3}{c|}{Solvers (Number of Iterations)}\\
    \hline
    Level&\#Vertices&$\|y_{\varrho h}-y_d\|$&EOC&PGMRES&SC-PCG/CG&PB-PCG\\
    \hline
    $1$&$178$&$5.938$e$-2$&$-$&$40$&$23/25$&$54$\\
    $2$&$235$&$2.587$e$-2$&$11.56$&$44$&$26/26$&$55$\\
    $3$&$315$&$2.073$e$-2$&$3.03$&$97$&$51/42$&$96$\\
    $4$&$715$&$1.987$e$-2$&$0.21$&$146$&$80/90$&$169$\\
    $5$&$1,493$&$1.971$e$-2$&$0.04$&$148$&$76/124$&$184$\\
    $6$&$2,185$&$1.855$e$-2$&$0.62$&$155$&$81/133$&$196$\\
    $7$&$3,465$&$1.772$e$-2$&$0.39$&$211$&$107/120$&$261$\\
    $8$&$9,225$&$1.552$e$-2$&$0.55$&$219$&$113/151$&$290$\\
    $9$&$19,057$&$1.261$e$-2$&$1.15$&$195$&$100/182$&$263$\\
    $10$&$26,593$&$1.039$e$-2$&$2.33$&$185$&$97/188$&$267$\\
    $11$&$47,073$&$9.086$e$-3$&$0.94$&$232$&$122/158$&$339$\\
    $12$&$134,113$&$6.737$e$-3$&$1.15$&$232$&$121/173$&$352$\\
    $13$&$273,281$&$5.345$e$-3$&$1.29$&$203$&$106/206$&$295$\\
    $14$&$372,481$&$4.909$e$-3$&$1.12$&$193$&$101/207$&$304$\\
    $15$&$700,161$&$4.158$e$-3$&$1.07$&$233$&$121/174$&$364$\\
    $16$&$2,051,841$&$2.934$e$-3$&$1.28$&$227$&$118/190$&$369$\\
    $17$&$4,144,513$&$2.185$e$-3$&$1.70$&$198$&$103/215$&$307$\\
    \hline
  \end{tabular}
  }
  \caption{{\bf Example~2} 
  (Continuous Target \eqref{Sec:NumericalResults:Eqn:Example2:ContinuousTarget}, $d=3$, energy regularization): 
    Convergence in the $L^2(Q)$-norm, and number of iterations
    for attaining the relative accuracy $10^{-11}$.} 
  \label{Table:wave_ex2_energy_3d}
\end{table} 

\begin{table}[ht]
  {\small
  \begin{tabular}{|cr|cc|ccc|}
    \hline
    \multicolumn{2}{|c|}{}&\multicolumn{2}{c|}{Convergence}&\multicolumn{3}{c|}{Solvers (Number of Iterations)}\\
    \hline
    Level&\#Vertices&$\|y_{\varrho h}-y_d\|$&EOC&PGMRES&SC-PCG/CG & PB-PCG\\
    \hline
    $1$&$125$&$2.502$e$-1$&$-$ & $60$ & $36/38$&$62$ \\
    $2$&$729$&$1.944$e$-1$&$0.36$&$100$ &$53/68$&$93$\\
    $3$&$4,913$&$1.485$e$-1$&$0.39$ &$104$ &$55/74$&$97$\\
    $4$&$35,937$&$1.093$e$-1$&$0.44$ &$103$&$55/74$&$99$\\
    $5$&$274,625$&$7.895$e$-2$&$0.47$ &$102$&$55/74$&$103$\\
    $6$&$2,146,689$&$5.648$e$-2$ &$0.48$ &$103$&$55/74$&$105$\\
    \hline
  \end{tabular}
  }
    \caption{{\bf Example~3} 
  (Discontinuous Target \eqref{Sec:NumericalResults:Eqn:Example3:DiscontinuousTarget}, $d=2$, energy regularization): 
  Convergence in the $L^2(Q)$-norm, and number of iterations  
  for attaining the relative accuracy $10^{-11}$.} 
  \label{Table:wave_ex3_energy_2d}
\end{table}

\begin{table}[ht]
  {\small
  \begin{tabular}{|cr|cc|ccc|}
    \hline
    \multicolumn{2}{|c|}{}&\multicolumn{2}{c|}{Convergence}&\multicolumn{3}{c|}{Solvers (Number of Iterations)}\\
    \hline
    Level&\#Vertices&$\|y_{\varrho h}-y_d\|$& EOC &PGMRES&SC-PCG/CG & PB-PCG\\
    \hline
    $1$&$178$&$2.337$e$-1$&$-$&$40$ & $23/24$&$51$\\
    $2$&$235$&$2.302$e$-1$&$0.21$&$44$ & $26/26$&$58$\\
    $3$&$315$&$1.672$e$-1$&$4.37$&$90$ & $51/43$&$86$\\
    $4$&$715$&$1.864$e$-1$&$-0.53$&$146$ & $78/90$&$164$ \\
    $5$&$1,493$&$1.651$e$-1$&$0.67$&$146$ & $77/128$&$169$  \\
    $6$&$2,185$&$1.701$e$-1$&$-0.31$&$150$ & $79/135$&$171$  \\
    $7$&$3,465$&$1.550$e$-1$&$0.80$&$206$ & $107/122$&$237$  \\
    $8$&$9,225$&$1.483$e$-1$&$0.18$&$220$ & $112/152$&$265$  \\
    $9$&$19,057$&$1.340$e$-1$&$0.56$&$195$ & $100/188$&$247$  \\
    $10$&$26,593$&$1.280$e$-1$&$0.55$&$187$ & $98/192$&$244$  \\
    $11$&$47,073$&$1.214$e$-1$&$0.37$&$234$ & $123/162$&$310$  \\
    $12$&$134,113$&$1.103$e$-1$&$0.37$&$235$ & $122/178$&$321$  \\
    $13$&$273,281$&$1.006$e$-1$&$0.51$&$207$ & $108/216$&$277$  \\
    $14$&$372,481$&$9.606$e$-2$&$0.60$&$199$ & $103/214$&$270$  \\
    $15$&$700,161$&$9.045$e$-2$&$0.38$&$236$ & $124/181$&$326$  \\
    $16$&$2,051,841$&$8.059$e$-2$&$0.43$&$232$ & $121/194$&$334$  \\
    $17$&$4,144,513$&$7.373$e$-2$&$0.51$&$207$ & $109/224$&$290$  \\
    \hline
  \end{tabular}
  }
  \caption{{\bf Example~3} 
  (Discontinuous Target \eqref{Sec:NumericalResults:Eqn:Example3:DiscontinuousTarget}, $d=3$, energy regularization): 
  Convergence in the $L^2(Q)$-norm, and number of iterations  
  for attaining the relative accuracy $10^{-11}$.} 
  \label{Table:wave_ex3_energy_3d}
\end{table}

\begin{figure}[htpb!]
  \begin{tikzpicture}[scale=0.725,transform shape]
    
    \begin{axis}[
        xmode = log,
        ymode = log,
        xlabel=$N$,
        ylabel=$\| y_{\varrho h}- y_d \|_{L^2(Q)}$,
        legend pos=south west,
        legend style={font=\tiny}]
      
      \addplot [solid, mark=diamond*, color=blue] table [col sep=
        &, y=err, x=N]{2d_smooth_energy.dat};
      \addlegendentry{target
        (\ref*{Sec:NumericalResults:Eqn:Example1:SmoothTarget}), $d=2$}

      \addplot  [solid, mark=square*, color=red] table [col sep=
        &, y=err, x=N]{2d_linear_energy.dat};
      \addlegendentry{target
        (\ref*{Sec:NumericalResults:Eqn:Example2:ContinuousTarget}), $d=2$}
      
      \addplot [solid, mark=*, color=brown] table [col sep=
        &, y=err, x=N]{2d_discontinuous_energy.dat};
      \addlegendentry{target
        (\ref*{Sec:NumericalResults:Eqn:Example3:DiscontinuousTarget}), $d=2$}
      
      \addplot [dashed, thin, color=blue] table [col sep=
        &, y=err, x=N]{2d_smooth_exact_energy.dat};
      \addlegendentry{$h^{1.5}$}
      
      
      \addplot [dashed, thin, color=brown] table [col sep=
        &, y=err, x=N]{2d_discontinuous_exact_energy.dat};
      \addlegendentry{$h^{0.5}$}
    \end{axis}
    
    \hspace{6.5cm}
    
    \begin{axis}[
        xmode = log,
        ymode = log,
        xlabel=$N$,
        ylabel=$\| y_{\varrho h}- y_d \|_{L^2(Q)}$,
        legend pos=south west,
        legend style={font=\tiny}]
      
      \addplot [solid, mark=diamond*, color=blue] table [col sep=
        &, y=err, x=N]{energy3d_smooth.dat};
      \addlegendentry{target (\ref*{Sec:NumericalResults:Eqn:Example1:SmoothTarget}), $d=3$}

      \addplot  [solid, mark=square*, color=red] table [col sep=
        &, y=err, x=N]{energy3d_linear.dat};
      \addlegendentry{target (\ref*{Sec:NumericalResults:Eqn:Example2:ContinuousTarget}), $d=3$}
      
      \addplot [solid, mark=*, color=brown] table [col sep=
        &, y=err, x=N]{energy3d_discontinuous.dat};
      \addlegendentry{target (\ref*{Sec:NumericalResults:Eqn:Example3:DiscontinuousTarget}), $d=3$}
      
      \addplot [dashed, thin, color=blue] table [col sep=
        &, y=err, x=N]{energy3d_smooth_exact.dat};
      \addlegendentry{$h^{1.5}$}
      
      \addplot [dashed, thin, color=red] table [col sep=
        &, y=err, x=N]{energy3d_continuous_exact.dat};
      \addlegendentry{$h^{1.5}$}
      
      \addplot [dashed, thin, color=brown] table [col sep=
        &, y=err, x=N]{energy3d_discontinuous_exact.dat};
      \addlegendentry{$h^{0.5}$}
      
    \end{axis}
  \end{tikzpicture}
  \caption{Convergence history for all the targets
    (\ref{Sec:NumericalResults:Eqn:Example1:SmoothTarget})-(\ref{Sec:NumericalResults:Eqn:Example3:DiscontinuousTarget}) 
    and for energy regularization: $d=2$ (left) and $d=3$ (right).}
  \label{fig:convergence_energy_2d_3d}
\end{figure}
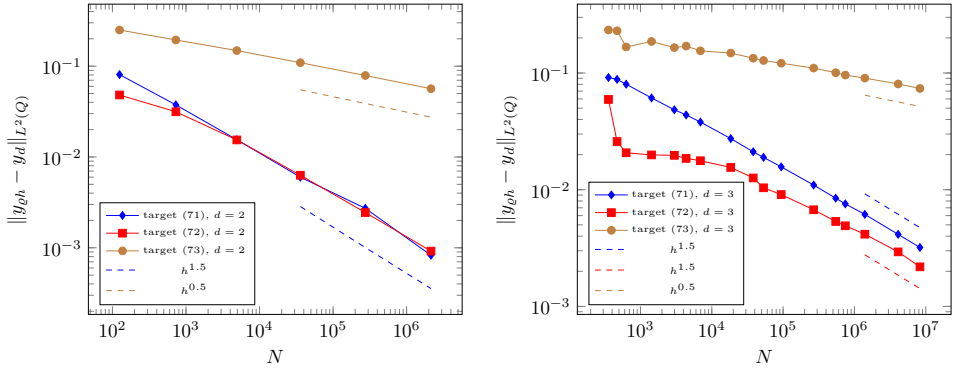 

As in the preceding Subsection~\ref{SubSec:L2Regularization} for the
mass-lumped Schur complement system, we may use the nested iteration
procedure on a sequence of uniformly and adaptively refined meshes 
to solve the inexact Schur complement system
\eqref{Sec:Solvers:Eqn:InexactSy=yd}. We again use the most interesting,
discontinuous target
\eqref{Sec:NumericalResults:Eqn:Example3:DiscontinuousTarget}
for our numerical test. To control the nested iteration error, we have
used the adaptive threshold 
$\alpha\left[N_{l}/N_{l-1}\right]^{\beta/3}$ for $l=2,3,\ldots,$,
with  $\alpha=0.5$ and $\alpha=0.1$ for $d=2$ and $d=3$, respectively, 
and $\beta=0.5$ and $\beta=0.75$
for the uniform and adaptive refinement, respectively.
For the implementation of the operation $A_{\varrho h}^{-1}$ within each
SC-PCG iteration, several inner AMG iterations are applied, namely $3$. 
The $L^2$ convergence, the number of SC-PCG iterations, and the
computational time in seconds 
are provided in Table~\ref{Table:wave_ex3_energy_2d_nested} for $d=2$.
From this table, we observe more efficiency without loss of accuracy 
compared with the results in Tables~\ref{Table:wave_ex3_energy_2d} 
obtained for the non-nested (single-grid) iterations.  

\begin{table}[ht]
  {\small
  \begin{tabular}{|r|cc|c|r|c|c|c|}
    \hline
    \multicolumn{4}{|c|}{Uniform}&\multicolumn{3}{c|}{Adaptive}\\
    \hline
    \#Vertices &$\|\tilde{y}_{\varrho h}-y_d\|$&EOC&SC-PCG&\#Vertices &$\|\tilde{y}_{\varrho h}-y_d\|$&SC-PCG\\
    \hline
    $125$&$2.502$e$-1$&$-$&$28$ [0.01 s] &$125$&$2.502$e$-1$&$28$ [0.01 s]\\
    $729$&$1.996$e$-1$&$0.33$&$2$ [0.005 s] &$223$&$2.513$e$-1$&$1$ [0.0007 s]\\
    $4,913$&$1.492$e$-1$&$0.42$&$2$ [0.05 s] &$1,044$&$1.668$e$-1$&$3$ [0.02 s]\\
    $35,937$&$1.096$e$-1$&$0.44$&$2$ [0.38 s] &$4,616$&$1.101$e$-1$&$3$ [0.12 s]\\
    $274,625$&$7.902$e$-2$&$0.47$&$2$ [4.09 s]&$17,934$&$8.560$e$-2$&$5$ [0.87 s]\\
    $2,146,689$&$5.646$e$-2$ &$0.49$&$2$ [34.66 s]&$24,487$&$8.309$e$-2$&$3$ [0.54 s]\\
    $16,974,593$&$4.013$e$-2$ &$0.49$&$2$ [281.89 s]&$82,560$&$6.274$e$-2$&$3$ [2.12 s]\\
    &&& &$94,025$&$6.051$e$-2$&$3$ [2.70 s]\\
    &&& &$349,864$&$4.400$e$-2$&$3$ [11.96 s]\\
    \hline
  \end{tabular}
  }
    \caption{{\bf Example~3} 
    (Discontinuous Target \eqref{Sec:NumericalResults:Eqn:Example3:DiscontinuousTarget}, $d=2$, energy regularization, 
      nested iteration):
    Convergence in the  $L^2(Q)$-norm, number of nested SC-PCG iterations,
    and time in seconds.}
  \label{Table:wave_ex3_energy_2d_nested}
\end{table}

\section{Conclusion and outlook}
\label{Sec:ConclusionAndOutlook}
We have considered tracking-type, distributed OCPs with both the standard 
$L^2$ regularization and the more general energy regularization subject to
hyperbolic state equations without additional control and{/}or state
constraints. The regularization parameter $\varrho$ is related to the mesh
size $h$ in such way that the deviation of the computed FE state $y_h$ from
the desired state $y_d$ is of asymptotically optimal order wrt the $L^2$ norm
in dependence on the smoothness of $y_d$. In particular, the case of
discontinuous targets,  that is the most interesting case 
from a practical point of view, is covered by the analysis. 
The predicted convergence rate $h^{1/2 - \varepsilon}$ is observed in 
all our numerical experiments. This rate can easily be improved by a
simple space-time FE adaptivity based on the computable and localizable
error $\|y_d - y_h\|_{L^2(Q)}$ and a variable choice of $\varrho$ adapted
to the local mesh size accordingly. In all cases, the primal Schur
complement $S_{\varrho h}$ is spectrally equivalent to the mass matrix $M_h$,
and, therefore, to some diagonal approximation to mass matrix like the
lumped mass matrix $D_h = \mbox{lump}(M_h)$. This is the basis for the
construction of fast iterative solvers like the PB-PCG and the PCG for the
SID system \eqref{Sec:Introduction:Eqn:AbstractDiscreteReducedOptimalitySystem} 
and the SPD SC system \eqref{Sec:Introduction:Eqn:SchurComplementSystem},
respectively. In order to ensure a fast multiplication of SC $S_{\varrho h}$
with some vector in the latter case, we can replace  $A_{\varrho}^{-1}$ by
$(\mbox{lump}(\overline{M}_{\varrho h}))^{-1}$ for the $L^2$ regularization,
whereas inner multigrid iterations for the approximate inversion of the
algebraic space-time Laplacian must be used in the case of energy
regularization. We can control the number of inner iterations 
in such a way that the discretization is not disturbed. 

In practice, these solvers should always be used in a nested iteration
framework on a sequence of uniformly or adaptively refined meshes producing
state approximations $y_l^{k_l}$, that differ from the desired state $y_d$
wrt to the $L^2(Q)$-norm in the order of the discretization error,
in asymptotically  optimal or, at least, almost optimal complexity
as one can observe from
Tables~\ref{Table:wave_ex3_l2_2d_lumpedmass_nested},  
\ref{Table:wave_ex3_l2_2d_parallel_lumpedmass_nested},
\ref{Table:wave_ex3_l2_3d_lumpedmass_nested},
and \ref{Table:wave_ex3_energy_2d_nested}. So, the nested iteration process
can be stopped as soon as some given (relative) accuracy of the nested
iteration approximation $y_l^{k_l}$ of the desired state $y_d$ is reached or 
the cost of the control measured in terms of $\| u_l^{k_l}\|_U$ exceeds
some given threshold, where $u_l^{k_l}$ is the discrete control recovered
from the last nested state iterate $y_l^{k_l} \in Y_l = Y_{h_l}$. We refer
the reader to  \cite{LLSY:LangerLoescherSteinbachYang:2024CAMWA}
for a more detailed 
description of this nested iteration procedure in the case of elliptic OCPs.

It is possible to generalize these results to other hyperbolic state
equation like dynamic elasticity initial-boundary value problems. Control
and{/}or state constraints can be considered in the same way as was done in
\cite{LLSY:GanglLoescherSteinbach:2023} for elliptic 
state equations. The corresponding non-linear algebraic system can be solved 
by semi-smooth Newton methods \cite{LLSY:HintermuellerItoKunisch2003SIOPT}.
The linear system arising at each step of the semi-smooth Newton iteration 
has the same structure as the linear systems studied in this paper.
 
\section*{Acknowledgments}
We would like to thank the computing resource support of the supercomputer
MACH--2\footnote{https://www3.risc.jku.at/projects/mach2/} from Johannes
Kepler Universit\"{a}t Linz and of the high performance computing cluster
Radon1\footnote{https://www.oeaw.ac.at/ricam/hpc} from Johann Radon Institute
for Computational and Applied Mathematics (RICAM).
Further, the financial support for the fourth author by the
Austrian Federal Ministry for Digital and Economic Affairs, the National
Foundation for Research, Technology and Development and the Christian
Doppler Research Association is gratefully acknowledged.


\bibliography{LLSY2024CAMWA}

\newcommand{\noopsort}[1]{} \newcommand{\printfirst}[2]{#1}
  \newcommand{\singleletter}[1]{#1} \newcommand{\switchargs}[2]{#2#1}
\begin{thebibliography}{10}

\bibitem{LLSY:AxelssonKaratson:2020NumerMath}
O.~Axelsson and J.~Kar\'{a}tson.
\newblock Superior properties of the {PRESB} preconditioner for operators on
  two-by-two block form with square blocks.
\newblock {\em Numer. Math}, 146(2):335–368, 2020.

\bibitem{LLSY:AxelssonNeytchevaStroem:2018JNM}
O.~Axelsson, M.~Neytcheva, and A.~Str\"om.
\newblock An efficient preconditioning method for state box-constrained optimal
  control problems.
\newblock {\em J. Numer. Math.}, 26(4):185--207, 2018.

\bibitem{LLSY:BabuskaVogelius:1984NumerMath}
I.~Babu\v{s}ka and M.~Vogelius.
\newblock Feedback and adaptive finite element solution of one-dimensional
  boundary value problems.
\newblock {\em Numer. Math.}, 44:75--102, 1984.

\bibitem{LLSY:BaiBenziChenWang:2013IMANA}
Z.-Z. Bai, M.~Benzi, F.~Chen, and Z.-Q. Wang.
\newblock Preconditioned {MHSS} iteration methods for a class of block
  two-by-two linear systems with applications to distributed control problems.
\newblock {\em IMA J. Numer. Anal.}, 33(1):343--369, 2013.

\bibitem{LLSY:BaiPan:2021Book}
Z.-Z. Bai and J.-Y. Pan.
\newblock {\em Matrix Analysis and Computations}.
\newblock SIAM, 2021.

\bibitem{LLSY:BenziGolubLiesen:2005ActaNumerica}
M.~Benzi, G.~H. Golub, and J.~Liesen.
\newblock Numerical solution of saddle point problems.
\newblock {\em Acta Numer.}, 14:1--137, 2005.

\bibitem{LLSY:JB95}
J.~Bey.
\newblock Tetrahedral grid refinement.
\newblock {\em Computing}, 55:355--378, 1995.

\bibitem{LLSY:BramblePasciak:1988a}
J.~H. Bramble and J.~E. Pasciak.
\newblock A preconditioning technique for indefinite systems resulting from
  mixed approximations of elliptic problems.
\newblock {\em Math. Comp.}, 50(181):1--17, 1988.

\bibitem{LLSY:Dauge:1988}
M.~Dauge.
\newblock {\em Elliptic Boundary Value Problems on Corner Domains}, volume 1341
  of {\em Lecture Notes in Mathematics}.
\newblock Springer, Berlin, Heidelberg, 1988.

\bibitem{LLSY:DeLosReyes:2015Book}
J.~{De los Reyes}.
\newblock {\em Numerical PDE-Constrained Optimization}.
\newblock Springer Cham, 2015.

\bibitem{LLSY:DravinsNeytcheva:2021}
I.~Dravins and M.~Neytcheva.
\newblock {\em On the Numerical Solution of State- and Control-constrained
  Optimal Control Problems}.
\newblock Department of Information Technology, Uppsala Universitet, 2021.

\bibitem{LLSY:ElmanSilvesterWathen:2005Book}
H.~C. Elman, D.~J. Silvester, and A.~J. Wathen.
\newblock {\em Finite elements and fast iterative solvers: With applications in
  incompressible fluid dynamics}.
\newblock Oxford University Press, 2005.

\bibitem{LLSY:GanglLoescherSteinbach:2023}
P.~Gangl, R.~L\"oscher, and O.~Steinbach.
\newblock Regularization and finite element error estimates for elliptic
  distributed optimal control problems with energy regularization and state or
  control constraints, 2023.
\newblock arXiv:2306.15316.

\bibitem{LLSY:Grisvard:1985}
P.~Grisvard.
\newblock {\em Elliptic problems in nonsmooth domains}, volume~69 of {\em
  Classics in Applied Mathematics}.
\newblock SIAM, Philadelphia, 2011.

\bibitem{LLSY:HerzogHeinkenschlossKaliseStadlerTrelat:2022Proceedings}
R.~Herzog, M.~Heinkenschloss, D.~Kalise, G.~Stadler, and E.~Tr\'{e}lat,
  editors.
\newblock {\em Optimization and Control for Partial Differential Equations},
  volume~29 of {\em Radon Series on Computational and Applied Mathematics}. de
  Gruyter, 2022.

\bibitem{LLSY:HintermuellerItoKunisch2003SIOPT}
M.~Hinterm\"uller, K.~Ito, and K.~Kunisch.
\newblock The primal-dual active set method as a semismooth {N}ewton method.
\newblock {\em SIAM J. Optimization}, 13:865--888, 2003.

\bibitem{LSTY:HinzePinnauUlbrichUlbrich:2009Book}
M.~Hinze, R.~Pinnau, M.~Ulbrich, and S.~Ulbrich.
\newblock {\em Optimization with PDE Constraints}, volume~23 of {\em
  Mathematical Modelling: Theory and Applications}.
\newblock Springer-Verlag, Berlin, 2009.

\bibitem{LLST:JungLangerMeyerQueckSchneider:1989a}
M.~Jung, U.~Langer, A.~Meyer, W.~Queck, and M.~Schneider.
\newblock Multigrid preconditioners and their applications.
\newblock In G.~Telschow, editor, {\em Third Multigrid Seminar, Biesenthal
  1988}, pages 11--52, Karl--Weierstrass--Institut Berlin, Report
  R--MATH--03/89, 1989.

\bibitem{LLSY:KunischReiterer2015ANM}
K.~Kunisch and S.~Reiterer.
\newblock A {G}autschi time-stepping approach to optimal control of the wave
  equation.
\newblock {\em Appl. Numer. Math.}, 90:55--76, 2015.

\bibitem{LSTY:Ladyzhenskaya:1985a}
O.~A. Ladyzhenskaya.
\newblock {\em The boundary value problems of mathematical physics}, volume~49
  of {\em Applied Mathematical Sciences}.
\newblock Springer, New York, 1985.

\bibitem{LLSY:LangerLoescherSteinbachYang:2023arXiv:2304.14664}
U.~Langer, R.~L{\"o}scher, O.~Steinbach, and H.~Yang.
\newblock Mass-lumping discretization and solvers for distributed elliptic
  optimal control problems, 2023.
\newblock arXiv:2304.14664.

\bibitem{LLSY:LangerLoescherSteinbachYang:2024CAMWA}
U.~Langer, R.~L\"oscher, O.~Steinbach, and H.~Yang.
\newblock An adaptive finite element method for distributed elliptic optimal
  control problems with variable energy regularization.
\newblock {\em Comput. Math. Appl.}, 160:1--14, 2024.

\bibitem{LLSY:LangerSteinbachTroeltzschYang:2021c}
U.~Langer, O.~Steinbach, F.~Tr{\"o}ltzsch, and H.~Yang.
\newblock Space-time finite element discretization of parabolic optimal control
  problems with energy regularization.
\newblock {\em SIAM J. Numer. Anal.}, 59(2):660--674, 2021.

\bibitem{LLSY:LangerSteinbachTroeltzschYang:2021b}
U.~Langer, O.~Steinbach, F.~Tr{\"o}ltzsch, and H.~Yang.
\newblock Unstructured space-time finite element methods for optimal control of
  parabolic equation.
\newblock {\em SIAM J. Sci. Comput.}, 43(2):A744--A771, 2021.

\bibitem{LSTY:Lions:1968Book}
J.~L. Lions.
\newblock {\em Contr\^{o}le optimal de syst\`{e}mes gouvern\'{e}s par des
  \'{e}quations aux d\'{e}riv\'{e}es partielles}.
\newblock Dunod Gauthier-Villars, Paris, 1968.

\bibitem{LLSY:LoescherSteinbach:2024SINUM}
R.~L\"{o}scher and O.~Steinbach.
\newblock Space-time finite element methods for distributed optimal control of
  the wave equation.
\newblock {\em SIAM J. Numer. Anal.}, 62(1):452--475, 2024.

\bibitem{LLSY:MardalWinther:2011NLA}
K.-A. Mardal and R.~Winther.
\newblock Preconditioning discretizations of systems of partial differential
  equations.
\newblock {\em Numer. Linear Algebra Appl.}, 18(1):1--40, 2011.

\bibitem{LLSY:NeumuellerSteinbach:2021a}
M.~Neum\"uller and O.~Steinbach.
\newblock Regularization error estimates for distributed control problems in
  energy spaces.
\newblock {\em Math. Methods Appl. Sci.}, 44:4176--4191, 2021.

\bibitem{LLSY:PearsonStollWathen:2014NLA}
J.~Pearson, M.~Stoll, and A.~Wathen.
\newblock Preconditioners for state-constrained optimal control problems with
  moreau-yosida penalty function.
\newblock {\em Numer. Linear Algebra Appl.}, 21(1):81--97, 2014.

\bibitem{LLSY:PearsonWathen:2012NLA}
J.~Pearson and A.~Wathen.
\newblock A new approximation of the {S}chur complement in preconditioners for
  {PDE}-constrained optimization.
\newblock {\em Numer. Linear Algebra Appl.}, 12(5):816--829, 2012.

\bibitem{PeraltaKunisch:2022}
G.~Peralta and K.~Kunisch.
\newblock Mixed and hybrid {P}etrov--{G}alerkin finite element discretization
  for optimal control of the wave equation.
\newblock {\em Numer. Math.}, 150:591--627, 2022.

\bibitem{LLSY:Rozloznik:2018Book}
M.~Rozlo\v{z}n\'{\i}ík.
\newblock {\em Saddle-Point Problems and Their Iterative Solution}.
\newblock Ne\u{c}as Center Series. Birkhäuser Cham, 2018.

\bibitem{LLSY:RugeStuben}
J.~W. Ruge and K.~St\"{u}ben.
\newblock Algebraic multigrid ({AMG}).
\newblock In S.~F. McCormick, editor, {\em Multigrid Methods}, pages 73--130.
  SIAM, Philadelphia, 1987.

\bibitem{LLSY:SchielaUlbrich:2014SIOPT}
A.~Schiela and S.~Ulbrich.
\newblock Operator preconditioning for a class of inequality constrained
  optimal control problems.
\newblock {\em SIAM J. Optim.}, 24(1):435--466, 2014.

\bibitem{LLSY:SchoeberlZulehner:2007SIMAX}
J.~Sch\"oberl and W.~Zulehner.
\newblock Symmetric indefinite preconditioners for saddle point problems with
  applications to {PDE}-constrained optimization problems.
\newblock {\em SIAM J. Matrix Anal. Appl.}, 29:752--773, 2007.

\bibitem{LLSY:SchulzWittum:2008CVS}
V.~Schulz and G.~Wittum.
\newblock Transforming smoothers for pde constrained optimization problems.
\newblock {\em Comput. Visual. Sci.}, 11:207--219, 2008.

\bibitem{LLSY:ScottZhang:1990}
L.~R. Scott and S.~Zhang.
\newblock Finite element interpolation of nonsmooth functions satisfying
  boundary conditions.
\newblock {\em Math. Comp.}, 54(190):483--493, 1990.

\bibitem{LLSY:SteinbachZank:2020}
O.~Steinbach and M.~Zank.
\newblock Coercive space-time finite element methods for initial boundary value
  problems.
\newblock {\em Electron. Trans. Numer. Anal.}, 52:154--194, 2020.

\bibitem{LLSY:SteinbachZank:2022}
O.~Steinbach and M.~Zank.
\newblock A generalized inf-sup stable variational formulation for the wave
  equation.
\newblock {\em J. Math. Anal. Appl.}, 505(1):Paper No. 125457, 24, 2022.

\bibitem{LLSY:RS08}
R.~Stevenson.
\newblock The completion of locally refined simplicial partitions created by
  bisection.
\newblock {\em Math. Comp.}, 77(261):227--241, 2008.

\bibitem{LLSY:StollWathen:2012NLA}
M.~Stoll and A.~Wathen.
\newblock Preconditioning for partial differential equation constrained
  optimization with control constraints.
\newblock {\em Numer. Linear Algebra Appl.}, 19:53--71, 2012.

\bibitem{LLSY:Troeltzsch:2010Book}
F.~Tr\"{o}ltzsch.
\newblock {\em Optimal control of partial differential equations: Theory,
  methods and applications}, volume 112 of {\em Graduate Studies in
  Mathematics}.
\newblock American Mathematical Society, Providence, Rhode Island, 2010.

\bibitem{LLSY:Wathen:2015ActaNumerica}
A.~Wathen.
\newblock Preconditioning.
\newblock {\em Acta Numerica}, 24:329--376, 2015.

\bibitem{LLSY:Zank:2020}
M.~Zank.
\newblock {\em Inf-sup stable space-time methods for time-dependent partial
  differential equations}, volume~38 of {\em Computation in Engineering and
  Science}.
\newblock Verlag der Technischen Universität Graz, 2020.

\bibitem{LLSY:Zulehner:2002MCOM}
W.~Zulehner.
\newblock Analysis of iterative methods for saddle point problems: a unified
  approach.
\newblock {\em Math. Comp.}, 71(238):479--505, 2002.

\end{thebibliography}
\bibliographystyle{abbrv} 
  
\ifnum\pub=0 
\fi
\ifnum\pub=1
\pagebreak
\section*{Appendix}
Recall, that in order to derive the finite element error estimates in Theorem \ref{thm:l2-reg:fem-estiamtes} for the $L^2$-regularization, resulting in the optimal choice $\varrho = h^4$, we needed to assume the regularization error estimates in Proposition \ref{prop:regularization-interpolation-estimates-L2}, given as 
\begin{equation*}
    |y_\varrho - y_d |_{H^1(Q)} \leq c \, \varrho^{1/4} \,
    \| \Box y_d \|_{L^2(Q)},
  \end{equation*}
  and
  \begin{equation*}
    | p_\varrho |_{H^1(Q)} \leq c \, \varrho^{3/4} \, \| \Box y_d \|_{L^2(Q)},
  \end{equation*}
for $y_d\in H_{0;0,}^{1,1}(Q)$ such that $\Box y_d\in L^2(Q)$. 
Although, in Remark \ref{rem:regularization-interpolation-estimates-L2}, we already gave an example, showing that the estimates \eqref{Estimate p H1 H2} and \eqref{Estimate y H1 H2} do not hold for any target function $y_d$ that is smooth enough, we want to numerically demonstrate that the interpolation estimates are indeed true for some targets. 
To this end we will consider the one-dimensional case in space, i.e., $Q = \Omega\times (0,T)=(0,1)^2\subset \mathbb{R}^2$ and the smooth targets $y_{d,i}\in H^2(Q)\cap H_{0;0,}^{1,1}(Q)$, see Figure \ref{fig:targets}, given as
\begin{equation*}
    y_{d,1}(x,t) = \begin{cases}
        \frac{1}{2}(6t-3x-2)^3(3x-6t)^3\sin(\pi x),& x\leq 2t \; \text{and}\; 6t-3x\leq 2,\\
        0,& \text{else}
    \end{cases}
\end{equation*}
\begin{equation*}
    y_{d,2}(x,t) = \sin(\pi x)\sin(\pi t),
\end{equation*}
\begin{equation*}
    y_{d,3}(x,t) = t^2\sin(\pi x).
\end{equation*}

\begin{figure}[ht]
	\centering
	\begin{subfigure}[b]{0.3\textwidth}
		\includegraphics[width=\textwidth]{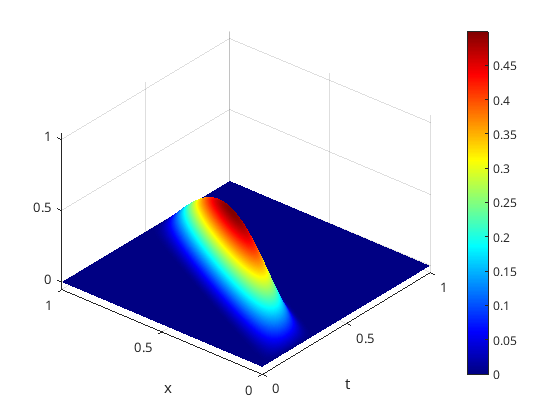}
		\caption{$y_{d,1}$}
	\end{subfigure}
	\hfill
    \begin{subfigure}[b]{0.3\textwidth}
		\includegraphics[width=\textwidth]{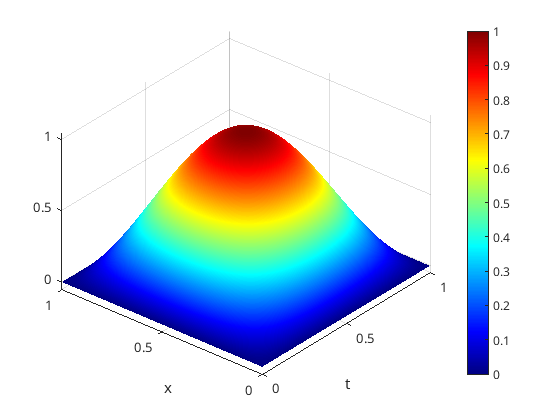}
		\caption{$y_{d,2}$}
	\end{subfigure}
	\hfill
    \begin{subfigure}[b]{0.3\textwidth}
		\includegraphics[width=\textwidth]{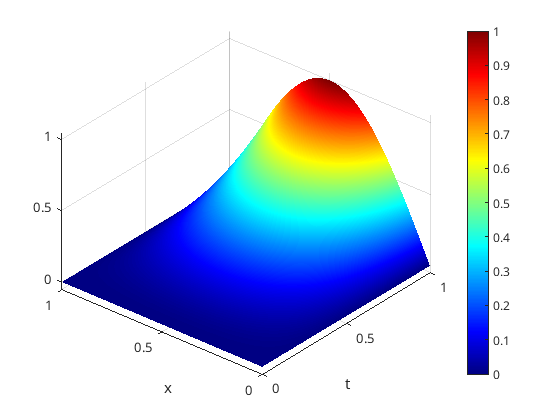}
		\caption{$y_{d,3}$}
	\end{subfigure}
	\caption{Targets $y_{d,i}\in H_{0;0,}^{1,1}(Q)\cap H^2(Q)$, $i=1,2,3$. }
	\label{fig:targets}
\end{figure}

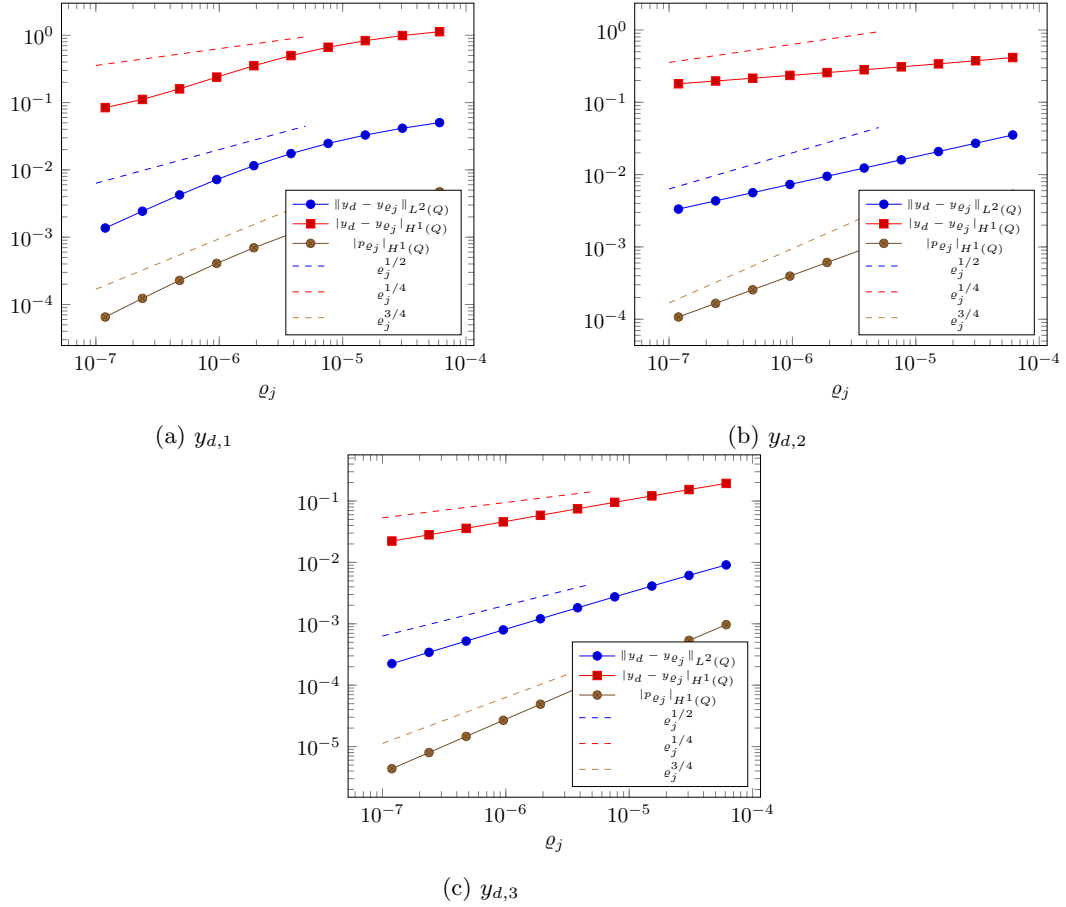
\begin{figure}
  \centering
  \begin{subfigure}[b]{0.4\textwidth}
  \begin{tikzpicture}[scale=0.8]
    \begin{axis}[
      xmode = log,
      ymode = log,
      xlabel=$\varrho_j$,
      legend pos=south east, legend style={font=\tiny}]
      \addplot table [col sep=
      &, y=erryL2, x=rho, blue]{tab-p-y-H1interpolation-ydwave-rho2-i.dat};
      \addlegendentry{$\|y_d-y_{\varrho_j}\|_{L^2(Q)}$}
      \addplot table [col sep=
      &, y=erryH1, x=rho, red]{tab-p-y-H1interpolation-ydwave-rho2-i.dat};
      \addlegendentry{$|y_d-y_{\varrho_j}|_{H^1(Q)}$}
      \addplot table [col sep=
      &, y=pH1, x=rho, brown]{tab-p-y-H1interpolation-ydwave-rho2-i.dat};
      \addlegendentry{$|p_{\varrho_j}|_{H^1(Q)}$}
      \addplot[
      domain = 10^(-7):0.5*10^(-5),
      samples = 2,
      dashed, thin,blue] {20*x^(0.5)};
      \addlegendentry{$\varrho_j^{1/2}$}
      \addplot[
      domain = 10^(-7):0.5*10^(-5),
      samples = 2,
      dashed, thin,red] {20*x^(0.25)};
      \addlegendentry{$\varrho_j^{1/4}$}
      \addplot[
      domain = 10^(-7):0.5*10^(-5),
      samples = 2,
      dashed, thin,brown] {30*x^(0.75)};
      \addlegendentry{$\varrho_j^{3/4}$}
    \end{axis}
    \end{tikzpicture}
    \caption{$y_{d,1}$ }
    \end{subfigure}
    \hfill
  \begin{subfigure}[b]{0.4\textwidth}
  \begin{tikzpicture}[scale=0.8]
    \begin{axis}[
      xmode = log,
      ymode = log,
      xlabel=$\varrho_j$,
      legend pos=south east, legend style={font=\tiny}]
      \addplot table [col sep=
      &, y=erryL2, x=rho, blue]{tab-p-y-H1interpolation-ydsin-rho2-i.dat};
      \addlegendentry{$\|y_d-y_{\varrho_j}\|_{L^2(Q)}$}
      \addplot table [col sep=
      &, y=erryH1, x=rho, red]{tab-p-y-H1interpolation-ydsin-rho2-i.dat};
      \addlegendentry{$|y_d-y_{\varrho_j}|_{H^1(Q)}$}
      \addplot table [col sep=
      &, y=pH1, x=rho, brown]{tab-p-y-H1interpolation-ydsin-rho2-i.dat};
      \addlegendentry{$|p_{\varrho_j}|_{H^1(Q)}$}
      \addplot[
      domain = 10^(-7):0.5*10^(-5),
      samples = 2,
      dashed, thin,blue] {20*x^(0.5)};
      \addlegendentry{$\varrho_j^{1/2}$}
      \addplot[
      domain = 10^(-7):0.5*10^(-5),
      samples = 2,
      dashed, thin,red] {20*x^(0.25)};
      \addlegendentry{$\varrho_j^{1/4}$}
      \addplot[
      domain = 10^(-7):0.5*10^(-5),
      samples = 2,
      dashed, thin,brown] {30*x^(0.75)};
      \addlegendentry{$\varrho_j^{3/4}$}
    \end{axis}
    \end{tikzpicture}
  \caption{$y_{d,2}$}
    \end{subfigure}
    \\
    \begin{subfigure}[b]{0.4\textwidth}
  \begin{tikzpicture}[scale=0.8]
    \begin{axis}[
      xmode = log,
      ymode = log,
      xlabel=$\varrho_j$,
      legend pos=south east, legend style={font=\tiny}]
      \addplot table [col sep=
      &, y=erryL2, x=rho, blue]{tab-p-y-H1interpolation-ydsint-rho2-i.dat};
      \addlegendentry{$\|y_d-y_{\varrho_j}\|_{L^2(Q)}$}
      \addplot table [col sep=
      &, y=erryH1, x=rho, red]{tab-p-y-H1interpolation-ydsint-rho2-i.dat};
      \addlegendentry{$|y_d-y_{\varrho_j}|_{H^1(Q)}$}
      \addplot table [col sep=
      &, y=pH1, x=rho, brown]{tab-p-y-H1interpolation-ydsint-rho2-i.dat};
      \addlegendentry{$|p_{\varrho_j}|_{H^1(Q)}$}
      \addplot[
      domain = 10^(-7):0.5*10^(-5),
      samples = 2,
      dashed, thin,blue] {2*x^(0.5)};
      \addlegendentry{$\varrho_j^{1/2}$}
      \addplot[
      domain = 10^(-7):0.5*10^(-5),
      samples = 2,
      dashed, thin,red] {3*x^(0.25)};
      \addlegendentry{$\varrho_j^{1/4}$}
      \addplot[
      domain = 10^(-7):0.5*10^(-5),
      samples = 2,
      dashed, thin,brown] {2*x^(0.75)};
      \addlegendentry{$\varrho_j^{3/4}$}
    \end{axis}
    \end{tikzpicture}
    \caption{$y_{d,3}$ }
    \end{subfigure}
    \caption{Convergence plots for the targets $y_{d,i}$, $i=1,2,3$, choosing $\varrho_j =2^{-j}$, $j=14,\ldots,23$ for the $L^2$-regularization where the reference solution $y_{\varrho_j}=y_{\varrho_j h}\in Y_h$ is computed via a finite element method on a uniform mesh with $n_h=131072$ simplicial elements and $m_h = 65280$ DoFs with mesh size $h=2.7621$e$-3$.}
    \label{fig:regularization-interpolation-estimates-rho2}
\end{figure}

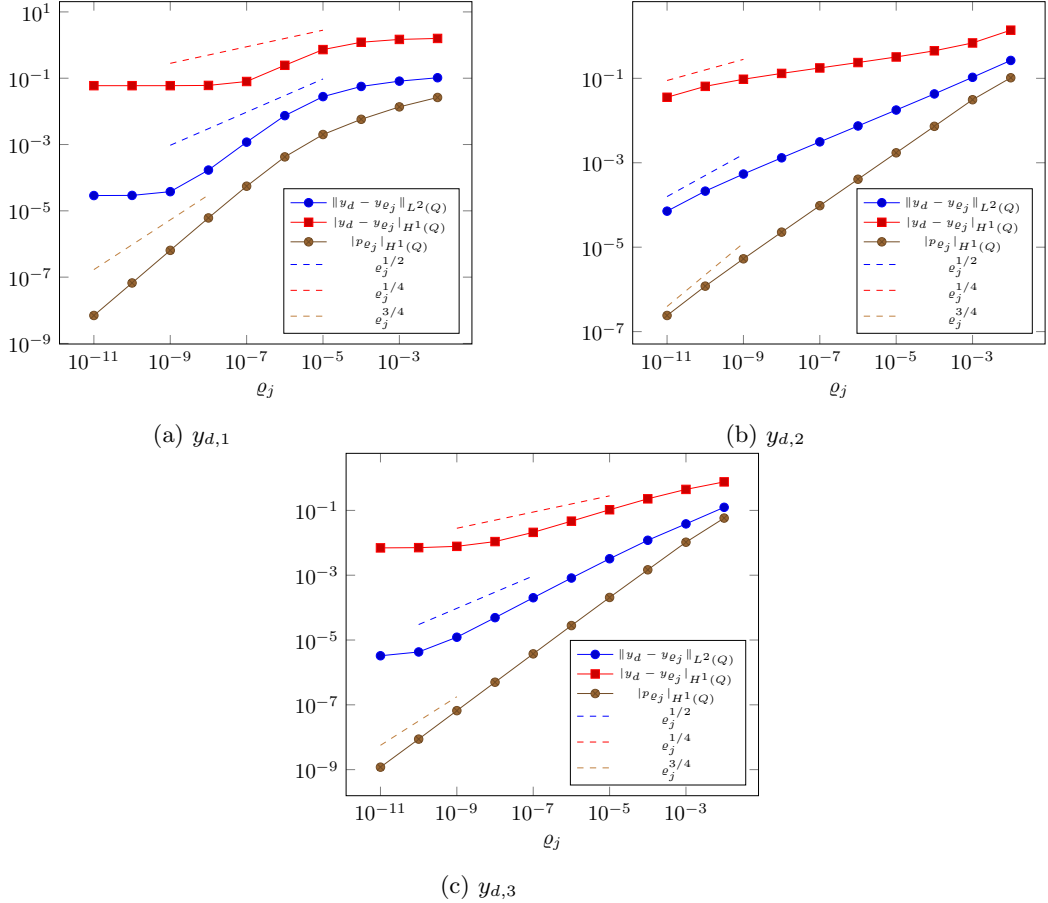
\begin{figure}
  \centering
  \begin{subfigure}[b]{0.4\textwidth}
  \begin{tikzpicture}[scale=0.8]
    \begin{axis}[
      xmode = log,
      ymode = log,
      xlabel=$\varrho_j$,
      legend pos=south east, legend style={font=\tiny}]
      \addplot table [col sep=
      &, y=erryL2, x=rho, blue]{tab-p-y-H1interpolation-ydwave.dat};
      \addlegendentry{$\|y_d-y_{\varrho_j}\|_{L^2(Q)}$}
      \addplot table [col sep=
      &, y=erryH1, x=rho, red]{tab-p-y-H1interpolation-ydwave.dat};
      \addlegendentry{$|y_d-y_{\varrho_j}|_{H^1(Q)}$}
      \addplot table [col sep=
      &, y=pH1, x=rho, brown]{tab-p-y-H1interpolation-ydwave.dat};
      \addlegendentry{$|p_{\varrho_j}|_{H^1(Q)}$}
      \addplot[
      domain = 10^(-9):10^(-5),
      samples = 2,
      dashed, thin,blue] {30*x^(0.5)};
      \addlegendentry{$\varrho_j^{1/2}$}
      \addplot[
      domain = 10^(-9):10^(-5),
      samples = 2,
      dashed, thin,red] {50*x^(0.25)};
      \addlegendentry{$\varrho_j^{1/4}$}
      \addplot[
      domain = 10^(-11):10^(-8),
      samples = 2,
      dashed, thin,brown] {30*x^(0.75)};
      \addlegendentry{$\varrho_j^{3/4}$}
    \end{axis}
    \end{tikzpicture}
    \caption{$y_{d,1}$ }
    \end{subfigure}
    \hfill
  \begin{subfigure}[b]{0.4\textwidth}
  \begin{tikzpicture}[scale=0.8]
    \begin{axis}[
      xmode = log,
      ymode = log,
      xlabel=$\varrho_j$,
      legend pos=south east, legend style={font=\tiny}]
      \addplot table [col sep=
      &, y=erryL2, x=rho, blue]{tab-p-y-H1interpolation-ydsin.dat};
      \addlegendentry{$\|y_d-y_{\varrho_j}\|_{L^2(Q)}$}
      \addplot table [col sep=
      &, y=erryH1, x=rho, red]{tab-p-y-H1interpolation-ydsin.dat};
      \addlegendentry{$|y_d-y_{\varrho_j}|_{H^1(Q)}$}
      \addplot table [col sep=
      &, y=pH1, x=rho, brown]{tab-p-y-H1interpolation-ydsin.dat};
      \addlegendentry{$|p_{\varrho_j}|_{H^1(Q)}$}
      \addplot[
      domain = 10^(-11):10^(-9),
      samples = 2,
      dashed, thin,blue] {50*x^(0.5)};
      \addlegendentry{$\varrho_j^{1/2}$}
      \addplot[
      domain = 10^(-11):10^(-9),
      samples = 2,
      dashed, thin,red] {50*x^(0.25)};
      \addlegendentry{$\varrho_j^{1/4}$}
      \addplot[
      domain = 10^(-11):10^(-9),
      samples = 2,
      dashed, thin,brown] {70*x^(0.75)};
      \addlegendentry{$\varrho_j^{3/4}$}
    \end{axis}
    \end{tikzpicture}
  \caption{$y_{d,2}$}
    \end{subfigure}
    \\
    \begin{subfigure}[b]{0.4\textwidth}
  \begin{tikzpicture}[scale=0.8]
    \begin{axis}[
      xmode = log,
      ymode = log,
      xlabel=$\varrho_j$,
      legend pos=south east, legend style={font=\tiny}]
      \addplot table [col sep=
      &, y=erryL2, x=rho, blue]{tab-p-y-H1interpolation-ydsint.dat};
      \addlegendentry{$\|y_d-y_{\varrho_j}\|_{L^2(Q)}$}
      \addplot table [col sep=
      &, y=erryH1, x=rho, red]{tab-p-y-H1interpolation-ydsint.dat};
      \addlegendentry{$|y_d-y_{\varrho_j}|_{H^1(Q)}$}
      \addplot table [col sep=
      &, y=pH1, x=rho, brown]{tab-p-y-H1interpolation-ydsint.dat};
      \addlegendentry{$|p_{\varrho_j}|_{H^1(Q)}$}
      \addplot[
      domain = 10^(-10):10^(-7),
      samples = 2,
      dashed, thin,blue] {3*x^(0.5)};
      \addlegendentry{$\varrho_j^{1/2}$}
      \addplot[
      domain = 10^(-9):10^(-5),
      samples = 2,
      dashed, thin,red] {5*x^(0.25)};
      \addlegendentry{$\varrho_j^{1/4}$}
      \addplot[
      domain = 10^(-11):10^(-9),
      samples = 2,
      dashed, thin,brown] {x^(0.75)};
      \addlegendentry{$\varrho_j^{3/4}$}
    \end{axis}
    \end{tikzpicture}
    \caption{$y_{d,3}$ }
    \end{subfigure}
    \caption{Convergence plots for the targets $y_{d,i}$, $i=1,2,3$, choosing $\varrho_j =10^{-j}$, $j=2,\ldots,11$ for the $L^2$-regularization where the reference solution $y_{\varrho_j}=y_{\varrho_j h}\in Y_h$ is computed via a finite element method on a uniform mesh with $n_h=131072$ simplicial elements and $m_h = 65280$ DoFs with mesh size $h=2.7621$e$-3$.}
    \label{fig:regularization-interpolation-estimates-rho10}
\end{figure}

In order to check the interpolation error estimates, we consider a sequence of fixed $\varrho_j >0$, $j\in\mathbb{N}$, and compute for each target $y_{d,i}$, $i=1,2,3$, a related state $y_{\varrho_j} = y_{\varrho_j h} \in Y_h$ on a fine mesh with $n_h=131072$ elements and $m_h = 65280$ DoFs. In Figure \ref{fig:regularization-interpolation-estimates-rho2} the results for $\varrho_j = 2^{-j}$, $j=14,\ldots,23$ are depicted. We clearly see the predicted behavior, i.e., $|p_\varrho|_{H^1(Q)}\simeq \varrho^{3/4}$ and $|y_d-y_\varrho|_{H^1(Q)}\simeq \varrho^{1/4}$. Morover, we also plot the $L^2$-error of the state to the target, where we observe the behavior $\|y_d-y_\varrho\|_{L^2(Q)}\simeq \sqrt{\varrho}$, which fits perfectly to the theoretical findings. In Figure \ref{fig:regularization-interpolation-estimates-rho10} we show the results for $\varrho_j = 10^{-j}$, $j=2,\ldots,11$. Note, that after a while the $L^2$-convergence breaks down, as a result of the best approximation property of $Y_h$ in $L^2(Q)$ when computing $y_{\varrho_j} = y_{\varrho_j h}\in Y_h$. Having a closer look, this happens when $\varrho_j \simeq h^4$. This supports the optimal choice $\varrho = h^4$, since choosing a smaller parameter $\varrho>0$ will not lead to a better approximation of the desired target for a given mesh size $h>0$. Note, that the $H^1$-error seems to stagnate even earlier.  

\fi
\end{document}